\documentclass[a4paper,10.5pt]{amsart}

\usepackage{amsmath}
\usepackage{amssymb}
\usepackage{amsthm}
\usepackage{amscd}
\usepackage{hyperref}
\usepackage[dvips]{graphicx}
\usepackage{enumerate}
\usepackage{breqn}
\usepackage{cleveref}
\usepackage{enumerate}

\newtheorem{theorem}{Theorem}[section]
\newtheorem{lemma}{Lemma}[section]

\newtheorem{prop}{Proposition}[section]
\newtheorem{remark}{Remark}

\numberwithin{equation}{section}
\newcommand{\real}{\mathbb{R}}
\newcommand{\gz}{g_{\zeta}}
\newcommand{\gzone}{g_{\zeta_1}}
\newcommand{\gztwo}{g_{\zeta_2}}
\newcommand{\tr}{\textnormal{ tr}}

\def\eqn {\begin{equation}}
\def\eeqn {\end{equation}}

\def\real{{\mathbb R}}

\def\ep{\epsilon}

\def\pa{\partial}

\def\G{\mathcal G}
\def\F{\mathcal F}

\def\L{\mathcal L}
\def\M{\mathcal M}

\def\ka{\kappa}

\begin{document}
\title{STEADY STATES OF ROTATING STARS AND GALAXIES}
\author{Walter A. Strauss} 
\address{Department of Mathematics, Brown University, Providence, RI 02912}
\author{Yilun Wu}
\address{Department of Mathematics, Brown University, Providence, RI 02912}
\date{}
\maketitle

\begin{abstract}
A rotating continuum of particles attracted to each other byÊ
gravity may be modeled by the Euler-Poisson system. The existenceÊ
of solutions is a very classical problem. ÊHere it is proven that aÊ
curve of solutions exists, parametrized by the rotation speed, withÊ
a fixed mass independent of the speed. ÊThe rotation is allowed to varyÊ
with the distance to the axis. A special case is when  
the equation of state is $p=\rho^\gamma,\ 6/5<\gamma<2,\ \gamma\ne4/3$, 
in contrast to previous variational methods which have required $4/3 < 
\gamma$.  

The continuum of particles may alternatively be modeled microscopicallyÊ
by the Vlasov-Poisson system. ÊThe kinetic density is a prescribedÊ
function. ÊWe prove an analogous theorem asserting the existence of aÊ
curve of solutions with constant mass. ÊIn this model the 
whole range $(6/5,2)$ is allowed, including $\gamma=4/3$. 
\end{abstract}

\tableofcontents

\section {Introduction}

We consider a continuum of particles attracted to each other by gravity 
but subject to no other forces.  
Initially they are static and spherical but then they begin to rotate after some perturbation and 
thereby flatten at the poles and expand at the equator.  This is a simple 
model of a rotating star or planet.  It can also model a rotating galaxy
with its billions of stars.  In this paper we consider slow rotations and look
for steady states of the resulting configuration. We find a connected set
of such states with constant mass.

This is a very classical problem that goes back to 
MacLaurin, Jacobi, Poincar\'e, Liapunov et al.,  who
assumed the density of the rotating fluid to be homogeneous or almost homogeneous, which is of course physically unrealistic if we want to consider a rotating gaseous star or a rotating galaxy. See Jardetzky \cite{jardetzky2013theories} for a nice account of the classical history on this problem.
More realistic was the later work of Lichtenstein \cite{lichtenstein1933untersuchungen} and Heilig \cite{heilig1994lichtenstein}, who
approached the problem by means of an implicit function theorem (IFT) in
function space.  They made realistic assumptions on the density but the mass of their solutions changes as the body changes its speed of rotation.

A different approach was begun by Auchmuty and Beals \cite{auchmuty1971variational} using a variational
method (VAR) with a mass constraint.  
The main difficulty in this approach is to prove that the minimizing solution has compact support.  
Their approach was generalized by many authors, including Auchmuty \cite{auchmuty1991global}, Caffarelli and Friedman \cite{caffarelli1980shape}, Friedman and Turkington \cite{friedman1981existence}, Li \cite{li1991uniformly}, Chanillo and Li \cite{chanillo1994diameters}, Luo and Smoller \cite{luo2009existence}, Wu \cite{wu2015rotating}, and
Wu \cite{wu2016existence}. Compared with IFT, the VAR method has the advantage that the mass is constant.   
On the other hand,
compared with VAR, the IFT method has the advantage that the construction
provides a {\it continuous} curve of solutions depending on the angular
velocity $\omega$ and with obviously compact support of oblate shape. 

In this paper we improve the IFT approach by constructing solutions that
keep the mass constant, so that there is no loss or gain of particles when the body changes its rotation speed. 
Also, we allow the angular velocity to be non-uniform, thus including the physically interesting cases of differential rotation into our model.

Our first model, which we call EP (Euler-Poisson), is generalized from Lichtenstein and Heilig.  
In EP we assume that the particles move inside the body according
to the steady compressible Euler equations, subject to internal forces 
of Newtonian gravity given by the Poisson equation, with a variable speed $\omega$ of rotation around an axis,
and an equation of state
for the pressure $p=p(\rho)$ where $\rho$ is the density.

Our second model, which we call VP (Vlasov-Poisson), is generalized from Rein \cite{rein2000stationary}.  
In VP the particles are given by a microscopic density $f(x,v)$ satisfying the Vlasov equation, 
where the macroscopic density is given by $\rho(x)=\int_{\real^3} f(x,v)~dv$ as in kinetic
theory.  The rotation is provided by a rather arbitrary function of the
microscopic angular momentum $x_1v_2-x_2v_1$.
Although the two models are clearly different, they have some similarities. 

In this paper we treat both models by constructing a continuous curve of solutions using the IFT approach.  
The parameter along this curve is the intensity of the rotation speed.  
As distinguished from all the previous literature using the IFT approach, the mass is constant 
along the curve, and the angular velocity can be non-uniform.  

\subsection {Informal statement of results}   
For the EP model we begin with the steady compressible Euler-Poisson equations 
for the density $\rho\ge0$, subject to the internal forces of gravity due to the particles themselves.  
The speed $\omega(r)$ of rotation around the $x_3$-axis is allowed to depend on $r=r(x)=\sqrt{x_1^2+x_2^2}$.  
The inertial forces are entirely due to the rotation. In the region $\{\rho_{\kappa}>0\}$ 
the governing equation turns out to be 
\eqn \label{NG} 
\rho_\kappa *\frac 1{|x|} + \kappa\int_0^{r} \omega^2(s)s\,ds  
     -  h(\rho_\kappa)     = constant,   \eeqn
 Here $\omega(r)$ is a given function, $\kappa$ is a constant measuring the intensity of rotation, and $h$ is defined by $h'(\rho)=\frac{p'(\rho)}{\rho}$ with $h(0)=0$.  
The pressure is $p$ and the specific enthalpy is $h$.  The constant of gravity is assumed to be 1.  
Our theorem for the EP model, informally stated in a special case is as follows.  

Let the equation of state be the power law
$p(\rho)=C\rho^\gamma$ where $\frac65 < \gamma < 2$ and $\gamma\ne\frac43$.  
For any mass $M$ of a non-rotating star, there exists $\bar{\kappa}>0$ 
and a continuous curve $\kappa \mapsto \rho_\kappa$ from $(-\bar{\kappa} , \bar{\kappa})$ 
into $C_c^1(\real^3)$ such that each $\rho_\kappa$ is an 
axisymmetric solution of \eqref{NG} with total mass $M=\int_{\real^3} \rho_\kappa\ dx$.     

For the VP model we begin with the steady Vlasov-Poisson system for the 
microscopic density $f(x,v)$ and the gravitational potential $U(x)$. 
Extending Rein \cite{rein2000stationary}, we look for solutions $f$ that have the form 
\eqn \label{ReinAnsatz} 
f_\kappa(x,v)  =  C_\kappa\ \phi(E,L),\quad 
E=\tfrac12|v|^2 + U_\kappa(x), \quad L=\kappa(x_1v_2-v_1x_2),  \eeqn 
where $\phi$ is a prescribed function.
The constant $C_\kappa$ is chosen so that the total mass $\iint_{\real^6} f_\kappa(x,v)~dvdx$ 
is a given constant $M$ independent of $\kappa$.  
Because $x_1v_2-x_2v_1$ is the $x_3$ component of the angular momentum, 
$\kappa$ provides the intensity of rotation.  The governing equation then takes the form
\eqn\label{Rein governing}
-\Delta_x U_\kappa = 4\pi \int_{\real^3}f_\kappa(x,v)~dv.   \eeqn
Since  \eqref{ReinAnsatz} determines $f_\ka$ in terms of $U_\ka$, \eqref{Rein governing} 
is a single equation for $U_\kappa$.
Our theorem for the VP model, informally stated in a special case, is as follows.  

Assume that $\phi(E,L) = (-E)_+^{-\mu}\, \psi(L)$ with $-\frac72<\mu<\frac12$, and $\psi$  a suitably regular non-negative function.
For any mass $M$ of a non-rotating star, there exists $\bar{\kappa}>0$ 
and a continuous curve $\kappa \mapsto U_\kappa$ from $(-\bar{\kappa} , \bar{\kappa})$ 
into $C^3(\real^3)$ such that each $U_\kappa$ provides an axisymmetric 
solution of \eqref{Rein governing}, with total mass $M$.

When this paper was almost complete, we learned of very recent work \cite{jang2016slowly} by Jang and Makino,   
who also studied the EP model  using an IFT approach 
in the case of the power law $p=C\rho^\gamma$  and constant rotation speed.  
The perturbation they take is very different from the one of this paper.  
Rather than deforming a given non-rotating star solution as we do, 
they directly perturb the specific enthalpy in a function space, 
which appears to be a more general type of perturbation. 
However, as in Lichtenstein and Heilig's work, their perturbation also does not keep the total mass constant.  
Their analysis is restricted to the range $\frac65 < \gamma < \frac32$. 

\subsection {Technique and outline} 
Following Lichtenstein and his successors, we construct rotating solutions by deforming 
the corresponding spherically symmetric stationary solution $\rho_0$. 
The deformation $\gz$, characterized 
by a function $\zeta(x)$ that is axisymmetric around the $x_3$-axis, 
even in $x_3$, is defined by 
\eqn
\gz(x) = \left(1+\frac{\zeta(x)}{|x|^2}\right)x.
\eeqn
The factor $1/{|x|^2}$ is a convenience that differs from the previous authors. 
For EP we then define
$\rho_\zeta(y) = \M(\zeta)\,\rho(g_\zeta^{-1}(y))$, 
where $\M(\zeta)$ is chosen to assure the total mass to be independent of $\zeta$. 
For VP an analogous but slightly more complicated definition is used, 
again designed to keep the total mass independent of $\zeta$. 

Both models are formulated implicitly in the form 
$\F(\zeta,\kappa)=0 \in X$ for $\zeta$ in a function space $X$ of axisymmetric functions.  
Here $\F(0,0)=0$ corresponds to the spherically symmetric non-rotating solution and $\kappa$ measures the intensity of rotation, which is assumed to be small.  
Therefore what must be proven is that $\F$ is differentiable and that 
$\L=\frac{\pa\F}{\pa\zeta}(0,0)$ is an isomorphism.  

Sections \ref{sec: 2} and \ref{sec: 3} are devoted to a precise formulation of the EP model  
and the statement of the main EP theorem (Theorem \ref{main Euler theorem}).  
In Section \ref{sec: 5} we prove the Fr\'echet differentiability of $\F$. This turns out to be 
a surprisingly technical task.  
Our space $X$ has the norm $\|\zeta\|_X  =  \sup_{|x|\le1} |\nabla\zeta(x)|/|x|$,   
which is simpler than the norm used in the previous literature.  
The simplification is aided by our extra factor $1/|x|^2$ in the definition of $g_\zeta$.  
We compute the formal derivative, then show that it maps $X$ to $X$, and finally 
prove that it is Fr\'echet differentiable.  

In Section \ref{sec: 4} we prove for the EP model that the linearized operator $\L=\frac{\pa\F}{\pa\zeta}(0,0)$ 
is essentially of the form $I+K$, where $K$ is compact.  Thus the main task is to prove that 
the nullspace of $\L$ is trivial.  This task is considerably more difficult than the previous studies 
(Lichtenstein and Heilig) because the mass constraint adds a whole new nonlocal term to $L$  
and because the rotation speed $\omega$ depends on $r$, the distance to the rotation axis. 
There are several novel aspects to the proof.  
Using a delicate scaling argument, we prove in Theorem \ref{injectivity} that the nullspace is indeed trivial, 
assuming that the mass of a radial solution strictly changes as the density at the origin changes.  
In the example  $p(\rho)=C\rho^\gamma$, this assumption corresponds to 
the condition that $\gamma\ne\frac43$.  The power $\gamma=\frac43$ is the ``white dwarf" model, 
for which the mass is invariant under scaling and for which the nullspace of $\L$ is not trivial. 
At the end of Section \ref{sec: 4}, the constant angular velocity case is examined in more detail. 
Properties of $\L$ imply that the supports of the perturbed solutions are wider at the
equator than at the poles, confirming the usual physical intuition of the shape of slowly rotating stars.

The rest of the paper is devoted to the VP model.  In Section \ref{sec: 6} we present the precise 
formulation of the model, define an operator $\F$ different from the EP model, and state the main theorem  
(Theorem \ref{main Vlasov thm}).  
In Section \ref{sec: 8} we prove the Fr\'echet differentiability of $\F$, which is analogous to the previous proof for EP.  
Section \ref{sec: 7} is devoted to the linearized operator $\L=\frac{\pa\F}{\pa\zeta}(0,0)$ for the VP model.  
The triviality of its nullspace is quite delicate and is significantly different both from that for the EP model 
and from that in Rein \cite{rein2000stationary}. 
For the VP model there are no exceptional cases sensitive to mass invariance.  
The special choice of $\phi(E,L)$ mentioned above with $-\frac72<\mu<\frac12$   
corresponds to $\frac65<\gamma<2$, where $\gamma = 1 + (\frac32-\mu)^{-1}$.

\section{The Euler model}\label{sec: 2}
In this model, the gas is described by the compressible Euler-Poisson equations. The equations in full generality are given as
\begin{equation}\label{eq: full Euler-Poisson}
\begin{cases}
\rho _t + \nabla\cdot (\rho v)=0, \\
(\rho v)_t + \nabla\cdot(\rho v\otimes v) + \nabla p = \rho \nabla U, \\
U(x,t) = \int_{\real^3}\frac{\rho(x',t)}{|x-x'|}~dx'.
\end{cases}
\end{equation}
Here, the first two equations hold where $\rho>0$, and the last equation defines $U$ on the entire $\real^3$.
To close the system, one prescribes an isentropic equation of state
\begin{equation}
p = p(\rho).
\end{equation}
To model a rotating star, we look for steady axisymmetric rotating solutions to \eqref{eq: full Euler-Poisson}, 
i.e. we assume $\rho = \rho(x)=\rho(Ax)$ for any rotation $A$ about the 
$x_3$-axis, $v=\ka\,\omega(r)(-x_2,x_1,0) $, where $r=\sqrt{x_1^1+x_2^2}$, 
where the angular velocity distribution $\omega(r)$ is prescribed. 
With such specifications, the first equation in \eqref{eq: full Euler-Poisson} concerning mass conservation 
is identically satisfied. The second equation in \eqref{eq: full Euler-Poisson} 
concerning momentum conservation can be simplified to
\begin{equation}\label{eq: pre Euler-Poisson}
-\rho\, \ka\,\omega^2(r) e_r + \nabla p = \rho\,\nabla U, \qquad e_r=(x_1,x_2,0)
\end{equation}
The first term in \eqref{eq: pre Euler-Poisson} can be written as
$-\rho\nabla \left(\int_0^r \omega^2(s)s~ds\right).  $
If we introduce the {\it specific enthalpy} $h$ as
\begin{equation}\label{def: h from p}
h(\rho) = \int_0^{\rho}\frac{p'(\alpha)}{\alpha}~d\alpha, 
\end{equation}
then \eqref{eq: pre Euler-Poisson} becomes 
\begin{equation}\label{eq: Euler-Poisson vector}
\nabla \left(U +\ka\int_0^r \omega^2(s)s~ds - h(\rho) \right)=0,
\end{equation}
where again
\begin{equation}\label{potential} 
U(x) = \int_{\real^3}\frac{\rho(x')}{|x-x'|}~dx'.
\end{equation}
{\it Thus the Euler model is reduced to \eqref{eq: Euler-Poisson vector}, \eqref{potential}}.
We wish to solve for $\rho$ with prescribed $\omega(r)$ and $p(\rho)$. 

\subsection{Assumptions}
We now state our assumptions on $\omega(r)$ and $p(\rho)$.
For $\omega(r)$ we simply assume 
\eqn  \label{assump omega}
\omega^2(r) \in C_{loc}^{1,\beta} [0,\infty) \text{ for some }\beta=\beta(\omega)\in(0,1).
\eeqn 

For $p(s)$ we make the following three assumptions.  
\eqn\label{assum: 2}
p(s)\in C^3(0,\infty),~ p'>0.
\eeqn
			There exists $\gamma\in(1,2)$ such that,
\begin{equation}\label{assum: 3}
\lim_{s\to 0^+}s^{3-\gamma}p'''(s) = -c_0<0.
\end{equation}

			There exists  $\gamma^*\in(\frac{6}{5},2)$ such that
\begin{equation}\label{assum: 4}
\lim_{s\to \infty}s^{1-\gamma^*}p'(s) = c_1>0.
\end{equation}

\noindent {\bf Example.}  All of the above assumptions are satisfied if we let $\omega(r)$ be a constant 
and take the equation of state to be a power law $p(s) = s^{\gamma}$ 
for some $\gamma\in (\frac{6}{5},2)$.
Indeed, in this case $ \gamma^*=\gamma = c_1,\  
c_0=\gamma(\gamma-1)(2-\gamma),\ h(s)=\frac\gamma{\gamma-1}s^{\gamma-1}$.  

\subsection{Main theorem}  
The existence of {\it radial} (spherically symmetric) solutions is well-known, as 
will be discussed  in Section \ref{sec: 3}, and is summarized in the following proposition.   
			\begin{prop}			
\label{thm: existence of Lane-Emden stars}
Let $p$ satisfy the assumptions above.   
For every $R>0$, there exists a solution $\rho_0(x)$ 
of \eqref{eq: Euler-Poisson vector}, \eqref{potential}  with the following properties.  
\begin{itemize}
\item $\rho_0$ is radial (with $\omega\equiv 0$),     

\item $\rho_0 >0$ in $B_R=\{|x|<R\}$, $\rho_0 =0$ in $\real^3\backslash B_R$, 

\item $\rho_0\in C^2(B_R)\cap C^{1,\alpha}(\real^3)$,  
where $\alpha = \min\left(\frac{2-\gamma}{\gamma-1},1\right)$. 
\end{itemize}
\end{prop}

To state a key condition in the main theorem, we need to express the total mass of radial solutions near a given $\rho_0$. In Section \ref{sec: triviality Euler} we will show that a unique radial solution $\rho$ to \eqref{eq: Euler-Poisson vector}, \eqref{potential} with $\omega(r)=0$ exists provided the density $\rho(0)$ at the center is sufficiently close to $\rho_0(0)$, where $\rho_0$ is a solution given in Proposition \ref{thm: existence of Lane-Emden stars}. Let $M(\rho(0))=\int_{\real^3} \rho ~dx$ be the total mass of such a solution with center density $\rho(0)$. We also denote the mass of $\rho_0$ simply by $M$.  
Our main theorem regarding {\it non-radial} solutions of the EP model is as follows.  

\begin{theorem} \label{main Euler theorem}
Let $\omega$ and $p$ satisfy the assumptions above.  
Assume 
\eqn\label{cond: mass cond}
M'(\rho_0(0))\ne 0.
\eeqn 
Then there exists $\bar\kappa>0$ such that, for every $|\kappa|<\bar\kappa$, there exists 
a solution $\rho_\kappa$ of  \eqref{eq: Euler-Poisson vector}, \eqref{potential} that is 
\begin{itemize}
\item axisymmetric and even in $x_3$;

\item $\rho_\ka \ge0$ and is compactly supported (with support near $B_R$);

\item$\rho_\ka \in C^{1,\alpha}(\real ^3)$ where $\alpha$ is the same as in Proposition \ref{thm: existence of Lane-Emden stars};

\item the mapping $\ka\mapsto\rho_\ka$ is continuous from $(-\bar\ka,\bar\ka)$ to $C^1_c(\real ^3)$;

\item $\int_{\real^3} \rho_\ka~ dx = M$, where the constant $M$ is the total mass of $\rho_0$.
\end{itemize}
\end{theorem}

\begin{remark}
The prime ( $'$) in \eqref{cond: mass cond} means differentiation. This condition can be interpreted as saying that nearby solutions have genuinely different masses. In fact, we will show that \eqref{cond: mass cond} is a necessary and sufficient condition for the kernel of a key linearized operator to be trivial.
\end{remark}

\begin{remark}
$\rho_{\kappa}$ has higher regularity in the interior of its support, but in general is only $C^{1,\alpha}$ 
up to the boundary of its support.
\end{remark}

It is of interest, due to Theorem \ref{main Euler theorem}, to find conditions on $p$, in addition to the ones given in \eqref{assum: 2}, \eqref{assum: 3} and \eqref{assum: 4}, for which \eqref{cond: mass cond} is satisfied. In the following theorem, we give two types of such conditions. One of them concerns power laws $p(s)=s^{\gamma}$, and the other is a general type of condition.

\begin{theorem}\label{sec Euler theorem}
\eqref{cond: mass cond} is satisfied if either
\begin{enumerate}[(a)]
\item $p(s) = s^{\gamma}$, $\gamma\in (\frac65, 2)$, $\gamma\ne \frac43$, or

\item $p(s)$ satisfies  \eqref{assum: 2}, \eqref{assum: 3} and \eqref{assum: 4}, as well as 
\eqn \label{p extra cond}
p'(s)< h(s) \le 2 p'(s) \text{ for }s>0.    
\eeqn
\end{enumerate}
\end{theorem}
The proof of Theorem \ref{sec Euler theorem} is is given in Section 4.5.  

\begin{remark}
By the definition of $h$ given in \eqref{def: h from p}, condition \eqref{p extra cond} is positive linear in $p$. It follows that, if $p_1(s)$ and $p_2(s)$ both satisfy \eqref{p extra cond}, then any positive linear combination of them also. As a consequence, \eqref{cond: mass cond} is satisfied if, for example,  $p(s) = s^{\gamma_1}+s^{\gamma_2}$ with $\gamma_1,\gamma_2\in [\frac32, 2)$.
\end{remark}

\begin{remark}
The classical results using variational methods started by \cite{auchmuty1971variational} only include solutions for power laws with $\gamma>\frac43$. Theorem \ref{main Euler theorem} and \ref{sec Euler theorem} construct rotating star solutions for power laws with $\gamma\in (\frac65, \frac43)$ as well.
\end{remark}

\subsection{Construction}\label{sec: construction}
The solutions $\rho_\ka$ will be perturbations of the radial solution $\rho_0$ given in 
\Cref{thm: existence of Lane-Emden stars}, as we shall now describe.  
The radial solution satisfies 
\begin{equation}\label{eq: rotationless Euler-Poisson}
\Delta (h(\rho_0)) + 4\pi \rho_0 = 0, \quad u_0=h(\rho_0)
\end{equation}
Since the analysis is identical for any value of the radius $R$, 
{\it we assume without loss of generality that $R=1$}. 
As stated above, the support of the perturbed solutions will have a single connected component. 
Thus \eqref{eq: Euler-Poisson vector} can equivalently be written as
\begin{equation}\label{eq: Euler-Poisson basic}
U(x)-U(0)+\kappa\int_0^{r(x)}\omega^2(s)s~ds - h(\rho)(x) + h(\rho)(0) = 0.
\end{equation}
Here $\kappa$ is a constant quantifying the smallness of the rotation.
			Following Lichtenstein and Heilig, 
we look for solutions $\rho_\ka$ of the form 
\begin{equation}\label{def: rho zeta}
\rho_\ka(x) =  \M(\zeta)\rho_0\left(g_{\zeta}^{-1}(x)\right), \quad \zeta = \zeta_\ka, 
\end{equation}
for some axisymmetric function $\zeta:\overline{B_1}\to \real$. 
The {\it dilating function} is 
\begin{equation}\label{def: gzeta}
\gz(x)=x\left(1+\frac{\zeta(x)}{|x|^2}\right) 
\end{equation}
and the {\it mass factor} is 
\begin{equation}\label{def: script M}
\M(\zeta) = \frac{M}{\int_{B_1}\rho_0(x)\det D\gz(x)~dx}
=\frac{\int_{B_1}\rho_0(x)~dx}{\int_{B_1}\rho_0(x)\det D\gz(x)~dx}.
\end{equation}

This means that we are perturbing $\rho_0$ by composing with an axisymmetric 
diffeomorphism $\gz^{-1}$, and rescaling the whole function by $\M(\zeta)$ to get a 
new density distribution which has the same total mass as the unperturbed solution $\rho_0$. 
Such a solution is different from that of Lichtenstein and Heilig in that 
\begin{enumerate}[(i)]
\item the diffeomorphism \eqref{def: gzeta} has an $|x|^2$ factor on the denominator in the last term; 

\item the rescaling factor $\M(\zeta)$ is here to keep the total mass unchanged; and 

\item the function $\omega(r)$ is not necessarily constant, thus allowing differential rotation. 
\end{enumerate}
\noindent The difference (i) is technical and will allow us to present a more elegant argument 
of the whole construction, while (ii) and (iii) are our improvements of the basic physical construction. 

Using \eqref{def: rho zeta}, \eqref{eq: Euler-Poisson basic} becomes
\begin{align}\label{eq: Euler-Poisson basic 1}
&~\M(\zeta)\int_{\gz(B_1)}\rho_0\left(\gz^{-1}(y')\right)\left(\frac{1}{|z-y'|}-\frac{1}{|y'|}\right)~dy' 
+ \kappa\int_0^{r(z)}\omega^2(s)s~ds\notag\\
&\quad - h\left(\M(\zeta)\rho_0\left(\gz^{-1}(z)\right) \right)
+ h\left(\M(\zeta)\rho_0\left(\gz^{-1}(0)\right)\right)=0.
\end{align}
\eqref{eq: Euler-Poisson basic 1} holds for $z\in \gz(B_1)$. 
Composing 
the left hand side with $\gz$, that is $z=g_\zeta(x)$, the equation takes the form 
\begin{align}\label{eq: Euler-Poisson basic 2}
&~\M(\zeta)\int_{\gz(B_1)}\rho_0\left(\gz^{-1}(y')\right)\left(\frac{1}{|\gz(x)-y'|}-\frac{1}{|y'|}\right)~dy' + \kappa\int_0^{r(\gz(x))}\omega^2(s)s~ds\notag\\
&\quad - h\left(\M(\zeta)\rho_0(x) \right)+ h\left(\M(\zeta)\rho_0(0)\right)=0 
\end{align}
for $x\in B_1$.  

Now we rewrite \eqref{eq: Euler-Poisson basic 2} in terms of a {\it nonlinear operator} $\F$. 
Let the function $\zeta$ in \eqref{def: gzeta} belong to the space 
\eqn  \label{space X}
 X=C^1(\overline{B_1})\cap\bigg\{\zeta~\bigg|~\zeta(x) 
\text{ is axisymmetric and even in }x_3, \zeta(0)=0, 
\sup_{x\in \dot{B_1}}\frac{|\nabla \zeta (x)|}{|x|}<\infty\bigg\} \eeqn
endowed with the norm
\begin{equation}  \label{norm X}
\|\zeta\|_{X}=\sup_{x\in \dot{B_1}}\frac{|\nabla \zeta(x)|}{|x|}. 
\end{equation}
If we define the operator $\F$ as 
\begin{align}\label{def: Euler F}
\F(\zeta,\kappa) = &~\M(\zeta)\int_{\gz(B_1)}\rho_0\left(\gz^{-1}(y')\right)\left(\frac{1}{|\gz(x)-y'|}-\frac{1}{|y'|}\right)~dy' + \kappa \int_0^{r(\gz(x))}\omega^2(s)s~ds\notag\\
&\quad - h\left(\M(\zeta)\rho_0(x) \right)+ h\left(\M(\zeta)\rho_0(0)\right),  
\end{align} 
then \eqref{eq: Euler-Poisson basic 2} is equivalent to 
\begin{equation}
\F(\zeta,\kappa)=0.
\end{equation}
We will prove that $\F: B_{\epsilon}(X)\times \real \to X$, where  
 $B_{\epsilon}(X) = \{\zeta\in X~\big|~\|\zeta\|_X<\epsilon\}$, and $\epsilon$ is suitably small, as will be specified in the following discussion. 



\section{Euler model: basic properties}\label{sec: 3}
In this section, we describe the standard construction of $\rho_0$ in Proposition \ref{thm: existence of Lane-Emden stars}, and prove a few basic properties of the mapping $\gz$ defined in \eqref{def: gzeta}. As has been explained above, we work with the solutions supported on the unit ball $B_1\subset\real^3$.

\subsection{Properties of the pressure and enthalpy}  
\eqref{assum: 3} implies 
\begin{equation}\label{eq: lim p''}
\lim_{s\to 0^+}s^{2-\gamma}p''(s) = \frac{c_0}{2-\gamma}, \qquad 
\lim_{s\to 0^+}s^{1-\gamma}p'(s) = \frac{c_0}{(\gamma-1)(2-\gamma)}.
\end{equation}
					It is equally straightforward to show that 
\begin{equation}\label{eq: lim h}
\lim_{s\to 0^+}s^{1-\gamma}h(s) = \frac{c_0}{(\gamma-1)^2(2-\gamma)}, \qquad
\lim_{s\to 0^+}s^{2-\gamma}h'(s) = \frac{c_0}{(\gamma-1)(2-\gamma)},
\end{equation}
\begin{equation}\label{eq: lim h''}
\lim_{s\to 0^+}s^{3-\gamma}h''(s) = -\frac{c_0}{\gamma-1},\qquad 
\lim_{s\to 0^+}s^{4-\gamma}h'''(s) = c_0\frac{\gamma+1}{\gamma-1}.
\end{equation}
\eqn\label{eq: lim h at infty}
\lim_{s\to \infty}s^{1-\gamma^*}h(s) = \frac{c_1}{\gamma^*-1}.
\eeqn
Since $1<\gamma<2$, \eqref{eq: lim p''} implies $p'(s)/s$ is integrable at $0$, 
hence $h$ is continuous on $[0,\infty)$. 
Since $p$ is strictly increasing, $h$ is also. 
Equation \eqref{eq: lim h at infty} implies that the image of $h$ is $[0,\infty)$. 
Hence $h$ is invertible with $h^{-1}:[0,\infty)\to[0,\infty)$.  
					Furthermore, we easily see that
\eqn\label{eq: lim h inv infty}
\lim_{s\to \infty}\frac{h^{-1}(s)}{s}=\infty,\qquad 
\lim_{s\to \infty}\frac{h^{-1}(s)}{s^5}=0
\eeqn
by \eqref{eq: lim h at infty} and the fact that $\gamma^*\in(\frac65,2)$.
The reader is reminded that all of these conditions are satisfied for $p(s)=s^\gamma,\ \frac65<\gamma<2$.

\begin{lemma}\label{lem: C1alpha h inv} 
The inverse enthalpy 
$h^{-1}$ is locally $C^{1,\alpha}$ on $[0,\infty)$, where $\alpha = \min\left(\frac{2-\gamma}{\gamma-1},1\right)$. Also
$h^{-1}(0)=(h^{-1})'(0)=0. $ 
\end{lemma}
\begin{proof}
$h^{-1}(0)=0$ is obvious. By \eqref{eq: lim h},  
$  h^{-1}(s) = O(s^{\frac{1}{\gamma-1}}) \quad \text{as }s\to 0^+. $ 
Hence
\begin{equation}
(h^{-1})'(s) = \frac{1}{h'(h^{-1}(s))} = O(s^{\frac{2-\gamma}{\gamma-1}})
\end{equation}
as $s\to 0^+$. Since $\frac{2-\gamma}{\gamma-1}>0$ as $1<\gamma<2$, $(h^{-1})'(0)=0$ and $(h^{-1})'$ is continuous at zero.
Also  $h^{-1}$ is $C^3$ on $(0,\infty)$. Therefore we only need to show that $h^{-1}$ is locally $C^{1,\alpha}$ near zero.
To that end, we compute for $0\le s_1<s_2$ sufficiently small,
\eqn  \label{est: h inv ' holder}
~|(h^{-1})'(s_2)-(h^{-1})'(s_1)| 
\le ~(s_2-s_1) \int_0^1 |(h^{-1})''(s_1+t(s_2-s_1))|~dt.  \eeqn 
By direct calculation we have 
\begin{equation}
(h^{-1})''(s) = -\frac{h''(h^{-1}(s))}{[h'(h^{-1}(s))]^3}.
\end{equation}
As before, we may conclude that
\begin{equation}
(h^{-1})''(s) = O(s^{\frac{-2\gamma+3}{\gamma-1}})=O(s^{\frac{2-\gamma}{\gamma-1}-1})
\end{equation}
as $s\to 0^+$. 
					Therefore \eqref{est: h inv ' holder} is bounded by
\begin{align}\label{est: h inv ' holder 1}
C(s_2-s_1) \int_0^1 (s_1+t(s_2-s_1))^{\frac{2-\gamma}{\gamma-1}-1}~dt.
\end{align}
If $\frac{2-\gamma}{\gamma-1}\ge 1$, \eqref{est: h inv ' holder 1} is bounded by
$C(s_2-s_1)$. If $\frac{2-\gamma}{\gamma-1}<1$, \eqref{est: h inv ' holder 1} is bounded by
\eqn 
~C(s_2-s_1)^{\frac{2-\gamma}{\gamma-1}}\int_0^1\left(\frac{s_1}{s_2-s_1}+t\right)^{\frac{2-\gamma}{\gamma-1}-1}~dt \eeqn 
$$
\le ~C(s_2-s_1)^{\frac{2-\gamma}{\gamma-1}}\int_0^1 t^{\frac{2-\gamma}{\gamma-1}-1}~dt 
\le ~C(s_2-s_1)^{\frac{2-\gamma}{\gamma-1}}  $$
because $\frac{2-\gamma}{\gamma-1}>0$.
\end{proof}

\subsection{Properties of the radial density $\rho_0$} \label{sec: radial Euler}
If we define 
\begin{equation}\label{def: u_0}
u_0=h(\rho_0),
\end{equation}
\eqref{eq: rotationless Euler-Poisson} can be rewritten as
\begin{equation}\label{eq: gLane-Emden}
\Delta u_0 = -4\pi h^{-1}(u_0).
\end{equation}

\begin{lemma}\label{lem: u_0}
There is a positive solution $u_0\in C^2(\overline{B_1})$ with zero boundary condition to \eqref{eq: gLane-Emden}, which is radial, and which satisfies $\frac{\partial u_0}{\partial r}(r)<0$ for all $0<r\le 1$.
\end{lemma}
\begin{proof}
Of course, \eqref{eq: gLane-Emden} is just a second-order ODE for radial functions.  
So the existence of a positive solution is classical.  But to be specific, 
by \cite{ambrosetti1973dual} and \cite{de1982priori}, \eqref{eq: gLane-Emden} has a positive solution with zero boundary condition on a ball if $h^{-1}$ is locally Lipschitz, $h^{-1}(0)=0$, and satisfies \eqref{eq: lim h inv infty}.  
By \eqref{eq: lim h} the enthalpy also satisfies 
\eqn
\lim_{s\to 0^+}\frac{h^{-1}(s)}{s}=0.   \eeqn  
By \cite{gidas1979symmetry} a positive solution must be radial  
and satisfy $\frac{\partial u_0}{\partial r}(r)<0$ for all $0<r\le 1$.
\end{proof}

We have now constructed the $\rho_0$ in Proposition \ref{thm: existence of Lane-Emden stars}. 
The next lemma establishes the claimed regularity of $\rho_0$.

\begin{lemma}\label{lem: rho_0 C1alpha}
$\rho_0\in C^{1,\alpha}(\real ^3)$ if it is extended to be zero outside $B_1$, 
where $\alpha$ is given as in \Cref{lem: C1alpha h inv}.
\end{lemma}
\begin{proof}
Since $\rho_0 = h^{-1}(u_0)$, $u_0\in C(\overline{B_1})$, and $h^{-1}$ is continuous on $[0,\infty)$, 
we have $\rho_0\in C(\overline{B_1})$. Since $h^{-1}(0)=0$, $\rho_0$ 
vanishes on the boundary of $B_1$. 
Hence $\rho_0$ is continuous on $\real^3$. Furthermore, we have
\begin{equation}\label{eq: partial rho partial U}
\partial_i \rho_0(x) = (h^{-1})'(u_0(x))\partial_i u_0(x)
\end{equation}
Using the fact that $\partial_i u_0\in C^1(\overline{B_1})$ and $(h^{-1})'$ is 
locally $C^{\alpha}$ on $[0,\infty)$, we easily get $\partial_i \rho_0 
\in C^{\alpha}(\overline{B_1})$. 
Since $(h^{-1})'(0)=0$, $\partial_i \rho_0$ vanishes on the boundary of $B_1$. 
Hence $\rho_0\in C^{1,\alpha}(\real^3)$.
\end{proof}
\subsection{Properties of the dilating function $\gz$}
The basic intuition is that $\gz$, defined in \eqref{def: gzeta}, 
should be close to the identity map when $\|\zeta\|_X$ is sufficiently small. More precisely,

\begin{lemma}\label{lem: basic estimates g_zeta}
There is a $C>0$ such that if $\zeta\in B_{\epsilon}(X)$ and $\epsilon$ is small enough, 
then $g_{\zeta}:\overline{B_1}\to g_{\zeta}(\overline{B_1})$ 
is a homeomorphism, 
as well as a $C^1$ diffeomorphism on $\overline{B_1}\setminus \{0\}$. 
Denoting the Jacobian matrix by $Dg_\zeta$, the following estimates hold. 
\begin{equation}\label{est: Dg}
|Dg_{\zeta}(x)-I|<C\|\zeta\|_X \quad \text{for all }x\in \overline{B_1}\setminus\{0\}.
\end{equation}

\begin{equation}\label{est: Dg inv}
|(Dg_{\zeta})^{-1}(y)-I|<C\|\zeta\|_X \quad \text{for all }y\in g_{\zeta}(\overline{B_1})\setminus\{0\}.
\end{equation}

\begin{equation}\label{est: det Dg}
|\det Dg_{\zeta}(x)-1|<C\|\zeta\|_X \quad \text{for all }x\in \overline{B_1}\setminus\{0\}.
\end{equation}

\begin{equation}\label{est: g-x}
|g_{\zeta}(x)-x|\le  \|\zeta\|_X |x| \quad \text{for all }x\in \overline{B_1}.
\end{equation}

\begin{equation}\label{est: g inv-x}
|g_{\zeta}^{-1}(y)-y|\le C \|\zeta\|_X |y| \quad \text{for all }y\in g_{\zeta}(\overline{B_1}).
\end{equation}

\begin{equation}\label{est: deviation from x-x'}
|\left(g_{\zeta}(x)-g_{\zeta}(x')\right)-(x-x')|\le C \|\zeta\|_X|x-x'| \quad \text{for all }x, x'\in \overline{B_1}.
\end{equation}

\end{lemma}
\begin{proof}
The continuity of $\gz$ is obvious. To study the inverse, we employ a more careful estimate of $\gz(x)-\gz(x')$. Before we present the estimate, an elementary but useful geometric fact will be established. For any point $x\in \real^3$, let $Ob(x)$ be the {\it symmetry orbit} of $x$, i.e., the intersection of 
\begin{enumerate}
\item the cylinder about the $x_3$-axis through $x$ 
and 
\item the ball centered at $0$ with radius $|x|$.
\end{enumerate}
We claim that if 
 $x,x'\in \real^3$ be such that $|x|\ge|x'|$, then there is an $x''\in Ob(x')$ such that 
\begin{equation}\label{est: simple geo 0}
|x-x''|\le|x-x'|,
\end{equation}
and for any $\bar{x}$ on the line segment connecting $x$ with $x''$,
\begin{equation}\label{est: simple geo}
|x'|=|x''|\le \sqrt{2} |\bar{x}|.
\end{equation}

To prove the claim, 
\eqref{est: simple geo} is obvious if $x'=0$, so we assume $x'\ne 0$. Without loss of generality, assume $x$ is in the $x_2x_3$-plane. Then $Ob(x')$ intersects the $x_2x_3$-plane at several points (two points if $x'$ is on the $x_3$-axis, and four points if $x'$ is not on the $x_3$ axis). Choose $x''$ among these points such that $x''$ and $x$ belong to the same quadrant (including boundary) of the $x_2x_3$-plane. That \eqref{est: simple geo 0} holds is easy to see. Note that $x''$ is on the circle in the $x_2x_3$-plane with radius $|x''|$ and $x$ is outside this circle but in the same quadrant with $x''$. It is easy to see that for any such pairs of $x$ and $x''$ and any $\bar{x}$ between the two, the smallest $|\bar{x}|$ is given by $\frac{1}{\sqrt{2}}|x''|$.  This proves the claim. 

Noticing the elementary inequality 
\begin{equation}\label{est: zeta/x^2}
\frac{|\zeta(x)|}{|x|^2}\le C\|\zeta\|_X 
\end{equation}
from \eqref{norm X}, we now estimate
\begin{equation}\label{eq: gz(x)-gz(x')}
\gz(x)-\gz(x') = (x-x')\left(1+\frac{\zeta(x)}{|x|^2}\right)+x'\left(\frac{\zeta(x)}{|x|^2}-\frac{\zeta(x')}{|x'|^2}\right)
\end{equation}
as
\begin{equation}\label{est: gz(x)-gz(x')}
|\gz(x)-\gz(x')-(x-x')|\le C(1+\|\zeta\|_X)|x-x'|+|x'|\left|\frac{\zeta(x)}{|x|^2}-\frac{\zeta(x')}{|x'|^2}\right|
\end{equation}
To estimate $|x'|\left|\frac{\zeta(x)}{|x|^2}-\frac{\zeta(x')}{|x'|^2}\right|$, we assume without loss of generality that $|x|\ge|x'|$, and use the claim to find an $x''\in Ob(x')$ such that \eqref{est: simple geo 0} and \eqref{est: simple geo} hold. Now
\begin{align}
~&|x'|\left|\frac{\zeta(x)}{|x|^2}  - \frac{\zeta(x')}{|x'|^2}\right|
= ~|x'|\left|\frac{\zeta(x)}{|x|^2}-\frac{\zeta(x'')}{|x''|^2}\right|\notag\\
\le &~|x'|\left|\frac{\nabla \zeta(\bar{x})}{|\bar{x}^2|}-2\zeta(\bar{x})\frac{\bar{x}}{|\bar{x}|^4}\right||x-x''| 
\le ~C|x'|\frac{\|\zeta\|_X}{|\bar{x}|}|x-x'|
\le ~C\|\zeta\|_X|x-x'|.
\end{align}
By \eqref{est: gz(x)-gz(x')} we now have \eqref{est: deviation from x-x'}. Estimate \eqref{est: deviation from x-x'} in particular shows that 
$ |\gz(x)-\gz(x')|\ge \frac{1}{2}|x-x'| $
if $\|\zeta\|_X<\epsilon$ is small enough. 
Thus  $\gz$ is invertible and $\gz^{-1}$ is continuous, which implies that $\gz$ is a homeomorphism. To see that it is a diffeomorphism away from the origin, we compute 
\begin{equation}
D\gz(x) = \left(1+\frac{\zeta(x)}{|x|^2}\right)I+x\otimes\left(\frac{\nabla\zeta(x)}{|x|^2}-2\zeta(x)\frac{x}{|x|^4}\right).
\end{equation}
\eqref{est: Dg} is now obvious, which implies that $D\gz(x)$ is invertible when $x$ is away from $0$ and $\|\zeta\|_X<\epsilon$ is sufficiently small. Now for $y=\gz(x)$, 
\begin{equation}
D\gz^{-1}(y) = (D\gz(x))^{-1} = (I+(D\gz(x)-I))^{-1}=\sum_{k\ge 0}(-D\gz(x)+I)^k.
\end{equation}
Hence \eqref{est: Dg inv} follows. \eqref{est: det Dg} follows from \eqref{est: Dg} by elementary algebra. \eqref{est: g-x} and \eqref{est: g inv-x} can be obtained from \eqref{est: deviation from x-x'} easily by setting $x'=0$.
\end{proof}

It is crucial in our method to extend $\zeta$ to the entire $\mathbb{R}^3$ in a way that the symmetry and good properties of $g_{\zeta}$ are preserved. Such an {\it extension} is summarized in the following lemma.

\begin{lemma}\label{lem: extension}
Let $\tilde{X}$ be the same space as $X$ except that the $B_1$'s in the definition of $X$ are replaced by $\mathbb{R}^3$. There is a bounded linear map $T$ from $X$ to $\tilde{X}$ such that for all $\zeta\in X$:
\begin{enumerate}
\item $T\zeta$ is supported in $B_2$.
\item $T\zeta(x)=\zeta(x)$ for $x\in \overline{B_1}$.
\end{enumerate}
\end{lemma}
\begin{proof}
First of all, an extension $\zeta_1$ of $\zeta$ satisfying everything in the statement of this lemma except the symmetry requirement exists with suitable bounds on the derivatives. This can be accomplished by a partition of unity and the so-called higher-order reflection (see \cite{evans2010partial}, for instance). Once $\zeta_1$ is constructed, a symmetrized version of it can be obtained by
\begin{equation}
\zeta_2(x) = \frac{1}{4\pi}\int_0^{2\pi}[\zeta_1(T_{\theta}x)+\zeta_1(T_{\theta}Tx)]~d\theta,
\end{equation}
where
\begin{equation}
T_{\theta}(x_1,x_2,x_3)=(x_1\cos \theta+x_2\sin\theta, -x_1\sin \theta+x_2\cos\theta,x_3),
\end{equation}
and
\begin{equation}
T(x_1,x_2,x_3) = (x_1,x_2,-x_3).
\end{equation}
We easily verify that $\zeta_2$ is a $C^1$ extension of $\zeta$ satisfying all the conditions 
required by the lemma, for which the map $\zeta\mapsto\zeta_2$ is linear.
\end{proof}

By a slight abuse of notation, we will still write $\zeta$ in place of the 
extended function $T\zeta$. 
Thus in the following presentation, 
$\zeta$ is considered to be defined on the entire $\real^3$. $g_{\zeta}$ will be a 
homeomorphism on $\real^3$ and a $C^1$ diffeomorphism on $\mathbb{R}^3\setminus\{0\}$, 
where all the estimates in \Cref{lem: basic estimates g_zeta} 
are satisfied for all $x\in \mathbb{R}^3$.

Notice that $\gz(x)-x$ is linear in $\zeta$. 
Therefore all of the estimates in \Cref{lem: basic estimates g_zeta} can be applied 
to differences of $g_{\zeta_1}$ and $g_{\zeta_2}$, 
provided that $\zeta_1,\zeta_2\in B_{\epsilon}(X)$. In particular, 
\begin{equation}\label{est: Dg 12}
|Dg_{\zeta_1}(x)-Dg_{\zeta_2}(x)|<C\|\zeta_1-\zeta_2\|_X \quad \text{for all }x\in \real^3\setminus\{0\}.
\end{equation}

\begin{equation}\label{est: Dg inv 12}
|Dg_{\zeta_1}^{-1}(y)-Dg_{\zeta_2}^{-1}(y)|<C\|\zeta_1-\zeta_2\|_X 
\quad \text{for all }y\in \real^3\setminus\{0\}.
\end{equation}

\begin{equation}\label{est: det Dg 12}
|\det Dg_{\zeta_1}(x)-\det Dg_{\zeta_2}(x)|<C\|\zeta_1-\zeta_2\|_X 
\quad \text{for all }x\in \real^3\setminus\{0\}.
\end{equation}

\begin{equation}\label{est: g-x 12}
|g_{\zeta_1}(x)-g_{\zeta_2}(x)|\le  \|\zeta_1-\zeta_2\|_X |x| \quad \text{for all }x\in \real^3.
\end{equation}

\begin{equation}\label{est: g inv-x 12}
|g_{\zeta_1}^{-1}(y)-g_{\zeta_2}^{-1}(y)|\le C \|\zeta_1-\zeta_2\|_X |y| \quad \text{for all }y\in \real^3.
\end{equation}

\begin{equation}\label{est: deviation from x-x' 12}
|\left(g_{\zeta_1}(x)-g_{\zeta_1}(x')\right)-\left(g_{\zeta_2}(x)-g_{\zeta_2}(x')\right)|\le C \|\zeta_1-\zeta_2\|_X|x-x'| \quad \text{for all }x\in \real^3.
\end{equation}

To estimate the $X$ norm of the Fr\'{e}chet derivative,  the following lemma will be useful. 
						\begin{lemma}
\label{lem: Lip estimate at zero}
Let $f$ be a continuous function on $B_2$ which is axisymmetric about the $x_3$-axis and is even in $x_3$. Furthermore assume 
\begin{equation}
|f(y)-f(0)|\le C_f|y|
\end{equation}
for $y\in B_2$.
Let
\begin{equation}
V_f(y)=\int_{B_2}f(y')\frac{-(y-y')}{|y-y'|^3}~dy'. 
\end{equation}
Then for some constant $C>0$,
\begin{equation}  \label{V_f bound}
|V_f(g_{\zeta}(x))|\le C(C_f+|f(0)|)(1+\|\zeta\|_X)|x|,
\end{equation}
and for any $0<\beta<1$, there exists a constant $C_\beta$ such that 
\begin{equation}   \label{V_f difference}
|V_f(g_{\zeta_1}(x))-V_f(g_{\zeta_2}(x))|\le C_{\beta}(C_f+|f(0)|)\|\zeta_1-\zeta_2\|_X^{\beta}|x|
\end{equation}
for $x\in \overline{B_1}$ if $\zeta,\zeta_1,\zeta_2\in B_{\epsilon}(X)$ and $\epsilon$ is small enough.
\end{lemma}
\begin{proof}
The proof of this lemma is similar to that of Lemma 3.4 in \cite{rein2000stationary}.
We have
\begin{equation}\label{eq: V_f split}
V_f(\gz(x)) = f(0)\int_{B_2}\frac{-(\gz(x)-y')}{|\gz(x)-y'|^3}~dy' + \int_{B_2}(f(y')-f(0))\frac{-(\gz(x)-y')}{|\gz(x)-y'|^3}~dy'.
\end{equation}
For the first term we estimate 
\begin{align}
~\left|\int_{B_2}\frac{-(\gz(x)-y')}{|\gz(x)-y'|^3}~dy' \right| 
&= ~\left|\int_{\partial B_2}\frac{1}{|\gz(x)-y'|}n(y')~d\sigma(y') \right|  \\ 
= ~\left|\int_{\partial B_2}\left(\frac{1}{|\gz(x)-y'|}-\frac{1}{|y'|}\right)n(y')~d\sigma(y') \right| 
&\le ~C\int_{\partial B_2}\frac{1}{|\theta\gz(x)-y'|^2}|\gz(x)|~d\sigma(y') \notag\\
\le ~C\frac{1}{(2-(1+C\|\zeta\|_x))^2}(1+\|\zeta\|_X)|x| 
&\le ~C(1+\|\zeta\|_X)|x|, \notag
\end{align}
where $\theta\in (0,1)$ and the penultimate step follows by noticing that  $|y'|=2$ 
and $|\gz(x)|\le 1+C\|\zeta\|_X$ for $|x|\le 1$. On the other hand, by the symmetry of $f$ and $\zeta$, we have $$\int_{B_2}(f(y')-f(0))\frac{y'}{|y'|^3}~dy'=0.$$ 
Hence
\begin{align*}
&~\left|\int_{B_2}(f(y')-f(0))\frac{-(\gz(x)-y')}{|\gz(x)-y'|^3}~dy'\right| = \left|\int_{B_2}(f(y')-f(0))\left(\frac{-(\gz(x)-y')}{|\gz(x)-y'|^3}-\frac{y'}{|y'|^3}\right)~dy'\right|\\
\le &~C_f \left(\left|\int_{|y'|\le 2|x|}|y'|\left(\frac{-(\gz(x)-y')}{|\gz(x)-y'|^3}-\frac{y'}{|y'|^3}\right)~dy'\right|+\left|\int_{2|x|\le |y'|\le 2}|y'|\left(\frac{-(\gz(x)-y')}{|\gz(x)-y'|^3}-\frac{y'}{|y'|^3}\right)~dy'\right|\right)\\
\le &~C_f(I_1+I_2).
\end{align*}
Since $|\gz(x)|\le (1+\|\zeta\|_X)|x|$ by \eqref{est: g-x}, the ball $\{|y'|\le 2|x|\}$ is contained in the ball $\{|y'-\gz(x)|\le 4|x|\}$. Thus
\eqn
I_1\le 2\int_{|y'|\le 4|x|}\frac{1}{|y'|^2}~dy' \le C|x|.
\eeqn
We observe that $|y'|-|\gz(x)|$ is comparable to $|y'|$ when $|y'|\ge 2|x'|$ and $\|\zeta\|_X$ is small. By the mean value theorem,
\eqn
I_2\le \int_{2|x|\le| y'|\le 2}\frac{|y'|(1+\|\zeta\|_X)|x|}{|\theta \gz(x)-y'|^3}~dy' \le C (1+\|\zeta\|_X)|x| \int_{|y'|\le 2}\frac1{|y'|^2}~dy'
\eeqn
Thus \eqref{V_f bound}  follows.  
								Similarly
\begin{align}\label{eq: V_f split 12}
&~V_f(g_{\zeta_1}(x))-V_f(g_{\zeta_2}(x))\notag\\
= &~f(0)\int_{B_2}\left(\frac{-(\gzone(x)-y')}{|\gzone(x)-y'|^3}-\frac{-(\gztwo(x)-y')}{|\gztwo(x)-y'|^3}\right)~dy'\notag \\
&~+\int_{B_2}\left(f(y')-f(0)\right)\left(\frac{-(\gzone(x)-y')}{|\gzone(x)-y'|^3}-\frac{-(\gztwo(x)-y')}{|\gztwo(x)-y'|^3}\right)~dy',
\end{align}
								We can estimate the first term as before:
\begin{align}
&~\left| \int_{B_2}\left(\frac{-(\gzone(x)-y')}{|\gzone(x)-y'|^3}-\frac{-(\gztwo(x)-y')}{|\gztwo(x)-y'|^3}\right)~dy' \right|  
\le ~ \int_{\partial B_2}\left|\frac{1}{|\gzone(x)-y'|}-\frac{1}{|\gztwo(x)-y'|}\right|~d\sigma(y') \notag\\
\le &~ C \int_{\partial B_2}\frac{1}{|\theta \gzone(x)+(1-\theta)\gztwo(x) - y'|^2}|\gzone(x)-\gztwo(x)|~d\sigma(y') 
\le ~ C\|\zeta_1-\zeta_2\|_X|x|, \notag 
\end{align}
where we have used \eqref{est: g-x 12} to get the last step. The second term of \eqref{eq: V_f split 12} can be estimated as follows. Let $\bar{x}=\frac{1}{2}(\gzone(x)+\gztwo(x))$, and $\delta = \|\zeta_1-\zeta_2\|_X$. By \eqref{est: g-x 12}, $\delta|x|$ is no less than $2|\gzone(x)-\bar{x}|$ or $2|\gztwo(x)-\bar{x}|$.
\begin{align*}
&~\left|\int_{B_2}\left(f(y')-f(0)\right)\left(\frac{-(\gzone(x)-y')}{|\gzone(x)-y'|^3}-\frac{-(\gztwo(x)-y')}{|\gztwo(x)-y'|^3}\right)~dy'\right|\\
\le &~C_f \left|\int_{B_2}|y'|\left(\frac{-(\gzone(x)-y')}{|\gzone(x)-y'|^3}-\frac{-(\gztwo(x)-y')}{|\gztwo(x)-y'|^3}\right)~dy'\right|.
\end{align*}
Assuming $\|\zeta_1\|_X$ and $\|\zeta_2\|_X$ are small, we now split the integral into three pieces $I_1+I_2+I_3$, on the regions $\{|y'-\bar{x}|\le \delta |x|\}$, $\{\delta |x|\le |y'-\bar{x}|\le 2|x|\}$, and $\{|y'-\bar{x}|\ge 2|x|\}$, all within the ball $\{|y'|\le 2\}$. Since the ball $\{|y'-\bar{x}|\le \delta |x|\}$ is contained in the balls $\{|y'-\gzone(x)|\le 4\delta|x|\}$ and $\{|y'-\gztwo(x)|\le 4\delta|x|\}$, we have
\eqn
I_1\le C\int_{|y'|\le 4\delta}\frac{1}{|y'|^2}~dy' \le C\delta|x| = C\|\zeta_1-\zeta_2\|_X|x|.
\eeqn
On the other hand, when $|y'-\bar{x}|\ge \delta|x|$, $|y'-\bar{x}|$ is comparable to $|y'-\theta \gzone(x)-(1-\theta)\gztwo(x)|$ for every $\theta\in (0,1)$. By the mean value theorem,
\begin{align*}
I_2&\le C\int_{\delta|x|\le |y'-\bar{x}|\le2 |x|}\frac{|\gzone(x)-\gztwo(x)|}{|\theta \gzone(x)+(1-\theta)\gztwo(x)-y'|^3}~dy'\\
&\le C\int_{\delta|x|\le |y'-\bar{x}|\le 2|x|}\frac{|\gzone(x)-\gztwo(x)|}{|y'-\bar{x}|^3}~dy'\\\
&\le C\left(\log \frac2{\delta}\right)\delta|x| \\
&\le C_{\beta}\|\zeta_1-\zeta_2\|_X^{\beta}|x|
\end{align*}
for every $\beta\in (0,1)$. 
Finally, when $|y'-\bar{x}|\ge 2|x|$, $|y'|$ is comparable to $|y'-\theta \gzone(x)-(1-\theta)\gztwo(x)|$ for every $\theta\in (0,1)$. Again by the mean value theorem,
\begin{align*}
I_3&\le C\int_{ |y'-\bar{x}|\ge2 |x|, |y'|\le 2}\frac{|y'||\gzone(x)-\gztwo(x)|}{|\theta \gzone(x)+(1-\theta)\gztwo(x)-y'|^3}~dy'\\
&\le C\int_{|y'|\le 2}\frac{|\gzone(x)-\gztwo(x)|}{|y'|^2}~dy'\\
&\le C\|\zeta_1-\zeta_2\|_X|x|
\end{align*}
Thus \eqref{V_f difference}  follows. 
\end{proof}

\section {Euler model: analysis of the linearized operator}\label{sec: 4}
In this section we prove the main theorems (Theorem \ref{main Euler theorem} and Theorem \ref{sec Euler theorem}). 
\subsection{Linearization}  
We postpone calculating the Fr\'echet derivative of $\F(\zeta,\kappa)$ until the next section.  
$\F$ has three terms shown in \eqref{F terms} given in \eqref{F1 term}-\eqref{F3 term}.  
The formula for $\frac{\pa \F}{ \pa \zeta}$ is given in \eqref{eq: F' formula}-\eqref{eq: F3'}.  
We freely use that result here.  

It is the operator $\L=\frac{\partial \F}{\partial \zeta}(0,0)$ 
that requires detailed analysis. We write $\F$ as $\M\F_1+\kappa\F_2+\F_3$ as in \eqref{F terms}, where $\M$, $\F_1$, $\F_2$, $\F_3$ are given in \eqref{F1 term}-\eqref{F3 term}. 
First, by Lemma \ref{lem: gateaux diff} and \eqref{eq: M'} we have 
\eqn 
~\M'(0)\xi
= ~-\frac1M \int_{B_1}\rho_0(x)\nabla\cdot\left(\xi(x)\frac{x}{|x|^2}\right)~dx 
= ~\frac1M \int_{B_1}\frac{\rho_0'(x)}{|x|}\xi(x)~dx.\eeqn 
Here we integrated by parts once, and $\rho_0'$ denotes the radial derivative of $\rho_0$. 
By \eqref{F1 term},
\eqn 
\F_1(0)(x) = \int_{B_1}\rho_0(y)\left(\frac{1}{|x-y|}-\frac{1}{|y|}\right)~dy 
= h(\rho_0)(x)-h(\rho_0)(0).\eeqn 
					The last equality is a consequence of the fact that $\F(0,0)=0$, which is implied by
\begin{equation}\label{eq: restate rotless ep}
\int_{B_1}\rho_0(y)\frac{1}{|x-y|}~dy - h(\rho_0)(x) = const,
\end{equation}
which is a restatement of \eqref{eq: rotationless Euler-Poisson}. 
By \eqref{eq: F1' 0},  
\begin{align}\label{eq: F1' 0 term 2}
[\F_1'(0)\xi](x) 
= \int_{B_1}-\frac{\rho_0'(y)}{|y|}\xi(y)\left(\frac{1}{|x-y|}-\frac{1}{|y|}\right)dy + \int_{B_1}\rho_0(y)\frac{-(x-y)}{|x-y|^3}dy\cdot \xi(x)\frac{x}{|x|^2}
\end{align}
					The last term in \eqref{eq: F1' 0 term 2} can be simplified as
\begin{align}\label{eq: u_0'/x reduction}
\int_{B_1}\rho_0(y)\frac{-(x-y)}{|x-y|^3}~dy\cdot \frac{x}{|x|^2} 
= \frac{x}{|x|^2}\cdot \nabla_x \int_{B_1}\rho_0(y)\frac{1}{|x-y|}~dy 
=\frac{1}{|x|}(h(\rho_0)'(x)) 
= \frac{1}{|x|}u_0'(x).
\end{align}
Here prime ($'$) means radial derivative.
\eqref{eq: F1' 0 term 2} now becomes
\begin{align}
[\F_1'(0)\xi](x) 
= \int_{B_1}-\frac{\rho_0'(y)}{|y|}\xi(y)\left(\frac{1}{|x-y|}-\frac{1}{|y|}\right)~dy + \frac{u_0'(x)}{|x|}\xi(x).
\end{align}
By \eqref{eq: F3'}, 
\eqn
[\F_3'(0)\xi](x) = (h'(\rho_0)\rho_0(0)-h'(\rho_0)\rho_0(x))\M'(0)\xi.
\eeqn
Combining the previous equations as in Lemma \ref{lem: gateaux diff}, we have 
the following formula for the linearized operator around zero.  
\begin{align}\label{def: L}
&[\L\xi](x) = \left[\frac{\partial F}{\partial \zeta}(0,0)\xi\right](x)\notag\\
 &=  \frac{u_0'(x)}{|x|} \xi(x)   -\int_{B_1}  \frac{\rho_0'(y)}{|y|} \xi(y) \left(\frac1{|x-y|}-\frac1{|y|}\right) dy 
 +\frac{k(\rho_0)(x) - k(\rho_0)(0)}{M}  \int_{B_1} \frac{\rho_0'(y)}{|y|} \xi(y)dy ,    
\end{align}
where  $k(s)=h(s)-sh'(s)$.  

\subsection{Radial part of the Euler kernel} \label{sec: triviality Euler}
In the next two subsections we prove the following theorem 
in terms of the mass function defined in \eqref{mass function}.
\begin{theorem} \label{injectivity}
The linearized operator $\L$  is injective if and only if $M'(u_0(0))\ne 0$.
\end{theorem} 
We begin by proving that there is no nonzero {\it radial} function in the nullspace $N(\L)$ if and only if $M'(u_0(0))\ne 0$.  This is the 
most delicate part of the analysis.  It is significantly more subtle than the corresponding analysis in 
\cite{heilig1994lichtenstein}.  

From \eqref{def: L} we have, for any $\xi\in N(\L)$, 
\begin{align}    \label{xi eqn}
0 &= \frac{u_0'(x)}{|x|}\xi(x) 
-\int_{B_1}\frac{\rho_0'(y)}{|y|}\xi(y)\left(\frac{1}{|x-y|}-\frac{1}{|y|}\right)~dy \notag \\
&\qquad + \frac{1}{M}[k(\rho_0)(x)-k(\rho_0)(0)]\int_{B_1}\frac{\rho_0'(y)}{|y|}\xi(y)~dy.  
\end{align}
{\it We now assume that $\xi\in X$ is radial. }Let $\alpha=\frac{u_0'}{|x|}\xi$. Then $\alpha$ is also radial. 
Applying $\Delta$ to both sides of \eqref{xi eqn}, we have 
\begin{equation}\label{eq: alpha kernel}
\Delta \alpha +4\pi \frac{\rho_0'}{u_0'}\alpha 
+\frac{1}{M}\Delta(k(\rho_0))\int_{B_1}\frac{\rho_0'(y)}{u_0'(y)}\alpha(y)~dy=0.
\end{equation}
Integrating this equation on $B_1$ and using prime ($'$) to denote the radial derivative, we get
\begin{equation} \label{alpha'(1)}
\alpha'(1) = -\left(1+\frac{1}{M}(k(\rho_0))'(1)\right)\int_{B_1}\frac{\rho_0'(y)}{u_0'(y)}\alpha(y)~dy.
\end{equation}
						On the other hand, 
since $k(\rho_0) = h(\rho_0) - \rho_0 h'(\rho_0)  =  u_0-\frac{u_0'}{\rho_0'}\rho_0$ 
and $\Delta u_0=-4\pi \rho_0$, we have
\begin{equation} \label{u_0-k(rho_0}
\Delta (u_0-k(\rho_0))+4\pi\frac{\rho_0'}{u_0'}(u_0-k(\rho_0))+\Delta(k(\rho_0))=0.  
\end{equation}
Integrating \eqref{eq: gLane-Emden} over $B_1$, we have 
$u_0'(1) = -\int_{B_1} \rho_0 ~dx = -M$, so that 
\begin{equation}\label{eq: beta'(1)}
(u_0-k(\rho_0))'(1) = u_0'(1)-(k(\rho_0))'(1) = -M-(k(\rho_0))'(1).
\end{equation}
						Therefore if we let 
$\beta=(u_0-k(\rho_0))\frac{1}{M}\int_{B_1}\frac{\rho_0'(y)}{u_0'(y)}\alpha(y)~dy$, 
then \eqref{u_0-k(rho_0} becomes 
\begin{equation}
\Delta \beta +4\pi \frac{\rho_0'}{u_0'}\beta 
+\frac{1}{M}\Delta(k(\rho_0))\int_{B_1}\frac{\rho_0'(y)}{u_0'(y)}\alpha(y)~dy=0.
\end{equation} 
This is the same equation as satisfied by $\alpha$.  
Furthermore from \eqref{alpha'(1)} and \eqref{eq: beta'(1)}, $\beta'(1)=\alpha'(1)$.   
Therefore the difference $w=\alpha-\beta$ satisfies 
\begin{equation}\label{eq: w homo}
\Delta w +4\pi \frac{\rho_0'}{u_0'}w=0,\quad w'(1)=0.  \end{equation}

\begin{lemma}						
The following statements are equivalent:
\begin{enumerate}
\item There is no nonzero radial solution to \eqref{eq: w homo}.
\item There is no nonzero radial function in $N(\L)$.	
\end{enumerate}
\end{lemma}					  
\begin{proof}
Suppose $w=0$ is the only radial solution to \eqref{eq: w homo}, then $\alpha=\beta$. But $\alpha(0)=0$, so that,  by definition of $\beta$, we would have 
$$
0 = \beta(0) \frac M { \int_{B_1} {\rho_0'}\alpha/{u_0'} ~dy} = (u_0-k(\rho_0))(0) 
=  \lim_{r\to 0} \frac{u_0'(r)}{\rho_0'(r)} \rho_0(0) = \rho_0(0)h'(\rho_0(0))  > 0.   $$
The only possible conclusion would be that 
$ \int_{B_1} {\rho_0'}\alpha/{u_0'} ~dy =0$.  
Looking at the definition of $\beta$, 
we see that $\beta$ vanishes.   Hence $\alpha=0$, meaning that the only radial function in $N(\L)$ is zero.
 
On the other hand, if $w$ is a nonzero solution to \eqref{eq: w homo}, 
we define
\eqn\label{def: new alpha}
\alpha = w + C(u_0-k(\rho_0)) = w+C\frac{u_0'}{\rho_0'}\rho_0.
\eeqn
It follows that $C=\frac{1}{M}\int_{B_1}\frac{\rho_0'}{u_0'}\alpha ~dy$. In fact, integrate \eqref{eq: w homo} on $B_1$ to get 
\eqn
\int_{B_1}\frac{\rho_0'(y)}{u_0'(y)}w(y)~dy=0.
\eeqn
Now integrate \eqref{def: new alpha} against $\frac{\rho_0'}{u_0'}$ on $B_1$ to get
\eqn
\int_{B_1}\frac{\rho_0'}{u_0'}\alpha ~dy = C\int_{B_1}\rho_0(y)~dy=CM.
\eeqn
By \eqref{u_0-k(rho_0} and \eqref{eq: w homo}, such an $\alpha$ satisfies \eqref{eq: alpha kernel}. We now choose $C$ such that $\alpha(0)=0$:
\eqn
0 = w(0)+C\frac{u_0'}{\rho_0'}(0)\rho_0(0),
\eeqn
and define $\xi=\frac{|x|}{u_0'}\alpha$.
It follows that the right hand side of \eqref{xi eqn} is a radial harmonic function on $B_1$ which vanishes at the origin, hence is identically zero. It is easy to verify that $\xi = \frac{|x|}{u_0'}\alpha\in X$, hence is in $N(\L)$.  
\end{proof}

Our goal is therefore to study uniqueness of radial solutions to \eqref{eq: w homo}.  
Since $w$ is radial, and $\rho_0 = h^{-1}(u_0)$, we may rewrite \eqref{eq: w homo} as 
\begin{equation}  \label{eq: w}
w''+\frac{2}{r}w+4\pi(h^{-1})'(u_0)w=0,\quad w'(0)=w'(1)=0.
\end{equation}
Let us consider the closely related {\it initial} value problem ($v=v(r;a)$) 
\begin{equation}\label{eq: sigma}
v''+\frac{2}{r}v+4\pi h^{-1}(v)=0,\quad v(0;a)=a,\quad v'(0;a)=0.
\end{equation}
Of course,  $v$ is differentiable with respect to $a$. 
Since $u_0$ solves \eqref{eq: sigma} with $a=u_0(0)$, we have $v(\cdot;u_0(0))=u_0$. 
						Also 
\begin{equation}
v(1,u_0(0))=u_0(1) = 0, \qquad v'(1,u_0(0)) = u_0'(1) \ne 0.
\end{equation}
By the implicit function theorem, there exists a function $R=R(a)$ for which $R(u_0(0))=1$ 
that solves  $v(R,a)=0$ for $(R,a)$ in a neighborhood of $(1,u_0(0))$. 
For such values of $(R,a)$, $v(\cdot;a)$ is a positive solution to the equation
\begin{equation}
\Delta v +4\pi h^{-1}(v)=0
\end{equation}
on $B_R$ with zero boundary value at $|x|=R$.        
We now define the {\it physical mass} of the solution $v(\cdot;a)$ as 
\begin{equation} \label{mass function}
M(a) = \int_{B_{R(a)}} h^{-1}(v(|x|,a))\ dx   =   \int_0^{R(a)}4\pi h^{-1}(v(r;a))\,r^2~dr. 
\end{equation}
Thus $M(u_0(0))=M$.  It suffices to prove the following lemma.  
\begin{lemma}\label{lem: w radial uniqueness} 
\eqref{eq: w homo} has a unique radial solution that is trivial if and only if $M'(u_0(0))\ne 0$.  
\end{lemma}
The condition of the lemma means that near $\rho_0$, the mass of a radial solution strictly changes 
as the density at the origin changes.  
						\begin{proof}
A direct computation yields
\begin{align}
M'(a)&=4\pi h^{-1}(v(R(a);a))R^2(a)R'(a) + \int_0^{R(a)}4\pi (h^{-1})'(v(s;a))v_a(s;a)s^2~ds\notag\\
&= \int_0^{R(a)}4\pi (h^{-1})'(v(s;a))v_a(s;a)s^2~ds
\end{align}
because $v(R(a);a)=0$ and $h^{-1}(0)=0$. 
In particular, putting $a=u_0(0)$, we have $R(a)=1$ 
so that 
\begin{equation}
M'(u_0(0)) = \int_0^1 4\pi (h^{-1})'(u_0(s))\ v_a(s;u_0(0))s^2~ds.
\end{equation}
On the other hand, for $a=u_0(0)$ the derivative  $v_a(\cdot;u_0(0))$ solves 
\begin{equation}
v_a'' +\frac2r v_a + 4\pi (h^{-1})'(u_0)v_a=0, \quad v_a(0,u_0(0))=1, \quad v_a'(0,u_0(0))=0.
\end{equation}
This is the same as the equation \eqref{eq: w} satisfied by $w$.  
Thus  $w$ must be a constant multiple of $v_a(\cdot;u_0(0))$. 
Integrating the equation $\Delta v_a + 4\pi (h^{-1})'(u_0)v_a=0$ on $B_1$, we get
\begin{equation}
v_a'(1;u_0(0))=-\int_0^1 4\pi (h^{-1})'(u_0(s))v_a(s;u_0(0))s^2~ds=-M'(u_0(0)). 
\end{equation}
If $M'(u_0(0))\ne 0$, together with the condition $w'(1)=0$, this  implies that $w$ must be the zero multiple of 
$v_a(\cdot;u_0(0))$.   So $w=0$. If $M'(u_0(0))=0$, then $w=v_a(\cdot; u_0(0))$ is a nontrivial radial solution to \eqref{eq: w homo}.  
\end{proof}

\begin{remark}
The mass function in this section is parametrized by the center value of $u$, while the mass function in the statement of Theorem \ref{main Euler theorem} is parametrized by the center value of $\rho$. It is easy to see that $M'$ is nonzero in one parametrization if and only if it is nonzero in the other, since $h'(s)>0$ for $s>0$.
\end{remark}

\subsection {Non-radial part of the Euler kernel}\label{sec: non-radial kernel}
In this subsection, we treat the non-radial part of $N(\L)$, thus concluding Theorem \ref{injectivity}.

\begin{lemma}  
All elements of the nullspace of $\L$ are radial.  \end{lemma} 

\begin{proof}
Let $\L\xi=0,\ \xi\in X$.  Let $Y_{lm}(\theta)$ be the standard spherical harmonics on $\mathbb{S}^2$ 
where $l=0,1,\dots ;m=-l,\dots,l$.  
For any function $\eta(x)$, we denote  its $lm$-component by  
$\eta_{lm}(r) = \langle\eta,Y_{lm}\rangle = \int_{\mathbb{S}^2} \eta(r\theta) Y_{lm}(\theta) dS_\theta,  $
where we write $r=|x|$ and $\theta=\frac xr$.  

It suffices to prove that $\xi_{lm}=0$ for all $l\ge1$.  
The equation \eqref{def: L} satisfied by $\xi$ is 
\eqn \label{xi phi eqn}
0 =  \frac{u_0'(x)}{|x|} \xi(x)   - \varphi(x)  
 +\frac{k(\rho_0)(x) - k(\rho_0)(0)}{M}  \int_{B_1} \frac{\rho_0'(y)}{|y|} \xi(y)~dy,  \eeqn
where 
\eqn\label{def: varphi}
\varphi(x) = \int_{B_1}  \frac{\rho_0'(y)}{|y|} \xi(y) \left(\frac1{|x-y|}-\frac1{|y|}\right)    ~ dy.  \eeqn
						Since 
$\rho_0\in C^{1,\alpha}(\real^3)$ by Lemma \ref{lem: rho_0 C1alpha}, 
and $\rho_0'=0$ on $\partial B_1$, we have $\varphi\in C^{2,\alpha}(\real^3)$ and
\eqn\label{eq: delta varphi}
\Delta\varphi (x)= -4\pi\frac{\rho_0'(x)}{|x|} \xi(x),
\eeqn
where the right side is taken to be zero outside $B_1$. 
Note that $\varphi_{lm}$ is bounded, radial and in $C^{2,\alpha}$ on $\real^3\setminus\{0\}$. 
We integrate \eqref{eq: delta varphi} against $Y_{lm}$ on $\mathbb{S}^2$. The left side becomes 
$$
\langle \Delta \varphi, Y_{lm}\rangle 
= \langle \partial_r^2\varphi +\frac{2}{r}\partial_r \varphi 
+\frac{1}{r^2}\Delta_{\mathbb{S}^2}\varphi, Y_{lm}\rangle $$ 
$$ 
= (\partial_r^2+\frac2r\partial_r)\varphi_{lm}
+\frac{1}{r^2}\langle \varphi, \Delta_{\mathbb{S}^2}Y_{lm}\rangle 
= \Delta\varphi_{lm}-\frac{l(l+1)}{r^2}\varphi_{lm}.$$ 
						For $r=|x| > 1$ the right side vanishes.
For $0<|x|<1$, the right side is 
$-4\pi \frac{\rho_0'}{|x|}\langle \xi, Y_{lm}\rangle. $ 
On the other hand, if we integrate \eqref{xi phi eqn} against $Y_{lm}$ with $l\ge1$, we get 
$\frac{u_0'}{|x|} \xi_{lm} = \varphi_{lm}$.  
Hence
$-4\pi \frac{\rho_0'}{|x|}\xi_{lm} = -4\pi \frac{\rho_0'}{u_0'}\varphi_{lm}. $
Thus 
$\varphi_{lm}$ satisfies the second order equation 
\eqn\label{eq: ode varphi lm}
\Delta \varphi_{lm}-\frac{l(l+1)}{r^2}\varphi_{lm} = \begin{cases}-4\pi\frac{\rho_0'}{u_0'}\varphi_{lm} &\text{if }0<|x|<1,\\ 0 &\text{if }|x|\ge 1. \end{cases}
\eeqn
For $r>1$ it follows that $\varphi$ is a linear combination of $r^l$ and $r^{-(l+1)}$.  
Being bounded at infinity, it must be that 
\eqn  \label{varphi infinity}
\varphi_{lm}(r) = C_{lm}r^{-(l+1)}  \eeqn
when $|x|\ge 1$ for some constant $C_{lm}$.

Now consider $r<1$.  
Recall that $u_0$ and $\rho_0$ are related by
$ u_0'' + \frac{2}{r}u_0' + 4\pi \rho_0 = \Delta u_0 +4\pi \rho_0 = 0. $  
Taking the radial derivative, we get
\eqn\label{eq: del u0'}
\Delta u_0'-\frac{2}{r^2}u_0' +4\pi \rho_0'=0.
\eeqn
In terms of another auxiliary function  
$\psi_{lm}(r) = \frac{\varphi_{lm}(r)}{u_0'(r)} = \frac {\xi_{lm}(r)}{r} $, 
\eqref{eq: ode varphi lm} can be written for $0<r=|x|<1$ as 
\eqn\label{eq: del psi lm 2}
\Delta (u_0' \psi_{lm} )-\frac{l(l+1)}{r^2}u_0'\psi_{lm}=-4\pi \rho_0'\psi_{lm}.
\eeqn
Dividing \eqref{eq: del psi lm 2} by $u_0'$ and using \eqref{eq: del u0'}, we see that 
$\psi_{lm}$ satisfies the equation 
\eqn\label{eq: del psi lm 1}
\Delta \psi_{lm}+2\frac{u_0''}{u_0'}\psi_{lm}' + \frac{2-l(l+1)}{r^2}\psi_{lm}=0
\eeqn
for $0<|x|< 1$. 
It suffices to prove that $\psi_{lm}=0$ for all $l\ge1$.  

Since $\left|\frac{\xi_{lm}(r)}{r}\right|\le C\|\xi\|_X r$, we know that $\psi_{lm}\in C(\overline B_1)$ and $\psi_{lm}(0)=0$.  
Note that the numerator in the last term of \eqref{eq: del psi lm 1} is $2-l(l-1)\le0$.  
Thus \eqref{eq: del psi lm 1} allows us to apply the strong maximum principle for $0<r<1$.  
There is no interior positive maximum of $\psi_{lm}$.  
Let $\Psi$ be the maximum of $\psi_{lm}$ over $\overline B_1$.  

Suppose that  $\Psi>0$.  Then the maximum occurs at $r=1$.  Furthermore, $\psi_{lm}'(1) > 0$.  
But from the definition of $\psi_{lm}$ and \eqref{varphi infinity}, we have 
$$
0< \Psi = \psi_{lm}(1) = \frac {C_{lm}} {u_0'(1)}, \qquad  
0 <  \psi_{lm}'(1) = \frac{C_{lm} (1-l)} {u_0'(1)}.  $$
We know that $u_0'(1) < 0$.  From these two inequalities we conclude that  $C_{lm}<0$ 
and $1-l>0$, which is a contradiction.   Therefore $\Psi\le0$ and $\psi_{lm}\le0$.  

There is also no interior negative minimum of $\psi_{lm}$ and by similar reasoning we deduce that 
$\psi_{lm} \ge0$.  Therefore $\psi_{lm} = 0$.  This completes the proof.  
\end{proof}

\subsection{Compactness}\label{sec: compactness}
\begin{lemma}\label{lem: structure L}
$\L: X\to X$ has the form $\L=J+K$ where $J$ is an isomorphism, and $K$ is a compact operator.
\end{lemma}
\begin{proof}
In fact, let
\eqn
[J\xi](x) = \frac{u_0'(x)}{|x|} \xi(x),
\eeqn
and
\eqn\label{def: K}
[K\xi](x) =  -\int_{B_1}  \frac{\rho_0'(y)}{|y|} \xi(y) \left(\frac1{|x-y|}-\frac1{|y|}\right)    ~ dy+\frac1M [k(\rho_0)(x) - k(\rho_0)(0)]  \int \frac{\rho_0'(y)}{|y|} \xi(y)~ dy.
\eeqn
To show $J$ is bounded, we write the spatial derivative of $J\xi$ as
\eqn
\partial_i [J\xi](x)  = \left[|x|\partial_i\left(\frac{u_0'(x)}{|x|}\right)\frac{\xi(x)}{|x|^2}+\frac{u_0'(x)}{|x|}\frac{\partial_i\xi(x)}{|x|}\right]|x|. 
\eeqn
We need the terms in the square brackets to be bounded by $C\|\xi\|_X$. 
Since $\frac{\xi(x)}{|x|^2}$ and $\frac{\nabla \xi(x)}{|x|}$ are both bounded by $C\|\xi\|_X$, 
we only need $|x|\partial_i\left(\frac{u_0'(x)}{|x|}\right)$ and $\frac{u_0'(x)}{|x|}$ 
to be bounded in $B_1$. 
This is true because  $u_0'(0)=0$. On the other hand, 
$[J^{-1}\xi](x) = \frac{|x|}{u_0'(x)}\xi(x).$
As above, in order for $J^{-1}$ to be bounded, we need $|x|\partial_i\left(\frac{|x|}{u_0'(x)}\right)$ and $\frac{|x|}{u_0'(x)}$ to be bounded in $B_1$. This is again easy to verify with the help of Lemma \ref{lem: u_0}. The key facts are that $u_0'(0)=0$, $u_0''(0)\ne 0$, and $u_0'(r)<0$ for all $0<r\le 1$.

Now we write $K=A+B$, where $A$ and $B$ are the two terms in \eqref{def: K}. The operator $B$ has rank one, 
hence is compact once we verify that it maps $X$ to $X$. 
To that end we just need $k(\rho_0)(x) - k(\rho_0)(0)$ to belong to $X$.  We compute 
$$
\partial_i k(\rho_0) = k'(\rho_0)\partial_i \rho_0 
= -\rho_0h''(\rho_0)\partial_i\rho_0 
= -\rho_0\frac{h''(\rho)}{h'(\rho_0)}\partial_iu_0.  $$
By  \eqref{eq: lim h''} and the fact that $|\partial_i u_0(x)|\le C|x|$, 
we see that $|\pa_ik(\rho_0)|$  is bounded by $C|x|$ on $B_1$.
We now show compactness of $A$.
Recall that $\frac{\rho_0'(y)}{|y|}\in C^{\alpha}(\overline{B_1})$ 
and $\|\xi\|_{C^1(\overline{B_1})}\le C\|\xi\|_X$ so that 
\eqn
\|{\rho_0'(y)}\xi(y)/|y|\|_{C^{\alpha}(\overline{B_1})}\le C\|\xi\|_X.
\eeqn
But the standard Schauder estimates  assert that 
\eqn
\|A\xi\|_{C^{2,\alpha}(\overline{B_1})}\le C\|{\rho_0'(y)}\xi(y)/|y|\|_{C^{\alpha}(\overline{B_1})} \eeqn
since $\rho_0'$ vanishes on the boundary of $B_1$. 
Because $C^{2,\alpha}(\overline{B_1})$ is compactly embedded in $C^2(\overline{B_1})$  
and because $A\xi(0)=0$ and $\|A\xi\|_X\le C\|A\xi\|_{C^2(\overline{B_1})}$, 
we conclude that $A:X\to X$ is compact. 
\end{proof}

Invoking the Fredholm alternative for $J^{-1}\L$ and combining the last three lemmas, 
we deduce that $\L$ is an isomorphism. By the discussion in Section \ref{sec: construction}, this completes the proof of Theorem \ref{main Euler theorem}. 
It only remains to prove the Fr\'echet differentiability of $\F$.




\subsection{Realization of the mass condition}
Theorem \ref{injectivity} shows that injectivity of $\L$ is equivalent to the mass condition $M'(u_0(0))\ne 0$. In this subsection, we prove for the two cases in Theorem \ref{sec Euler theorem} this condition is satisfied.

First, we consider the power laws.
\begin{lemma}
Suppose $p(s)=s^{\gamma}$, $\gamma\in(\frac65, 2)$. Then $M'(u_0(0))\ne 0$ if and only if $\gamma\ne \frac43$.
\end{lemma}
\begin{proof}
We begin by noticing that $p(s)=s^{\gamma}$, $\gamma\in (\frac65,2)$ satisfies \eqref{assum: 2}, \eqref{assum: 3} and \eqref{assum: 4}. As a consequence the radial density $\rho_0$, and the function $u_0$ exist as are constructed in Section \ref{sec: radial Euler}. By the proof of Lemma \ref{lem: w radial uniqueness}, $M'(u_0(0))=-v_a'(1;u_0(0))$. To find out about $v_a'(1;u_0(0))$, we compute $h^{-1}(s)=\left(\frac{\gamma-1}{\gamma}s\right)^{\frac{1}{\gamma-1}}$. By \eqref{eq: sigma}, $v(r;a)$ solves
\eqn\label{eq: v power ode}
v''+\frac2r v' +4\pi\left(\frac{\gamma-1}{\gamma}v\right)^{\frac{1}{\gamma-1}}=0,\quad v(0)=a, v'(0)=0.
\eeqn
By a simple scaling, we see that 
\eqn\label{eq: v scaling}
R^{\frac{2(\gamma-1)}{2-\gamma}}v(Rr;a) = v(r;R^{\frac{2(\gamma-1)}{2-\gamma}}a).
\eeqn
In fact, the two sides of \eqref{eq: v scaling} solve the same ODE with the same initial data. We differentiate \eqref{eq: v scaling} with respect to $R$ and set $R=1$, $a=u_0(0)$ to get
\eqn
\frac{2(\gamma-1)}{2-\gamma}v(r;u_0(0))+ rv'(r;u_0(0))=u_0(0)\frac{2(\gamma-1)}{2-\gamma}v_a(r;u_0(0)).
\eeqn
Now taking the $r$ derivative and setting $r=1$, we get
\begin{align}
u_0(0)\frac{2(\gamma-1)}{2-\gamma}v_a'(1;u_0(0)) &= \frac{2(\gamma-1)}{2-\gamma}v'(1;u_0(0))+ v'(1;u_0(0))+v''(1;u_0(0))\notag \\
&=\frac{2(\gamma-1)}{2-\gamma}v'(1;u_0(0))+ v'(1;u_0(0))-2v'(1;u_0(0))\notag\\
&=\frac{3\gamma-4}{2-\gamma}v'(1;u_0(0)), \label{eq: 4/3}
\end{align}
where we used \eqref{eq: v power ode} to write $v''$ in terms of $v'$. Since $v(r;u_0(0))=u_0(r)$, and $u_0(1)=0$, it follows that $v'(1;u_0(0))=u_0'(1) \ne 0$. Noticing that $u_0(0)\ne 0$, $\frac{2(\gamma-1)}{2-\gamma}\ne 0$ for $\gamma \in (\frac65, 2)$, we get $v_a'(1;u_0(0))\ne 0$ if and only if $\gamma \ne \frac43$ from \eqref{eq: 4/3}.
\end{proof}

The case with condition \eqref{p extra cond} is a lot more difficult to prove, because we no longer have the scaling symmetry provided by an exact power law. We first restate \eqref{p extra cond} in terms of $h^{-1}$.

\begin{lemma}\label{lem: h extra cond}
\eqref{p extra cond} implies
\eqn\label{h inv extra cond}
h^{-1}(s)< s(h^{-1})'(s) \le 2h^{-1}(s) \text{ for }s>0.
\eeqn
\end{lemma}
\begin{proof}
Denote $h^{-1}(s)$ by $t$. \eqref{h inv extra cond} can be rewritten as
\eqn
th'(t)<h(t)\le 2 th'(t),
\eeqn
which is \eqref{p extra cond} by the definition of $h$ in terms of $p$.
\end{proof}

To simplify equation, we make a change of variable $v(r)=\frac{w(r)}{r}$. If $v$ satisfies \eqref{eq: sigma}, then $w$ satisfies
\begin{equation}\label{eq: w g}
w''(r) + 4\pi g(w,r)=0, \quad w(0)=0, w'(0) = a,
\end{equation}
where $g(w,r)=rh^{-1}(\frac wr)$. Denote the solution to \eqref{eq: w g} by $w(r;a)$, then $v(r;a)=\frac{w(r;a)}{r}$, and $v_a'(1;u_0(0))=w_a'(1;u_0(0))-w_a(1;u_0(0))$. By the proof of Lemma \ref{lem: w radial uniqueness}, this is equal to $-M'(u_0(0))$. Hence we are left to show $w_a'(1;u_0(0))-w_a(1;u_0(0))\ne 0$.

To that end, we need the following Lemma, which extends Theorem 2.4 in \cite{ni1985uniqueness}.
\begin{lemma}\label{lem: Ni Nussbaum}
Suppose $g:\mathbb{R}\times (0,\infty) \to \mathbb{R}$ is a $C^1$ map such that $g(w,r)>0$ when $w>0$ and $g(0,r)=0$ when $r>0$. Let $w(r;a)$ be the solution to the initial value problem
\begin{equation}
w''(r)+g(w,r)=0, \quad w(0)=0, w'(0)=a,
\end{equation} 
and $w(1;a_0)=0$, $w(r;a_0)>0$ for $0<r<1$. In other words, $w(r;a_0)$ is a positive solution to the boundary value problem
\begin{equation}\label{eq: w ode}
w''(r)+g(w,r)=0, \quad w(0)=0, w(1)=0.
\end{equation} 
If we further assume the following conditions on $g$:
\begin{equation}
\lim_{r\to 0^+}g(w(r;a),r)=0 \text{ for } a \text{ close to }a_0.
\end{equation}
\begin{equation}\label{cond: g 1}
g-wg_w<0 \text{ for }w>0 , 0< r\le 1.
\end{equation}
\begin{equation}\label{cond: g 2}
g_r\le 0 \text{ for }w>0 , 0< r\le 1.
\end{equation}
\begin{equation}\label{cond: g 3}
rg_r+3g -wg_w\ge 0 \text{ for }w>0, 0< r\le 1.
\end{equation}
Then $w_a'(1,a_0)-w_a(1,a_0)< 0$.
\end{lemma}
\begin{proof}
We define several auxiliary functions. Let 
\begin{equation}
x(r;a) = rw'(r;a), ~y(r;a)=w'(r;a), ~z(r;a)=w_a(r;a).
\end{equation}
One has 
\begin{equation}
x(0^+;a)=0, ~x'(0^+;a) = w'(0;a)-\lim_{r\to 0^+}rg(w,r)=a.
\end{equation}
\begin{equation}
y(0^+;a) = a,~ y'(0^+;a) = -\lim_{r\to 0^+}g(w,r)=0.
\end{equation}
\begin{equation}\label{phi int}
z(0^+;a) = 0,~z'(0^+;a) =1. 
\end{equation}
We have
\begin{align}
x'' &= (rw')'' \notag\\
&= rw'''+2w''\notag\\
&= r(-g_r-g_ww')-2g\notag\\
&=-rg_r-g_w x-2g.
\end{align}
Here we used \eqref{eq: w ode}.
Hence $x$ satisfies the equation
\begin{equation}
x'' + g_w x +rg_r +2 g=0.
\end{equation}
Similarly $y$ satisfies the equation
\begin{equation}
y'' + g_w y + g_r=0.
\end{equation}
$z$ satisfies
\begin{equation}
z'' + g_wz=0.
\end{equation}
From these we get the ODEs of various Wronskians between $w$, $x$, $y$ and $z$:
\begin{equation}\label{wsk 1}
W(x,z)'=\begin{vmatrix} x &z \\ x'& z'\end{vmatrix}' = z(rg_r+2g).
\end{equation} 
\begin{equation}\label{wsk 2}
W(y,z)'=\begin{vmatrix} y &z \\ y' &z'\end{vmatrix}' = z g_r.
\end{equation} 
\begin{equation}\label{wsk 3}
W(w,z)'=\begin{vmatrix} w &z \\ w'& z'\end{vmatrix}' = z(g-wg_w).
\end{equation}
In the following, we set $a$ equal to $a_0$ in all functions. Since $w>0$ and $w''=-g<0$ for $r\in (0,1)$, we know that $w$ is a positive convex function with a unique maximum and zero boundary value on $[0,1]$. By \eqref{phi int}, $z(r)>0$ for $r$ close to $0$. We claim that there is an $r_0\in (0,1)$ such that $z(r_0)=0$. If not, then $z(r)>0$ for all $r\in (0,1)$. Integrating \eqref{wsk 3} on $(0,1)$ and using the boundary conditions of $w$ and $z$, we have
\begin{equation}\label{eq: contra 1}
-w'(1)z(1) = \int_0^1 z(g-wg_w)~dr<0.
\end{equation}
The inequality is a consequence of \eqref{cond: g 1}. However, we know that $w'(1)<0$ and $z(1)\ge 0$, making the left hand side of \eqref{eq: contra 1} non-negative. Such a contradiction implies the existence of $r_0$ mentioned above. We assume $r_0$ is the smallest value in $(0,1)$ for which $z(r_0)=0$. \eqref{cond: g 1} and \eqref{cond: g 3} imply $rg_r+2g>0$. If we now integrate \eqref{wsk 1} on $(0,r_0)$,
\begin{align}
x(r_0)z'(r_0)&= \int_0^{r_0}z(rg_r+2g)~dr>0
\end{align}
The last inequality follows from the fact that $z(r)>0$ for $r\in (0,r_0)$. Since $z'(r_0)<0$, we have $x(r_0)<0$, which also implies that $w'(r_0)<0$.   

Our next claim is that $z(r)<0$ for all $r_0<r\le 1$. If not, let $r_1\in (r_0,1]$ be the first zero of $z$ strictly bigger than $r_0$.   Integrating \eqref{wsk 2} on $(r_0,r_1)$ and recalling the definition of $y$, we obtain 
\begin{equation}\label{eq: contra 2}
w'(r_1)z'(r_1) - w'(r_0)z'(r_0) =y(r_1)z'(r_1) - y(r_0)z'(r_0) = \int_{r_0}^{r_1}z g_r~dr\ge 0.
\end{equation}
The last inequality follows from \eqref{cond: g 2} and the fact that $z(r)<0$ for $r\in (r_0,r_1)$. However, the fact that $w'(r_0)<0$ implies $w'(r_1)<0$. We also have $z'(r_0)<0$, and $z'(r_1)>0$. These conditions together imply that the left hand side of \eqref{eq: contra 2} is negative. Such a contradiction implies $z(r)<0$ for all $r_0<r\le 1$.   

We integrate $W(x,z)'+r_0W(w,z)'$ between $0$ and $r_0$, and integrate $W(y,z)'+W(w,z)'$ between $r_0$ and 1 to get
\eqn\label{eq: sign 1}
r_0[w'(r_0)+w(r_0)]z'(r_0) = \int_0^{r_0}z(rg_r+2g+r_0g - r_0wg_w)~dr.
\eeqn
\eqn\label{eq: sign 2}
w'(1)[z'(1)-z(1)] = z'(r_0)(w'(r_0)+w(r_0))+\int_{r_0}^1 z(g_r+g-wg_w)~dr.
\eeqn
Notice \eqref{cond: g 1} and \eqref{cond: g 2} imply $rg_r+2g+r_0g - r_0wg_w\ge 0$ for all $r_0\in (0,1)$. This together with the fact that $z(r)>0$ for $r\in (0,r_0)$ imply that \eqref{eq: sign 1} is non-negative. Since $r_0\in (0,1)$, it follows that the first term on the right hand side of \eqref{eq: sign 2} is non-negative. But the second term on the right hand side of \eqref{eq: sign 2} is positive by \eqref{cond: g 1}, \eqref{cond: g 2}, and the fact that $z(r)<0$ for $r\in (r_0,1)$. Therefore \eqref{eq: sign 2} is positive. Since $w'(1)<0$ as is explained above, we get $z'(1)-z(1)<0$, which is the conclusion of the lemma by the definition of $z$.
\end{proof}
 
To apply Lemma \ref{lem: h extra cond}, we verify the conditions on 
\begin{equation}
g(w,r) = 4\pi r h^{-1}\left(\frac w r\right).
\end{equation}
Condition \eqref{cond: g 1} reads:
\begin{equation}
rh^{-1}\left(\frac w r\right)-w \left(h^{-1}\right) '\left(\frac w r\right)<0 \text{ for } w>0 \text{ and }0<r\le 1,
\end{equation}
which is a consequence of
\begin{equation}\label{cond: g 11}
s(h^{-1})'(s)>h^{-1}(s) \text{ for } s>0.
\end{equation}
This is included in \eqref{h inv extra cond} in Lemma \ref{lem: h extra cond}.
Condition \eqref{cond: g 2} reads:
\begin{equation}
h^{-1}\left(\frac w r\right) -\frac w r \left(h^{-1}\right)'\left(\frac w r\right)\le 0, 
\end{equation}
which is implied by \eqref{cond: g 11}. 
Condition \eqref{cond: g 3} reads
\eqn
h^{-1}\left(\frac w r\right) -\frac w r \left(h^{-1}\right)'\left(\frac w r\right) + 3h^{-1}\left(\frac w r\right) - \frac wr \left(h^{-1}\right)'\left(\frac w r\right)\ge 0
\eeqn
or 
\eqn
2h^{-1}\left(\frac w r\right)-\frac w r \left(h^{-1}\right)'\left(\frac w r\right)\ge 0.
\eeqn
This is a consequence of 
\eqn
s(h^{-1})'(s)\le 2h^{-1}(s) \text{ for }s>0,
\eeqn
which is included in \eqref{h inv extra cond}. By the discussion above Lemma \ref{lem: Ni Nussbaum}, it implies the second case in Theorem \ref{sec Euler theorem}.
  
\subsection{Example: oblateness for constant rotation}

In this subsection, we exhibit the approximate support of $\rho_{\kappa}$ as constructed in 
Theorem \ref{main Euler theorem} when $\kappa$ is small in case  the rotation profile $\omega$ is constant.   
In particular, the approximation shows that the support of $\rho_{\kappa}$ is wider at the equator than at the poles. 
Thus the body has an oblate shape. 

As is explained in Section \ref{sec: construction}, $\rho_{\kappa}$ is constructed by deforming the ball 
using $g_{\zeta_{\kappa}}$. Therefore the boundary of the support of $\rho_{\kappa}$ 
is precisely $g_{\zeta_{\kappa}}(\partial B_1) = \{x(1+\zeta_{\kappa}(x)): x\in \partial B_1\}$. 
We just need to find out the value of $\zeta_{\kappa}$ on the boundary.   
If the conditions of Theorem \ref{main Euler theorem} are satisfied, we  expand $\zeta_{\kappa}$ 
 near $\kappa=0$ and use the implicit function theorem to evaluate the first derivative:
\eqn\label{eq: zeta near zero}
\zeta_\kappa= -\kappa\left(\frac{\partial \F}{\partial \zeta}(0,0)\right)^{-1}\frac{\partial \F}{\partial \kappa}(0,0) + R(\kappa)
\eeqn
where $\|R(\kappa)\|_X=o(\kappa)$ as $\kappa \to 0$.   
We now study the dominant term.  
 By \eqref{def: Euler F} and \eqref{eq: rotationless Euler-Poisson}, 
\eqn
\frac{\partial \F}{\partial \kappa}(0,0) = \int_0^{r(x)}\omega^2(s)s~ds=\frac12 r^2(x)=\frac{1}{2}r^2\sin^2\theta_1 = r^2\left(\frac23\sqrt{\pi}Y_{00}-\frac23\sqrt{\frac{\pi}{5}}Y_{20}\right).
\eeqn
Here we assume $\omega(r)=1$ as an example of a constant rotation profile, 
and used spherical coordinates $(r,\theta_1,\theta_2)$ to write the last expression.  
The spherical harmonics $Y_{00}$ and $Y_{20}$ are 
\eqn
Y_{00}(\theta_1,\theta_2) = \frac12\sqrt{\frac{1}{\pi}},\quad Y_{20}(\theta_1,\theta_2) = \frac14\sqrt{\frac{5}{\pi}}(3\cos^2\theta_1-1).
\eeqn
Using the expression for $\frac{\partial \F}{\partial \zeta}(0,0)$ given in \eqref{def: L}, we have
\eqn\label{eq: xi as deriv}
-\left(\frac{\partial \F}{\partial \zeta}(0,0)\right)^{-1}\frac{\partial \F}{\partial \kappa}(0,0)=\xi
\eeqn
where $\xi\in X$ is the unique solution to the equation 
\begin{align}\label{eq: xi omega}
&\frac{u_0'(x)}{|x|} \xi(x)   -\int_{B_1}  \frac{\rho_0'(y)}{|y|} \xi(y) \left(\frac1{|x-y|}-\frac1{|y|}\right) dy 
 +\frac{k(\rho_0)(x) - k(\rho_0)(0)}{M}  \int_{B_1} \frac{\rho_0'(y)}{|y|} \xi(y)dy \notag\\
 &= r^2\left(-\frac23\sqrt{\pi}Y_{00}+\frac23\sqrt{\frac{\pi}{5}}Y_{20}\right).
\end{align}
The existence and uniqueness of $\xi$ follows from the fact that $\frac{\partial \F}{\partial \zeta}(0,0)$ 
is an isomorphism, as is established in Section \ref{sec: compactness}.  
Since we are only interested in the deviation from spherical symmetry of the support of $\rho_{\kappa}$, 
we pay attention only to the non-radial components of $\xi$.  
We proceed in a fashion parallel to Section \ref{sec: non-radial kernel}.  
Defining $\varphi$ as in \eqref{def: varphi} and 
taking the Laplacian of $\varphi$, we again arrive at \eqref{eq: delta varphi}. 
Projecting onto spherical harmonics $Y_{lm}$ for $l\ge 1$ as before, 
we get  \eqref{eq: ode varphi lm} so long as $(l,m)\ne (2,0)$. 
The same argument as in section \ref{sec: non-radial kernel} shows that $\varphi_{lm}=0$ for $(l,m)\ne (2,0)$. 
However, for $(l,m)=(2,0)$, we get
\eqn
\Delta \varphi_{20}-\frac{6}{r^2}\varphi_{20}=\begin{cases}-4\pi \frac{\rho_0'}{|x|}\xi_{20} &\text{if }|x|\le 1,\\ 0 &\text{if }|x|>1.\end{cases}
\eeqn
					Projecting 
\eqref{eq: xi omega} onto $Y_{20}$, we get
\eqn\label{eq: relation xi varphi}
\frac{u_0'}{|x|}\xi_{20}-\varphi_{20}=\frac{2}{3}\sqrt{\frac{\pi}{5}}r^2.
\eeqn
Hence
\eqn
\Delta \varphi_{20}-\frac{6}{r^2}\varphi_{20}=\begin{cases}-4\pi \frac{\rho_0'}{u_0'}\left(\varphi_{20}+\frac{2}{3}\sqrt{\frac{\pi}{5}}r^2\right) &\text{if }|x|\le 1,\\ 0 &\text{if }|x|>1.\end{cases}
\eeqn
As before, we deduce $\varphi_{20}=\frac{C}{r^3}$ for $r\ge 1$.   
For $r<1$, the function $\psi_{20}=\frac{\varphi_{20}}{u_0'}$ satisfies 
\eqn
\Delta \psi_{20}+2\frac{u_0''}{u_0'}\psi_{20}'-\frac{4}{r^2}\psi_{20}=-4\pi \frac{\rho_0'}{(u_0')^2}\frac{2}{3}\sqrt{\frac{\pi}{5}}r^2>0
\eeqn
for $0<r<1$. The same maximum principle argument as in Section \ref{sec: non-radial kernel} 
yields $\psi_{20}(r)\le 0$ for $0<r\le1$. 
Since $u_0'<0$ for $0<r\le 1$ by Lemma \ref{lem: u_0}, 
$\varphi_{20}(r) \ge 0$ for $0<r\le 1$. By \eqref{eq: relation xi varphi},
\eqn
\xi_{20}(1) = \frac{1}{u_0'(1)}\left(\varphi_{20}(1)+\frac{2}{3}\sqrt{\frac{\pi}{5}}\right)\le \frac{1}{u_0'(1)}\frac{2}{3}\sqrt{\frac{\pi}{5}}<0.
\eeqn
The non-radial component of $\xi$ at $\partial B_1$ is given by 
\eqn
\xi_{20}(1) Y_{20}(\theta_1,\theta_2)=|\xi_{20}(1)| \frac14\sqrt{\frac{5}{\pi}}\ (1-3\cos^2\theta_1)
\eeqn
Of course, the poles correspond to $\theta_1=0,\pi$ and the equator  to $\theta_1=\frac{\pi}{2}$. This means $\xi$ is larger at the equator than at the poles. In view of \eqref{eq: zeta near zero} and \eqref{eq: xi as deriv}, we have proved
\begin{theorem}
If the conditions of Theorem \ref{main Euler theorem} are satisfied, and the angular velocity profile $\omega(r)$ is constant, then for $\kappa$ sufficiently small, the support of $\rho_{\kappa}$ is larger at the equator than at the poles.
\end{theorem}

\section{Euler Model: Fr\'{e}chet differentiability}\label{sec: 5}

In this section, we prove the Fr\'{e}chet differentiability of $\F(\zeta,\kappa)$ under the assumptions of Theorem \ref{main Euler theorem}.

\begin{theorem} \label{Frechet theorem}
The operator $\F: B_{\ep}(X)\times \real \to X$ with $\ep>0$ sufficiently small is continuously Fr\'echet differentiable with $\frac{\pa \F}{\pa \zeta}$ given by  \eqref{eq: F' formula} -   \eqref{eq:  F3'}. 
\end{theorem}

We will prove this theorem step by step by means of the following lemmas.  We write $\F$ as
\begin{equation} \label{F terms}
\F(\zeta,\kappa)=\M(\zeta)\F_1(\zeta)+\kappa\F_2(\zeta) + \F_3(\zeta)
\end{equation}
where
\begin{align} \label{F1 term}
\F_1(\zeta)(x) 
&= \int_{g_{\zeta}(B_1)}\rho_0(g_{\zeta}^{-1}(y'))\left(\frac{1}{|g_{\zeta}(x)-y'|}-\frac{1}{|y'|}\right)~dy'  
\\  \notag 
&=\int_{B_2}\rho_0(g_{\zeta}^{-1}(y'))\left(\frac{1}{|g_{\zeta}(x)-y'|}-\frac{1}{|y'|}\right)~dy'.\label{def: F1}
\end{align}
The last equality follows from the fact that $\rho_0$ is supported on $\overline{B_1}$ and that $g_{\zeta}(\overline{B_1})\subset B_2$.
\begin{equation}\label{F2 term}
\F_2(\zeta) (x)= \int_0^{r(\gz(x))}\omega^2(s)s~ds, 
\qquad r(\gz(x)) = \left(1+\frac{\zeta(x)}{|x|^2}\right)\sqrt{x_1^2+x_2^2},
\end{equation}
and
\begin{equation}\label{F3 term}
\F_3(\zeta)(x) = -h\left(\M(\zeta)\rho_0(x)\right)+h\left(\M(\zeta)\rho_0(0)\right).
\end{equation}

First of all, we want $\F$ to map $B_{\epsilon}(X)\times \mathbb{R}$ into $X$. It is easy to see that $\F(\zeta,\omega)$ has the symmetry requirements of $X$. 

\begin{lemma}\label{prop: est F}
There is a constant $C>0$ depending on $\rho_0$ and $\ep$ such that 
\begin{equation}
\|\F(\zeta,\kappa)\|_X\le C(1+\kappa)
\end{equation}
if $\zeta\in B_{\epsilon}(X)$ for $\epsilon$ small enough.
\end{lemma}
\begin{proof} 
To estimate the norm, we calculate the spatial derivatives of $\F(\zeta,\kappa)$.  
\begin{equation}
\partial_i \F(\zeta,\kappa)(x) = \M(\zeta)\partial_i \F_1(\zeta)(x) 
+ \kappa \partial_i \F_2(\zeta)(x) + \partial_i \F_3(\zeta)(x).
\end{equation}
We consider the three terms separately.  Note that 
\begin{equation}
|\M(\zeta)|\le \frac{1}{1-C\|\zeta\|_{X}} \le 2 
\end{equation}
if $\|\zeta\|_X<\epsilon$ is small enough, 
by \eqref{def: script M} and \eqref{est: det Dg}.
By  \Cref{lem: Lip estimate at zero}, \eqref{est: Dg} and \eqref{est: g inv-x},
\begin{equation}\label{eq: diff F_1}
\partial_i \F_1(\zeta)(x)=\left(\int_{B_2}\rho_0(g_{\zeta}^{-1}(y'))\frac{-(g_{\zeta}(x)-y')}{|g_{\zeta}(x)-y'|^3}~dy' \right)\cdot\partial_i g_{\zeta}(x) 
\end{equation}
satisfies 
\begin{align}\label{est: partial_i F_1}
|\partial_i \F_1(\zeta)(x)| 
&= |V_{\rho_0(g_{\zeta}^{-1}(\cdot))}(g_{\zeta}(x))\cdot \partial_ig_{\zeta}(x)| \notag \\
&\le C(\|\rho_0\|_{C^1(B_{1/2})}+\|\rho_0\|_{C^0(B_1)})(1+\|\zeta\|_X)|x|
\end{align}
for all $x\in \overline{B_1}$.
					Next, notice from  \eqref{est: zeta/x^2} that 
\begin{equation}\label{eq: d_i gz er}
\partial_i [r(\gz(x))]= \partial_i\left(\frac{\zeta(x)}{|x|^2}\right)\sqrt{x_1^2+x_2^2}+\left(1+\frac{\zeta(x)}{|x|^2}\right)\frac{x_1\delta_{i1}+x_2\delta_{i2}}{\sqrt{x_1^2+x_2^2}},
\end{equation}
\begin{equation}\label{est: gz er}
|r(\gz(x))|\le C|x|, \qquad |\partial_i [r(\gz(x))]|\le C. 
\end{equation}
					Hence 
\begin{equation}\label{est: d_i gz er}
|\partial_i \F_2(\zeta)(x)| = 
 \omega^2[r(\gz(x))]\ r(\gz(x))\ |\partial_i[r(\gz(x))] |  \le C|x|
\end{equation}
since $\omega^2$ is locally bounded.
To estimate $\F_3$, we recall that by \eqref{eq: partial rho partial U},
we have
\eqn \label{eq: partial_i F_3 1}
\partial_i \F_3(\zeta)(x)= -h'(\M(\zeta)\rho_0(x))\M(\zeta)\partial_i\rho_0(x)   
= -\M(\zeta)\frac{h'(\M(\zeta)\rho_0(x))}{h'(\rho_0(x))}\partial_i u_0(x). 
\eeqn 
By \eqref{eq: lim h}, we have 
\begin{equation}\label{eq: lim h' scaling}
\lim_{s\to 0^+}\frac{h'(\M s)}{h'(s)}=\M^{\gamma-2}.
\end{equation}
Together with the regularity of $u_0$ given in Lemma \ref{lem: u_0}, this implies that \eqref{eq: partial_i F_3 1} is continuous on $\overline{B_1}$. Furthermore it is easy to see that
$|\partial_i \F_3(\zeta)(x) |\le C|x|   $
because  $u_0\in C^2(\overline{B_1})$ and $\partial_i u_0(0)=0$.
This completes the proof.      
\end{proof}

We now compute the {\it formal} derivative of $\F$ with respect to $\zeta$, which we denote by $\F'(\zeta,\kappa)$. In order to facilitate future estimates of the formal derivative, we introduce the new notation 
\begin{equation}  \label {F(x,s)}
F(x,s) = \F(\zeta+s\xi,\kappa)(x)
\end{equation}
where $\zeta\in B_{\epsilon}(X)$, $\xi\in X$ are chosen, and $s$ is restricted to a sufficiently small neighborhood of $0$ so that $\zeta+s\xi\in B_{\epsilon}(X)$.   
We define the {\it formal derivative} of $\F$ with respect to $\zeta$ as the pointwise limit 
\begin{equation}
[\F'(\zeta,\kappa)\xi](x) = (\partial_sF)(x,0) = \partial_s\bigg|_{s=0}\F(\zeta+s\xi,\kappa)(x).
\end{equation}
for every {\it fixed} $x$. 

\begin{lemma}\label{lem: gateaux diff} 
The formal derivative is 
\begin{equation}\label{eq: F' formula}
\F'(\zeta,\kappa)\xi = [\M'(\zeta)\xi] \F_1(\zeta) + \M(\zeta)\F_1'(\zeta)\xi + \kappa\F_2'(\zeta)\xi + \F_3'(\zeta)\xi,
\end{equation}
where $\M'(\zeta)\xi =  $  
\eqn   \label{eq: M'}
~\frac{-M}{\left(\int_{B_1}\rho_0(x) \det Dg_{\zeta}(x)~dx\right)^2}  
\int_{B_1}\rho_0(x)\det Dg_{\zeta}(x)   
\tr \left[(Dg_{\zeta})^{-1}(x) D\left(\xi(x)\frac{x}{|x|^2}\right)\right]~dx,
\eeqn 
and   $[\F_1'(\zeta)\xi] (x)  =  $ 
\begin{align}     \label{eq: F1' 0}
&~ \int_{B_2}\bigg[-\nabla\rho_0(g_{\zeta}^{-1}(y'))  
(Dg_{\zeta})^{-1}(y')\xi(g_{\zeta}^{-1}(y'))\frac{g_{\zeta}^{-1}(y')}{|g_{\zeta}^{-1}(y')|^2}  
\left(\frac{1}{|g_{\zeta}(x)-y'|}-\frac{1}{|y'|}\right)       \notag\\
& \qquad \qquad +\rho_0(g_{\zeta}^{-1}(y'))\frac{-(g_{\zeta}(x)-y')}{|g_{\zeta}(x)-y'|^3}\cdot \xi(x)\frac{x}{|x|^2}\bigg]~dy',
\end{align}
\begin{equation}
[\F_2'(\zeta)\xi](x) = \omega^2[r(\gz(x))]\ r(\gz(x))\ \frac{\xi(x)}{|x|^2}\sqrt{x_1^2+x_2^2}.
\end{equation}
\begin{equation}\label{eq: F3'}
[\F_3'(\zeta)\xi](x) = \left[-h'(\M(\zeta)\rho_0(x))\rho_0(x)+h'(\M(\zeta)\rho_0(0))\rho_0(0)\right]\ \M'(\zeta)\xi.
\end{equation}
\end{lemma}
\begin{proof}
Note that $(Dg_\zeta)^{-1}$ denote the inverse matrix of $Dg_\zeta$.  
\eqref{eq: F' formula} is obvious.   We compute
\begin{align}
&~\M'(\zeta)\xi   
= ~\partial_s\bigg|_{s=0}\M(\zeta+s\xi)  \\ 
= &~\frac{-M}{\left(\int_{B_1}\rho_0(x) \det Dg_{\zeta}(x)~dx\right)^2}  
\partial_s\bigg|_{s=0}\int_{B_1}\rho_0(x) \det Dg_{\zeta+s\xi}(x)~dx\notag\\
= &~\frac{-M}{\left(\int_{B_1}\rho_0(x) \det Dg_{\zeta}(x)~dx\right)^2}  
\int_{B_1}\rho_0(x) \det D\gz(x) \tr \left[D\gz^{-1}\partial_s  
\bigg|_{s=0}Dg_{\zeta+s\xi}\right](x)~dx\notag\\
= &~\frac{-M}{\left(\int_{B_1}\rho_0(x) \det Dg_{\zeta}(x)~dx\right)^2}  
\int_{B_1}\rho_0(x)\det Dg_{\zeta}(x)  
\tr \left[Dg_{\zeta}^{-1}(x) D\left(\xi(x)\frac{x}{|x|^2}\right)\right]~dx.\notag\\
\end{align}
Differentiation under the integral sign is justified by   dominated convergence.  
We use 
\begin{equation}
\frac{d(\det F)}{ds}  =  (\det F) \tr \left(F^{-1} \frac{dF}{ds}\right), 
\qquad  \frac d{ds} Dg_{\zeta+s\xi}  =  D\left(\xi(x)\frac x{|x|^2} \right)  
\end{equation}
to obtain \eqref{eq: M'}. 
						Next
\eqn   \label{eq: F1' 1}
~[\F_1'(\zeta)\xi](x)  
= ~\partial_s\bigg|_{s=0}\int_{B_2}\rho_0(g_{\zeta+s\xi}^{-1}(y'))  
\left(\frac{1}{|g_{\zeta+s\xi}(x)-y'|}-\frac{1}{|y'|}\right)~dy'.
\eeqn  
To carry out the derivative  we split the integral into two terms 
\begin{equation}
I_1(s) = \int_{B_2}\rho_0(g_{\zeta+s\xi}^{-1}(y'))\frac{1}{|g_{\zeta+s\xi}(x)-y'|}~dy', 
\quad I_2(s) = \int_{B_2}\rho_0(g_{\zeta+s\xi}^{-1}(y'))\frac{1}{|y'|}~dy' , 
\end{equation}
					where 
$s$ belongs to a small neighborhood of 0. 
We use a cutoff function to avoid the singularity. Let $\chi$ be a non-negative, 
smooth, compactly supported function on the real line such that 
$\chi(s) = 1$ for $|s|<1$, $\chi(s)=0$ for $|s|>2$, and $\|\chi'\|_{\infty}\le 2$. Let 
\begin{equation}
I_{1,\epsilon}(s) = \int_{B_2}\rho_0(g_{\zeta+s\xi}^{-1}(y'))\frac{1}{|g_{\zeta+s\xi}(x)-y'|}\left[1-\chi\left(\frac{|g_{\zeta+s\xi}(x)-y'|}{\epsilon}\right)\right]~dy'
\end{equation}
					We see easily 
that $I_{1,\epsilon}(s)$ converges uniformly to $I_{1}(s)$ as $\epsilon \to 0$, and that $I_{1,\epsilon}(s)$ is $C^1$ with
\begin{align}
I_{1,\epsilon}'(s) = &~\int_{B_2}\partial_s\left(\rho_0(g_{\zeta+s\xi}^{-1}(y'))\frac{1}{|g_{\zeta+s\xi}(x)-y'|}\right)\left[1-\chi\left(\frac{|g_{\zeta+s\xi}(x)-y'|}{\epsilon}\right)\right]~dy'\notag\\
&~-\int_{B_2}\rho_0(g_{\zeta+s\xi}^{-1}(y'))\frac{1}{|g_{\zeta+s\xi}(x)-y'|}\partial_s\chi\left(\frac{|g_{\zeta+s\xi}(x)-y'|}{\epsilon}\right)~dy' .
\end{align}
Now
\begin{align}\label{est: 1/r^2}
&~\bigg|\partial_s\left(\rho_0(g_{\zeta+s\xi}^{-1}(y'))\frac{1}{|g_{\zeta+s\xi}(x)-y'|}\right)\bigg|    \\
= &~\bigg|-\nabla\rho_0(g_{\zeta+s\xi}^{-1}(y'))Dg_{\zeta+s\xi}^{-1}(y')\xi(g_{\zeta+s\xi}^{-1}(y'))\frac{g_{\zeta+s\xi}^{-1}(y')}{|g_{\zeta+s\xi}^{-1}(y')|^2}\frac{1}{|g_{\zeta+s\xi}(x)-y'|}\notag\\
&~+\rho_0(g_{\zeta+s\xi}^{-1}(y'))\frac{-(g_{\zeta+s\xi}(x)-y')}{|g_{\zeta+s\xi}(x)-y'|^3}\cdot \xi(x)\frac{x}{|x|^2}\bigg|  
\le ~C\frac{1}{|g_{\zeta+s\xi}(x)-y'|^2}.  \notag 
\end{align}
In the above calculation, we have used the formula
\eqn
\pa_s g_{\zeta+s\xi}^{-1}(y') = -Dg_{\zeta+s\xi}^{-1}(y')\xi(g_{\zeta+s\xi}^{-1}(y'))\frac{g_{\zeta+s\xi}^{-1}(y')}{|g_{\zeta+s\xi}^{-1}(y')|^2},
\eeqn					 
which follows if we differentiate $g_{\zeta+s\xi}(g_{\zeta+s\xi}^{-1}(y'))=y'$ with respect to $s$. In addition,
\begin{align}\label{est: 1/r}
\bigg|\rho_0(g_{\zeta+s\xi}^{-1}(y'))\frac{1}{|g_{\zeta+s\xi}(x)-y'|}\bigg| 
\le C\frac{1}{|g_{\zeta+s\xi}(x)-y'|}.  
\end{align}
					Also 
\begin{align}\label{est: 1/epsilon}
\bigg|\partial_s\chi\left(\frac{|g_{\zeta+s\xi}(x)-y'|}{\epsilon}\right)\bigg| &= \bigg|\chi'\left(\frac{|g_{\zeta+s\xi}(x)-y'|}{\epsilon}\right)\frac{g_{\zeta+s\xi}(x)-y'}{\epsilon|g_{\zeta+s\xi}(x)-y'|}\cdot \xi(x)\frac{x}{|x|^2}\bigg|\notag\\
& \le \frac{C}{\epsilon}  \bigg|\chi'\left(\frac{|g_{\zeta+s\xi}(x)-y'|}{\epsilon}\right)\bigg|,
\end{align}
					so that  
that $I_{1,\epsilon}'(s)$ converges uniformly to 
\begin{equation}\label{eq: F_1' 2}
\int_{B_2}\partial_s\left(\rho_0(g_{\zeta+s\xi}^{-1}(y'))\frac{1}{|g_{\zeta+s\xi}(x)-y'|}\right)~dy'
\end{equation}
as $\epsilon \to 0$.
Therefore $I_1(s)$ is $C^1$ with derivative given by \eqref{eq: F_1' 2}. 
The calculation of $I_2'(s)$ as well as of  $\F_2'$ and $\F_3'$ is similar.   
\end{proof}


Before we prove that the formal derivative $\F'(\zeta,\kappa)$ computed in \Cref{lem: gateaux diff} is really a Fr\'{e}chet derivative, we will show that it is a bounded linear map on $X$, and that the dependence of $\F'$ on $\zeta$ is continuous. To that end, we compute the spatial derivatives of $\F'(\zeta,\kappa)\xi$. 

\begin{lemma}
\begin{align}\label{eq: partial_i F'}
&~\partial_i[\F'(\zeta,\kappa)\xi](x) \notag\\
= &~[\M'(\zeta)\xi]\partial_i[\F_1(\zeta)](x) +\M(\zeta)\partial_i[\F_1'(\zeta)\xi](x) + \kappa\partial_i[\F_2'(\zeta)\xi](x) + \partial_i [\F_3'(\zeta)\xi](x),
\end{align}
					where $\partial_i [\F_1(\zeta)] (x)$ is given  in \eqref{eq: diff F_1},  
\begin{align}\label{eq: d_i of F_1'}
&~\partial_i [\F_1'(\zeta)\xi](x)  \\
= &~\int_{B_2} -\left(\nabla\rho_0(g_{\zeta}^{-1}(y'))Dg_{\zeta}^{-1}(y') 
\xi(g_{\zeta}^{-1}(y'))\frac{g_{\zeta}^{-1}(y')}{|g_{\zeta}^{-1}(y')|^2}\right)  
\frac{-(g_{\zeta}(x)-y')}{|g_{\zeta}(x)-y'|^3}~dy'\cdot \partial_i g_{\zeta}(x) \notag\\
&~+\sum_j \int_{B_2}\left(\nabla\rho_0(g_{\zeta}^{-1}(y'))\partial_j g_{\zeta}^{-1}(y')\right)  
\left(\frac{-(g_{\zeta}(x)-y')}{|g_{\zeta}(x)-y'|^3}  
\cdot \partial_i g_{\zeta}(x)\right)~dy'~ \xi(x)\frac{x_j}{|x|^2}\notag\\
&~+ \int_{B_2}\rho_0(g_{\zeta}^{-1}(y'))\frac{-(g_{\zeta}(x)-y')}{|g_{\zeta}(x)-y'|^3} ~dy'  
\cdot \partial_i\left(\xi(x)\frac{x}{|x|^2}\right),  \notag
\end{align}

\begin{align}    \label{eq: d_i of F_2'}
&\partial_i [\F_2'(\zeta)\xi](x) = (\omega^2)'[r(\gz(x))]\ \partial_i[r(\gz(x))]\ r(\gz(x))\ 
\frac{\xi(x)}{|x|^2}\sqrt{x_1^2+x_2^2}\notag\\
&\quad +\omega^2(r(\gz(x)))\bigg[\partial_i[r(\gz(x))]\frac{\xi(x)}{|x|^2}\sqrt{x_1^2+x_2^2}  
  + [r(\gz(x))]\partial_i\left(\frac{\xi(x)}{|x|^2}\sqrt{x_1^2+x_2^2}\right)\bigg],
\end{align}

\begin{align}\label{eq: d_i of F_3'}
\partial_i [\F_3'(\zeta)\xi](x) = [-h''(\M(\zeta)\rho_0(x))\M(\zeta)\rho_0(x)\partial_i\rho_0(x)-h'(\M(\zeta)\rho_0(x))\partial_i\rho_0(x)]\M'(\zeta)\xi.
\end{align}

\end{lemma}
\begin{proof}
The $\partial_i$ derivative of the first term in \eqref{eq: F1' 0} can be computed in a way similar 
to \Cref{lem: gateaux diff} by using the cutoff function $\chi$.   
This gives the first line of \eqref{eq: d_i of F_1'}.  
We rewrite the second line of \eqref{eq: F1' 0} as 
\begin{align}
&~\int_{B_2}\rho_0(\gz^{-1}(y'))\frac{-(\gz(x)-y')}{|\gz(x)-y'|^3}~dy'\cdot\xi(x)\frac{x}{|x|^2}\notag\\
=&~ -\int_{B_2}\rho_0(\gz^{-1}(y'))\nabla_{y'}\left(\frac{1}{|\gz(x)-y'|}\right)~dy'\cdot\xi(x)\frac{x}{|x|^2}\notag\\
=&~\sum_j\int_{B_2}\left(\nabla\rho_0(g_{\zeta}^{-1}(y'))\partial_j g_{\zeta}^{-1}(y')\right)\frac{1}{|g_{\zeta}(x)-y'|}~dy'~\xi(x)\frac{x_j}{|x|^2}
\end{align}
Now we take the $\partial_i$ derivative using the cutoff method of \Cref{lem: gateaux diff}.  
The last two lines of \eqref{eq: d_i of F_1'} follow. The calculations of the other terms are straightforward.

\end{proof}


\begin{lemma}\label{lem: mixed partial F} 
With $F(x,s)$ defined in \eqref{F(x,s)}, the mixed derivatives are  
\begin{equation}\label{eq: mixed partial F}
\partial_s \partial_ i F(x,s) = \partial_i \partial_s F (x,s) = \partial_i [\F'(\zeta+s\xi,\kappa)\xi](x)
\end{equation}
for every $s$ in a neighborhood of $0$.
\end{lemma}
\begin{proof}
The proof of \Cref{lem: gateaux diff} actually shows
\begin{equation}
\partial_sF(x,s) = [\F'(\zeta+s\xi,\omega)\xi](x),
\end{equation}
from which the second equality of \eqref{eq: mixed partial F} follows immediately. The other equality is certainly true if $F$ is $C^2$, but such a regularity is not yet established. Instead, we directly compute
\begin{align}
~\partial_iF(x,s) 
=& ~\partial_i\F(\zeta+s\xi,\omega)(x)\notag\\
= &~\M(\zeta+s\xi)\partial_i\F_1(\zeta+s\xi)(x)  
+\partial_i\F_2(\zeta+s\xi,\omega)(x)+\partial_i\F_3(\zeta+s\xi)(x).
\end{align}
Thus
\begin{align}
~\partial_s\partial_iF(x,s)  
&= ~[\partial_s\M(\zeta+s\xi)]\partial_i\F_1(\zeta+s\xi)(x) + \M(\zeta+s\xi)[\partial_s\partial_i\F_1(\zeta+s\xi)(x)]\notag\\
&\ \ \ ~+\partial_s\partial_i\F_2(\zeta+s\xi,\omega)(x) + \partial_s\partial_i\F_3(\zeta+s\xi)(x).
\end{align}
We now compute $\partial_s\partial_i\F_1(\zeta+s\xi)(x)$. By \eqref{eq: diff F_1},
\begin{equation}\label{eq: partial F1 z+se}
\partial_i\F_1(\zeta+s\xi)(x) = \left(\int_{B_2}\rho_0(g_{\zeta+s\xi}^{-1}(y'))\frac{-(g_{\zeta+s\xi}(x)-y')}{|g_{\zeta+s\xi}(x)-y'|^3}~dy'\right)\cdot \partial_i g_{\zeta+s\xi}(x).
\end{equation}
We again use the cutoff function $\chi$ defined in \Cref{lem: gateaux diff}, and write for every $x$ and component $j$:
\begin{equation}\label{eq: def I(s)}
I(s) = \int_{B_2}\rho_0(g_{\zeta+s\xi}^{-1}(y'))\frac{-(g_{\zeta+s\xi}(x)-y')_j}{|g_{\zeta+s\xi}(x)-y'|^3}~dy',
\end{equation}
and
\begin{equation}
I_{\epsilon}(s) = \int_{B_2}\rho_0(g_{\zeta+s\xi}^{-1}(y'))\frac{-(g_{\zeta+s\xi}(x)-y')_j}{|g_{\zeta+s\xi}(x)-y'|^3}\left[1-\chi\left(\frac{|g_{\zeta+s\xi}(x)-y'|}{\epsilon}\right)\right]~dy'.
\end{equation}
For $s$ in suitable neighborhood of $0$, $I_{\epsilon}(s)$ converges uniformly to $I(s)$, whereas 
$I_{\epsilon}'(s)  =  $  
\begin{align}\label{eq: I_e'(s)}
 &~\int_{B_2}\left(\nabla\rho_0(g_{\zeta+s\xi}^{-1}(y'))\partial_s g_{\zeta+s\xi}^{-1}(y')\right)\frac{-(g_{\zeta+s\xi}(x)-y')_j}{|g_{\zeta+s\xi}(x)-y'|^3}\left[1-\chi\left(\frac{|g_{\zeta+s\xi}(x)-y'|}{\epsilon}\right)\right]~dy'\notag\\
&~-\int_{B_2}\rho_0(g_{\zeta+s\xi}^{-1}(y'))\left(\partial_j\partial_s\frac{1}{|g_{\zeta+s\xi}(x)-y'|}\right)\left[1-\chi\left(\frac{|g_{\zeta+s\xi}(x)-y'|}{\epsilon}\right)\right]~dy'\notag\\
&~-\int_{B_2}\rho_0(g_{\zeta+s\xi}^{-1}(y'))\left(\partial_j\frac{1}{|g_{\zeta+s\xi}(x)-y'|}\right)\chi'\left(\frac{|g_{\zeta+s\xi}(x)-y'|}{\epsilon}\right)\frac{1}{\epsilon}\partial_s|g_{\zeta+s\xi}(x)-y'|~dy'.
\end{align}
After integration by parts, the second term in \eqref{eq: I_e'(s)} equals  
\begin{align}\label{eq: I_e'(s) 1}
&~\int_{B_2}\left(\nabla\rho_0(g_{\zeta+s\xi}^{-1}(y'))\partial_j g_{\zeta+s\xi}^{-1}(y')\right)\left(\partial_s\frac{1}{|g_{\zeta+s\xi}(x)-y'|}\right)\left[1-\chi\left(\frac{|g_{\zeta+s\xi}(x)-y'|}{\epsilon}\right)\right]~dy'\notag\\
&~+\int_{B_2}\rho_0(g_{\zeta+s\xi}^{-1}(y'))\left(\partial_s\frac{1}{|g_{\zeta+s\xi}(x)-y'|}\right)\chi'\left(\frac{|g_{\zeta+s\xi}(x)-y'|}{\epsilon}\right)\frac{1}{\epsilon}\partial_j|g_{\zeta+s\xi}(x)-y'|~dy'.
\end{align}
					Noticing that 
\begin{equation}
\left(\partial_s\frac{1}{|g_{\zeta+s\xi}(x)-y'|}\right)\partial_j|g_{\zeta+s\xi}(x)-y'| = \left(\partial_j\frac{1}{|g_{\zeta+s\xi}(x)-y'|}\right)\partial_s|g_{\zeta+s\xi}(x)-y'|,
\end{equation}
we see that the last term in \eqref{eq: I_e'(s) 1} cancels the last term in \eqref{eq: I_e'(s)}. Therefore
\begin{align}\label{eq: I_e'(s) 2}
&~I_{\epsilon}'(s)\notag\\
= &~\int_{B_2}\left(\nabla\rho_0(g_{\zeta+s\xi}^{-1}(y'))\partial_s g_{\zeta+s\xi}^{-1}(y')\right)\frac{-(g_{\zeta+s\xi}(x)-y')_j}{|g_{\zeta+s\xi}(x)-y'|^3}\left[1-\chi\left(\frac{|g_{\zeta+s\xi}(x)-y'|}{\epsilon}\right)\right]~dy'\notag\\
&~+\int_{B_2}\left(\nabla\rho_0(g_{\zeta+s\xi}^{-1}(y'))\partial_j g_{\zeta+s\xi}^{-1}(y')\right)\left(\partial_s\frac{1}{|g_{\zeta+s\xi}(x)-y'|}\right)\left[1-\chi\left(\frac{|g_{\zeta+s\xi}(x)-y'|}{\epsilon}\right)\right]~dy'.
\end{align}
By \eqref{eq: I_e'(s) 2} $I_{\epsilon}'(s)$ converges uniformly to
\begin{align}\label{eq: I'(s)}
&~\int_{B_2}\left(\nabla\rho_0(g_{\zeta+s\xi}^{-1}(y'))\partial_s g_{\zeta+s\xi}^{-1}(y')\right)\frac{-(g_{\zeta+s\xi}(x)-y')_j}{|g_{\zeta+s\xi}(x)-y'|^3}~dy'\notag\\
&~+\int_{B_2}\left(\nabla\rho_0(g_{\zeta+s\xi}^{-1}(y'))\partial_j g_{\zeta+s\xi}^{-1}(y')\right)\left(\partial_s\frac{1}{|g_{\zeta+s\xi}(x)-y'|}\right)~dy'
\end{align}
as $\epsilon\to 0$.
Therefore $I(s)$ is $C^1$ and $I'(s)$ is given by \eqref{eq: I'(s)}. \eqref{eq: partial F1 z+se}, \eqref{eq: def I(s)}, and \eqref{eq: I'(s)} now give
\begin{align}
&~\partial_s\partial_i \F_1(\zeta+s\xi)(x)\notag\\
&= ~ \int_{B_2}\left(\nabla\rho_0(g_{\zeta+s\xi}^{-1}(y'))\partial_s g_{\zeta+s\xi}^{-1}(y')\right)\frac{-(g_{\zeta+s\xi}(x)-y')}{|g_{\zeta+s\xi}(x)-y'|^3}~dy'\cdot \partial_i g_{\zeta+s\xi}\notag\\
&~+\sum_j \int_{B_2}\left(\nabla\rho_0(g_{\zeta+s\xi}^{-1}(y'))\partial_j g_{\zeta+s\xi}^{-1}(y')\right)\left(\frac{-(g_{\zeta+s\xi}(x)-y')}{|g_{\zeta+s\xi}(x)-y'|^3}\cdot \partial_i g_{\zeta+s\xi}(x)\right)~dy' ~\xi(x)\frac{x_j}{|x|^2}\notag\\
&~+\left(\int_{B_2}\rho_0(g_{\zeta+s\xi}^{-1}(y'))\frac{-(g_{\zeta+s\xi}(x)-y')}{|g_{\zeta+s\xi}(x)-y'|^3}~dy'\right)\cdot \partial_i \left(\xi(x)\frac{x}{|x|^2}\right)\notag \\
&= ~\partial_i[\F_1'(\zeta+s\xi)\xi](x).
\end{align}
The calculations related to the other terms are straightforward and therefore omitted.
\end{proof}

Next we show that the formal derivative is indeed a bounded linear map on $X$.

\begin{lemma}\label{prop: est F'}
If $\zeta\in B_{\epsilon}(X)$ and $\epsilon$ is small enough, there is a constant $C$ such that
\begin{equation}
\|\F'(\zeta,\kappa)\xi\|_X\le C(1+\kappa)\|\xi\|_X.
\end{equation}
\end{lemma}
\begin{proof}
We estimate the terms in \eqref{eq: partial_i F'} one by one. First of all,
\begin{equation}
|\M'(\zeta)\xi|\le C\|\rho_0\|_1 \|\xi\|_X
\end{equation}
by \eqref{est: det Dg} and \eqref{est: Dg inv}. $\partial_i [\F_1(\zeta)](x)$ was already estimated in \eqref{est: partial_i F_1}.
					Next we estimate 
$\partial_i[\F_1'(\zeta)\xi](x)$ from \eqref{eq: d_i of F_1'}.   
We call the three lines in \eqref{eq: d_i of F_1'} $I_1$, $I_2$ and $I_3$ respectively. 
We apply \Cref{lem: Lip estimate at zero} to $I_1$ with
\begin{equation}
f(y')= f_1(y') = \left(\nabla\rho_0(g_{\zeta}^{-1}(y'))Dg_{\zeta}^{-1}(y')
\frac{g_{\zeta}^{-1}(y')}{|g_{\zeta}^{-1}(y')|^2}\right)\xi(g_{\zeta}^{-1}(y')).
\end{equation}
Note that $f_1$ is continuous  and
$|f_1(y')|\le C\|\nabla \rho_0\|_{\infty}\|\xi\|_X|\gz^{-1}(y')|\le C\|\xi\|_X|y'|.$  
Thus $|I_1|\le C\|\xi\|_X|x| $
for all $x\in \overline{B_1}$.
Next we estimate the integral in $I_2$ directly as 
\begin{align}\label{eq: I2 1}
&~\left|\int_{B_2}\left(\nabla\rho_0(g_{\zeta}^{-1}(y'))\partial_j g_{\zeta}^{-1}(y')\right)\left(\frac{-(g_{\zeta}(x)-y')}{|g_{\zeta}(x)-y'|^3}\cdot \partial_i g_{\zeta}(x)\right)~dy'\right|\notag\\
\le&~C\|\nabla\rho_0\|_{\infty}\int_{B_2}\frac{1}{|g_{\zeta}(x)-y'|^2}~dy'  
\le ~C\|\nabla\rho_0\|_{\infty}\int_{B_4}\frac{1}{|y'|^2}~dy'  
\le ~C_1
\end{align}
if $x\in \overline{B_1}$. 
					On the other hand, 
$ \frac{|\xi(x)x_j|}{|x|^2}\le \frac{|\xi(x)|}{|x|}\le C\|\xi\|_X|x| $
so that $|I_2|\le C\|\xi\|_X|x|.$
					Next we apply 
\Cref{lem: Lip estimate at zero} to $I_3$ with 
$ f(y') = f_3(y') = \rho_0(\gz^{-1}(y')).$ 
Note that
$ |f_3(y')-f_3(0)|\le \|\nabla \rho_0\|_{\infty}|\gz^{-1}(y')|\le C|y'|. $
Therefore 
\begin{equation}
\left|\int_{B_2}\rho_0(g_{\zeta}^{-1}(y'))\frac{-(g_{\zeta}(x)-y')}{|g_{\zeta}(x)-y'|^3} ~dy'\right|\le C|x|
\end{equation}
for all $x\in \overline{B_1}$. 
Together with the fact that 
$ \left|\partial_i\left(\xi(x)\frac{x}{|x|^2}\right)\right|\le C\|\xi\|_X $, 
this implies 
$  |I_3|\le C\|\xi\|_X|x|.  $ 
					Combining 
the estimates on $I_1$, $I_2$ and $I_3$, we get
\begin{equation}
|\partial_i [\F_1'(\zeta)\xi](x)|\le C\|\xi\|_X|x|.
\end{equation}

We next estimate $\partial_i[\F_2'(\zeta)\xi](x)$ from \eqref{eq: d_i of F_2'}. 
Due to  \eqref{est: gz er} and \eqref{est: d_i gz er}, 
the first line of \eqref{eq: d_i of F_2'} is bounded by $C\|\xi\|_X|x|$ 
since $(\omega^2)'$ is locally bounded. The second and third lines are bounded by $C\|\xi\|_X|x|$ 
since  $\omega^2$ is locally bounded. 
$\partial_i[\F_3'(\zeta,\kappa)\xi](x)$ is given in \eqref{eq: d_i of F_3'}. 
It can be rewritten as
\begin{equation}\label{eq: d_i of F_3' 1}
-(\M'(\zeta)\xi)\left[\M \frac{h''(\M\rho_0)\rho_0}{h'(\rho_0)} + \frac{h'(\M\rho_0)}{h'(\rho_0)}\right]\partial_i u_0.
\end{equation}
By  \eqref{eq: lim h''},
\begin{equation}\label{eq: lim h'' h' scaling}
\lim_{s\to 0^+}\frac{h''(\M s)s}{h'(s)} = (\gamma-2)\M^{\gamma-3}.
\end{equation}
\eqref{eq: lim h'' h' scaling} and \eqref{eq: lim h' scaling} imply that \eqref{eq: d_i of F_3' 1} is continuous on $\overline{B_1}$. Again using the fact that $\partial_i u_0(0)=0$ we get
\begin{equation}
|\partial_i[\F_3'(\zeta,\kappa)\xi](x)|\le C|\M'(\zeta)\xi||x|\le C\|\xi\|_X|x|.
\end{equation}
This completes all the estimates.
\end{proof}

The next proposition asserts that the formal derivative depends continuously on $\zeta$.

\begin{lemma}\label{prop: continuity of formal deriv}
Let $\alpha = \min\{\frac{2-\gamma}{\gamma-1}, \beta(\omega)\}$, where 
$\beta(\omega)$ is given in Assumption \eqref{assump omega}.  
If $\zeta_1,\zeta_2\in B_{\epsilon}(X)$ and $\epsilon$ is small enough, there is a constant C such that
\begin{equation}\label{est: continuity of formal deriv}
\|(\F'(\zeta_1,\kappa)-\F'(\zeta_2,\kappa))\xi\|_X\le C(1+\kappa)\|\zeta_1-\zeta_2\|_X^{\alpha}\|\xi\|_X.
\end{equation}
\end{lemma}
\begin{proof}

$$\F'(\zeta_1,\kappa)-\F(\zeta_2,\kappa) = I_1+I_2+I_3+I_4+I_5+I_6\ ,  $$
$$ I_1 = [\M'(\zeta_1)\xi-\M'(\zeta_2)\xi]\F_1(\zeta_1), 
\quad I_2= [\M'(\zeta_2)\xi](\F_1(\zeta_1)-\F_1(\zeta_2)), $$
$$I_3=(\M(\zeta_1)-\M(\zeta_2))\F_1'(\zeta_1)\xi, 
\quad I_4=\M(\zeta_2)[(\F_1'(\zeta_1)-\F_1'(\zeta_2))\xi], $$
$$I_5=\kappa(\F_2'(\zeta_1)-\F_2'(\zeta_2))\xi,  
\quad I_6=(\F_3'(\zeta_1)-\F_3'(\zeta_2))\xi .$$
We estimate the terms $I_1$ through $I_6$ one by one. 
Starting with $I_1$, the $X$ norm of $\F_1(\zeta_1)$ was already estimated in \Cref{prop: est F}. 
We just need to estimate the size of $\M'(\zeta_1)\xi-\M'(\zeta_2)\xi$. 
In \eqref{eq: M'}  the function $\zeta$ appears in three different places. 
Here and in the following discussion 
we take the difference of the occurrences of $\zeta$ one at a time. 
For $\M'(\zeta)\xi$ the key differences to estimate are
$\det Dg_{\zeta_1}(x)-\det Dg_{\zeta_2}(x), Dg_{\zeta_1}^{-1}(x)-Dg_{\zeta_2}^{-1}(x)$ and 
\begin{equation}\label{eq: I11}
{\left(\int_{B_1}\rho_0(x) \det Dg_{\zeta_1}(x)~dx\right)^{-2}}
-   {\left(\int_{B_1}\rho_0(x) \det Dg_{\zeta_2}(x)~dx\right)^{-2}}.
\end{equation}
					We observe that 
all three differences are bounded by $C\|\zeta_1-\zeta_2\|_X$, either by directly using \eqref{est: Dg inv 12}, \eqref{est: det Dg 12}, or by combining them with a simple application of the mean value theorem (in the case of \eqref{eq: I11}).  Thus
\begin{equation}
\|I_1\|_{X}\le C\|\zeta_1-\zeta_2\|_X\|\xi\|_X.
\end{equation}

We now turn our attention to $I_2$.   
Again $\M'(\zeta_2)\xi$ was already estimated in \Cref{prop: est F'} and satisfies 
$\|\M'(\zeta_2)\xi\|_X\le C\|\xi\|_X.$ 
In order to estimate the $X$ norm of $\F_1(\zeta_1)-\F_1(\zeta_2)$, we use \eqref{eq: diff F_1}, 
which we break up into the three parts 
\begin{equation}
I_{21} = \left(\int_{B_2}\left(\rho_0(\gzone^{-1}(y'))-\rho_0(\gztwo^{-1}(y'))\right)\frac{-(g_{\zeta_j}(x)-y')}{|g_{\zeta_j}(x)-y'|^3}~dy'\right)\cdot \partial_i g_{\zeta_k}(x),
\end{equation}

\begin{equation}
I_{22} = \left(\int_{B_2}\rho_0(g_{\zeta_j}^{-1}(y'))\left(\frac{-(\gzone(x)-y')}{|\gzone(x)-y'|^3}-\frac{-(\gztwo(x)-y')}{|\gztwo(x)-y'|^3}\right)~dy'\right)\cdot \partial_i g_{\zeta_k}(x)
\end{equation}
and
\begin{equation}
I_{23} = \left(\int_{B_2}\rho_0(g_{\zeta_j}^{-1}(y'))\frac{-(g_{\zeta_k}(x)-y')}{|g_{\zeta_k}(x)-y'|^3}~dy'\right)\cdot \partial_i (g_{\zeta_1}(x)-g_{\zeta_2}(x)).
\end{equation}
					Here $j,k$ can be 1 or 2.   
To estimate $I_{21}$, we apply \Cref{lem: Lip estimate at zero} to 
$f(y') = f_1(y') = \rho_0(\gzone^{-1}(y'))-\rho_0(\gztwo^{-1}(y')).$ 
Note that $f_1$ is continuous, $f_1(0)=0$ and 
\begin{equation}
|f_1(x)|\le \|\nabla\rho_0\|_{\infty}|\gzone^{-1}(y')-\gztwo^{-1}(y')|\le C\|\zeta_1-\zeta_2\|_X|y'|
\end{equation}
by \eqref{est: g inv-x 12}. Therefore 
$|I_{21}|\le C\|\zeta_1-\zeta_2\|_X|x|. $ 
To estimate $I_{22}$ we apply \Cref{lem: Lip estimate at zero} to 
$ f(y') = f_2(y') = \rho_0(g_{\zeta_j}^{-1}(y')),$  
which is a continuous function that satisfies 
\begin{equation}
|f_2(y')-f_2(0)|\le \|\nabla\rho_0\|_{\infty}|g_{\zeta_j}^{-1}(y')|\le C|y'|.
\end{equation}
Therefore 
$|I_{22}|\le C\|\zeta_1-\zeta_2\|_X^{\beta}|x|. $ 
Finally,  
$|I_{23}|\le C\|\zeta_1-\zeta_2\|_X|x| $
by a similar application of \Cref{lem: Lip estimate at zero} and \eqref{est: g-x 12}.
Thus 
$ \|I_2\|_X\le C\|\zeta_1-\zeta_2\|_X^{\beta}\|\xi\|_X. $ 
The estimation of $I_3$ is similar to that of $I_1$ and is omitted. 

We focus our attention now on $I_4=\M(\zeta_2)[(\F_1'(\zeta_1)-\F_1'(\zeta_2))\xi]$. 
We need to estimate $\|(\F_1'(\zeta_1)-\F_1'(\zeta_2))\xi\|_X$ starting from \eqref{eq: d_i of F_1'}.  
This is a long equation with three lines, hence $\partial_i[(\F_1'(\zeta_1)-\F_1'(\zeta_2))\xi](x)$ 
can also be broken into three lines, which we denote by $I_{41}$, $I_{42}$ and $I_{43}$ respectively. 
For the first line we must estimate 
\begin{equation}
I_{41a} = \int_{B_2} \left(\G(\zeta_1)-\G(\zeta_2)\right)\frac{-(g_{\zeta_j}(x)-y')}{|g_{\zeta_j}(x)-y'|^3}~dy'\cdot \partial_i g_{\zeta_k}(x),
\end{equation}

\begin{equation}
I_{41b} = \int_{B_2}\G(\zeta_j)\left(\frac{-(g_{\zeta_1}(x)-y')}{|g_{\zeta_1}(x)-y'|^3}-\frac{-(g_{\zeta_2}(x)-y')}{|g_{\zeta_2}(x)-y'|^3}\right)~dy'\cdot \partial_i g_{\zeta_k}(x),
\end{equation}

\begin{equation}
I_{41c} = \int_{B_2} \G(\zeta_j)\frac{-(g_{\zeta_k}(x)-y')}{|g_{\zeta_k}(x)-y'|^3}~dy'\cdot \partial_i (g_{\zeta_1}(x)-g_{\zeta_2}(x)),
\end{equation}
$(j,k=1,2)$, 					where for brevity we denote 
\begin{equation}
\G(\zeta) = \nabla\rho_0(g_{\zeta}^{-1}(y'))Dg_{\zeta}^{-1}(y')\xi(g_{\zeta}^{-1}(y'))  
\frac{g_{\zeta}^{-1}(y')}{|g_{\zeta}^{-1}(y')|^2}. 
\end{equation}
 							We again apply 
 \Cref{lem: Lip estimate at zero} to these integrals. For $I_{41a}$, the function $f(y')= \G(\zeta_1)-\G(\zeta_2)$ is continuous and satisfies 
\begin{align}
&~|f(y')| 
\le ~C\|\rho_0\|_{C^{1,\alpha}}|\gzone^{-1}(y')-\gztwo^{-1}(y')|^{\alpha}\|\xi\|_X|y'|\notag \\
&~+ C\|\nabla \rho_0\|_{\infty}|D\gzone^{-1}(y')-D\gztwo^{-1}(y')|\|\xi\|_X|y'|  
~+ C\|\nabla \rho_0\|_{\infty}\|\nabla \xi\|_{\infty}|\gzone^{-1}(y')-\gztwo^{-1}(y')|\frac{1}{|g_{\zeta_j}(y')|}\notag \\
&~+ C\|\nabla \rho_0\|_{\infty}\|\xi\|_X|y'|^2\frac{1}{|\theta \gzone^{-1}(y')+(1-\theta)\gztwo^{-1}(y')|^2}|\gzone^{-1}(y')-\gztwo^{-1}(y')|\notag \\
\le &~ C\|\rho_0\|_{C^{1,\alpha}}\|\zeta_1-\zeta_2\|_X^{\alpha}\|\xi\|_{X}|y'|.  \notag 
\end{align}
					\Cref{lem: Lip estimate at zero} now implies 
$ |I_{41a}|\le C\|\zeta_1-\zeta_2\|_X^{\alpha}\|\xi\|_{X}|x|. $ 
The terms $I_{41b}$ and $I_{41c}$ are estimated similarly.  
As for $I_{42}$, we write 
\begin{equation}
I_{42a}=\int_{B_2}(\G(\zeta_1)-\G(\zeta_2))\left(\frac{-(g_{\zeta_k}(x)-y')}{|g_{\zeta_k}(x)-y'|^3}\cdot\partial_i g_{\zeta_l}(x)\right)~dy',
\end{equation}
\begin{equation}
I_{42b}=\int_{B_2}\G(\zeta_k)\left(\frac{-(g_{\zeta_1}(x)-y')}{|g_{\zeta_1}(x)-y'|^3}-\frac{-(g_{\zeta_2}(x)-y')}{|g_{\zeta_2}(x)-y'|^3}\right)\cdot\partial_i g_{\zeta_l}(x)~dy',
\end{equation}
\begin{equation}
I_{42c}=\int_{B_2}\G(\zeta_k)\frac{-(g_{\zeta_l}(x)-y')}{|g_{\zeta_l}(x)-y'|^3}\cdot(\partial_i g_{\zeta_1}(x)-\partial_i g_{\zeta_2}(x))~dy,  
\end{equation}
					where now 
$\G(\zeta) = \nabla\rho_0(\gz^{-1}(y'))\partial_j\gz^{-1}(y'). $  
 These terms are estimated in a similar fashion, as are  
the estimates on $I_{43}$.   Thus we get
\begin{equation}
\|I_4\|_X\le C\|\zeta_1-\zeta_2\|_{X}^{\alpha}\|\xi\|_X.
\end{equation}

To estimate $I_5$, we use \eqref{eq: d_i of F_2'}.   We only need
\begin{equation}
(\omega^2)'[r(\gzone(x))]-(\omega^2)'[r(\gztwo(x))]  
\quad \text{ and  }\quad  (\omega^2)[r(\gzone(x))]-(\omega^2)[r(\gztwo(x))]
\end{equation}
to be bounded by $C\|\zeta_1-\zeta_2\|_X^{\beta}$, which is true 
since  $\omega^2$ is locally $C^{1,\beta}$ on $[0,\infty)$.
Thus $$|\partial_i [\F_2'(\zeta_1)\xi-\F_2'(\zeta_2)\xi](x)|   \le  
C\|\zeta_1-\zeta_2\|_X^{\alpha}\|\xi\|_X|x|  $$ 
so that 
$ \|I_5\|_X\le C\kappa\|\zeta_1-\zeta_2\|_X^{\alpha}\|\xi\|_X.  $

						Finally 
to estimate $I_6$, we calculate $\partial_i [(\F_3'(\zeta_1)-\F_3'(\zeta_2))\xi]$ 
using \eqref{eq: d_i of F_3' 1} and  
splitting it into several pieces each involving a single difference as before.  
For instance, one piece is 
\begin{equation}\label{eq: I63}
I_{63} = [(\M_j')\xi]\left[\M_k \frac{[h''(\M_1\rho_0)-h''(\M_2\rho_0)]\rho_0}{h'(\rho_0)} \right]\partial_i u_0 , 
\end{equation}
where we use the shorthand $\M_j = \M(\zeta_j)$ with $j,k,l=1,2$. 
We write 
\begin{align}\label{est: I63 1}
\frac{[h''(\M_1\rho_0)-h''(\M_2\rho_0)]\rho_0}{h'(\rho_0)} = (M_1-M_2)\frac{h'''(\overline{\M}\rho_0)\rho_0^2}{h'(\rho_0)},
\end{align}
where $\overline{\M}$ is between $\M_1$ and $\M_2$. By \eqref{eq: lim h''},
$
\frac{h'''(\overline{\M}s)s^2}{h'(s)}
$
is bounded if $\overline{\M}$ is close to 1 and $s$ is close to zero.   
So the expression in \eqref{est: I63 1} is bounded by 
$C|\M_1-\M_2|\le C\|\zeta_1-\zeta_2\|_X  $  
and we get 
$ |I_{63}|\le C\|\zeta_1-\zeta_2\|_X\|\xi\|_X|x|. $ 
Therefore 
$\|I_6\|_X\le C\|\zeta_1-\zeta_2\|_X\|\xi\|_X.  $ 
This completes all the estimates needed to establish \eqref{est: continuity of formal deriv}.
\end{proof}

We are finally ready to show that the formal derivative is a genuine Fr\'{e}chet derivative.
\begin{lemma}     \label{prop: Frechet}
Let $\zeta\in B_{\epsilon}(X)$, and $\xi$ be such that $\zeta+s\xi\in B_{\epsilon}(X)$ for all $s\in [-1,1]$. Then
\begin{equation}\label{est: Frechet}
\|\F(\zeta+\xi,\kappa)-\F(\zeta,\kappa)-\F'(\zeta,\kappa)\xi\|_X\le C(1+\kappa)\|\xi\|_{X}^{1+\alpha},
\end{equation}
where  $\F'(\zeta,\kappa)\xi$ denotes the formal derivative defined in \eqref{eq: F' formula}.
\end{lemma}
\begin{proof}
To estimate the $X$ norm in \eqref{est: Frechet}, we write 
\begin{align}
&\partial_i(\F(\zeta+\xi,\kappa)-\F(\zeta,\kappa))(x) = \partial_iF(x,1)-\partial_iF(x,0)\notag\\
&= \partial_s\partial_i F(x,\theta)  = \partial_i[\F'(\zeta+\theta(x)\xi,\kappa)\xi](x)
\end{align}
for some $\theta(x)\in (0,1)$ by  \Cref{lem: mixed partial F}. 
Thus 
\begin{align}\label{est: Frechet 1}
&~\left|\partial_i(\F(\zeta+\xi,\kappa)-\F(\zeta,\kappa))(x) - \partial_i[\F'(\zeta,\kappa)\xi](x)\right|\notag\\
= &~\left|\partial_i[\F'(\zeta+s\xi,\kappa)\xi](x)-\partial_i[\F'(\zeta,\kappa)\xi](x)\right|\bigg|_{s=\theta(x)}  
\le ~C(1+\kappa)\|\xi\|_{X}^{1+\alpha}|x|  
\end{align}
by  \Cref{prop: continuity of formal deriv}, as desired.  
\end{proof}

\Cref{prop: Frechet} means that $\zeta\mapsto\F(\zeta,\kappa)$ is $C^1$ Fr\'{e}chet differentiable with Fr\'{e}chet derivative given by $\F'(\zeta,\kappa)$. In fact, $\F$ is jointly Fr\'{e}chet differentiable in both variables as the derivative with respect to $\kappa$ is very simple.

\begin{lemma}
$\F:B_{\epsilon}(X)\times \real \to X$ is continuously Fr\'{e}chet differentiable with Fr\'{e}chet derivative at $(\zeta,\kappa)$ given by
\begin{equation}\label{eq: joint Frechet}
(\xi, \upsilon) \mapsto \F'(\zeta,\kappa)\xi +  \upsilon \F_2(\zeta).
\end{equation}
\end{lemma}
\begin{proof}
By \Cref{prop: continuity of formal deriv} and \Cref{prop: Frechet},
\begin{align}
&~\|\F(\zeta+\xi,\kappa+\upsilon)-\F(\zeta,\kappa)-\F'(\zeta,\kappa)\xi - \upsilon \F_2(\zeta)\|_X\notag\\
\le &~\|\F(\zeta+\xi,\kappa+\upsilon)-\F(\zeta+\xi,\kappa)-\upsilon \F_2(\zeta)\|_X+\|\F(\zeta+\xi,\kappa)-\F(\zeta,\kappa)-\F'(\zeta,\kappa)\xi \|_X\notag\\
\le &~|\upsilon|\| \F_2(\zeta+\xi)-\F_2(\zeta)\|_X+\|\F(\zeta+\xi,\kappa)-\F(\zeta,\kappa)-\F'(\zeta,\kappa)\xi \|_X\notag\\
\le &~C|\upsilon|\|\xi\|_X+C(1+\kappa)\|\xi\|_X^{1+\alpha}\notag
\end{align}
if $\|\xi\|_X$ and $|\upsilon|$ are small enough, and $\zeta\in B_{\epsilon}(X)$. Hence $\F$ is Fr\'{e}chet differentiable with Fr\'{e}chet derivative given in \eqref{eq: joint Frechet}.   
The continuity of the mapping in \eqref{eq: joint Frechet} is easily obtained in a similar fashion 
using \Cref{prop: continuity of formal deriv}.
\end{proof}
This completes the proof of Theorem  \ref{Frechet theorem}.


\vskip 1in  
 \section{Vlasov model}\label{sec: 6}
 \subsection{Main theorem}
The Vlasov-Poisson system (VP) is 
\eqn \label{Vlasov} 
\pa_tf + v\cdot\nabla_x f  + \nabla_x U\cdot  \nabla_v f  = 0, \quad f\ge0, \eeqn 

\eqn \label{Poisson}
-\Delta U = 4\pi\rho, \quad \rho=\int_{\real^3} f\, dv \eeqn 
for $x,v\in\real^3$.  
Any function of the energy $\frac12v^2 -  U$ and the angular momentum $x \times v$ 
automatically satisfies the steady Vlasov equation \eqref{Vlasov} with $\pa_tf=0$.  
Therefore we specialize to the following form for the microscopic density $f(x,v)$.  
Let 
\eqn \label{microdensity}
f(x,v) = \frac M {D(\kappa,U)}\ 
\phi \left(\tfrac12 v^2-U(x), \kappa(x_1v_2-x_2v_1)\right),   \qquad x,v\in\real^3, \eeqn
where 
\eqn \label{D-scalar}
D(\kappa,U) = \iint_{\real^6} \phi\left(\tfrac12 v^2-U(x), \kappa(x_1v_2-x_2v_1)\right)\ dv\ dx,   \eeqn
where we denote $v^2=|v|^2$ for brevity.                
Division by the integral $D(\kappa,U)$ assures us that 
the mass $\iint_{\real^6} f~dvdx=M$ is a constant independent of the perturbation. 
The constant parameter $\kappa$ quantifies the smallness of the rotation.  
If $\ka=0$, there is no rotation.  
Such an $f$ automatically satisfies \eqref{Vlasov}.  
Of course, the potential $U$ must still be chosen to satisfy the Poisson equation \eqref{Poisson}.   

The following assumptions, which imply that the integrals are finite, are made on $\phi$.    
\eqn       \label{cond: phi 1}
 \phi(E,L)>0 \text{ for } E<0, \quad \phi(E,L)=0 \text { for } E > 0,  \eeqn  
 \eqn 
 \phi \in C^1((-\infty,0) \times (-\infty,\infty)) \text{ and }\pa_L^2\phi\in C((-\infty,0) \times (-\infty,\infty)), \eeqn
 \eqn  \label{cond: phi 2}
 \lim_{E\to 0^-}  (-E)^{\mu} \phi(E,L) = \lim_{E\to 0^-}  (-E)^{1/2} \pa_L^2\phi(E,L) = 0  \ \  
 				\text{ for some } 0<\mu<1/2,  \eeqn 
 \eqn  \label{cond: phi 3}
 \lim_{E\to -\infty}  (-E)^{1/2} \phi(E,L) = \infty,  \eeqn 
 \eqn   \label{cond: phi 4}
  \lim_{E\to -\infty}  (-E)^{-7/2} \phi(E,L) = 0,  \eeqn
 \eqn    \label{cond: phi 5}
         \pa_L\phi(E,0)=0 \text{ for } E<0.  \eeqn  
Each of the limits is assumed to be uniform for bounded $L$.  
We remark that the only place that $\mu\ne 1/2$ is needed is in the proof of \eqref{w3}    below.
As mentioned in the introduction, 
our prime example is $\phi(E,L) = (\max(-E,0))^{-\alpha}\ \psi(L)$, where $-\frac72 < \alpha < \frac12$   
and $\psi$ is any function such that $\psi'(0)=0$.

The existence of {\it radial} (spherically symmetric) solutions with $\ka=0$, as we now state, is well-known 
in the kinetic literature.   
By a radial solution we mean that $\ka=0$ and $U$ is a function of $|x|$ only.  

\begin{prop}			\label{Vlasov radial}
Let $\phi$ satisfy the assumptions above.   
Given $R>0$, there exists a solution $(f_0, U_0)$ 
of \eqref{Poisson}, \eqref{microdensity}  with $\ka=0$ that is 
radial,   for which 

\begin{itemize}
\item $U_0>0$ in $B_R$, $U_0=0$ on $\pa B_R$, $U_0<0$ in $\real^3\setminus B_R$,

\item $U_0'(|x|)<0$ for all $|x|>0$,   

\item $\rho_0>0$ in $B_R,\ \rho_0=0$ in $\real^3\backslash B_R$, 

\item $\rho_0 \in C^{1,\nu}(\real^3)$ and $U_0 \in C^{3,\nu}(\real^3)$ where  $\nu = \frac12-\mu$ 
and $\mu$ is given in \eqref{cond: phi 2}.
\end{itemize}
\end{prop}

We shall prove this proposition later in this section. 
The mass of the radial solution is defined as $M=\int_{\real^3} \rho_0~dx=\iint_{\real^6} f_0~dvdx = D(0,U_0)$ 
by \eqref{microdensity}.  
This is the constant $M$ that appears in \eqref{microdensity}.   
Our main theorem regarding {\it non-radial} solutions of the Vlasov model is as follows.  

\begin{theorem}   \label{main Vlasov thm}
There exists $\bar\kappa>0$ such that for all $|\kappa|<\bar\kappa$ there exists a solution 
$(f_\kappa,U_\kappa)$ of VP of the form \eqref{microdensity} 
with $f_\kappa$ axisymmetric and even in $x_3$, 
$U_\kappa \in C^3(\real^3)$,  
$$\iint_{\real^6} f_\kappa(x,v)\,dvdx = \int_{\real^3} \rho_\ka(x)\,dx =M $$ 
and the support of $\rho_\ka$ is a compact set close to $\bar B_1$.  
The mapping $\kappa \to U_\kappa$ is continuous from $(-\bar{\kappa},\bar{\kappa})$ to $C^3(\real ^3)$. 
\end{theorem}

\subsection{Construction} 
As is explained above, with the given ansatz \eqref{microdensity}, the unknown is $U$ and the equation to be solved is the Poisson equation, $-\Delta U =4\pi \rho$. A change of variables in \eqref{microdensity} leads to the formula 
\eqn
\rho(x) = \frac M{D(\kappa,U) } w(\kappa, r(x), U(x)),
\eeqn 
 where  
$$D(\kappa,U) =\int_{\real^3}w(\kappa,r,U)~dy, $$
and
\begin{align}   \label{w-equation}
w(\kappa,r,u)    &=   \int_{\real^3} \phi\left(\tfrac12v^2-u , \kappa(x_1v_2-x_2v_1)\right) ~dv \notag\\  
&=  2\pi\int_{-u}^{0}    \int_{-\sqrt{2(E+u)}} ^{\sqrt{2(E+u)}}  \phi(E,\kappa rs) ~dsdE.  
\end{align}
Here $r=r(x)=\sqrt{x_1^2+x_2^2}$.  Note that, due to \eqref{cond: phi 1}, $w(\ka,r,u)>0$ for $u>0$, while $w(\ka,r,u)=0$ for $u<0$. Thus $w$ is supported essentially where $u$ is positive.

Indeed, we justify formula \eqref{w-equation} as follows. 
Because  we may choose coordinates $v_1=\sigma\cos\theta, v_2=s, v_3=\sigma\sin\theta$  
and $x_2=0, x_1=r$,  the $v$-integral equals 
$$
w(\ka,r,u)  =  2\pi\int_\real  \int_0^\infty 
\phi(\tfrac12 s^2 + \tfrac12 \sigma^2 - u, \ka rs) ~\sigma d\sigma ds
 ,   $$
after which we introduce the variable $E=\frac12s^2+\frac12\sigma^2-u$ 
to obtain the double integral in  \eqref{w-equation}.  

With this notation the Poisson equation to be solved takes the form 
\eqn\label{eq: Poisson for U}
-\Delta U =  \frac{4\pi M}{D(\kappa,U) } w(\kappa,r,U) \qquad \text{ in } \real^3,
\eeqn
or equivalently,
\eqn\label{eq: Vlasov basic original}
-U(z) + \frac{M}{D(\kappa,U) }  \int_{\real^3} w\left(\ka,r(y),U(y)\right) \frac1{|z-y|}~ dy = constant
\eeqn
for $z\in \real^3$.

As with the Euler model, the solutions will be perturbations of the radial solution given in Proposition \ref{Vlasov radial}, as we shall now describe.  
As in Section 2 we can assume without loss of generality that $R=1$. 
We deform the domain exactly as in the Euler model, namely, by the homeomorphism
\eqn \label{scaling}
g_\zeta(x) = \left(1+\frac{\zeta(x)}{|x|^2}\right)x, 
\quad x\in \overline{B_1},   \eeqn 
for some axisymmetric functions $\zeta :\overline{B_1}\to \real$ in $X$ (and extend it when necessary to $\real^3$ by Lemma \ref{lem: extension}). The Banach space $X$ is the same as in the Euler case, namely 
\eqref{space X}, \eqref{norm X}. 
Our construction  differs from that of \cite{rein2000stationary}  in three respects.  
\begin{enumerate}[(i)]
\item the rescaling factor in \eqref{microdensity} keeps the total mass unchanged, 

\item the axisymmetric scaling has an $|x|^2$ factor on the denominator, and 

\item the form of $\phi$ is generalized.
\end{enumerate}
\noindent  
The difference (i) is our main improvement of the basic physical construction,    
(ii)  has certain technical advantages, and (iii) allows more general axisymmetric solutions.   


We look for solutions to \eqref{eq: Poisson for U} of the form $U = U_0(\gz^{-1}(z))$. 
An important observation is that, we essentially only need \eqref{eq: Vlasov basic original} to hold for $z\in \gz(B_1)$.
\begin{lemma}
Suppose $\zeta\in X$ with $\|\zeta\|_X$ small, and such that
\begin{align}\label{eq: perturb reduced}
&~U_0(\gz^{-1}(z)) \notag\\
= &~ U_0(0) + \frac{M}{D(\kappa, U_0\circ \gz^{-1})}  \int_{\real^3} w\left(\ka,r(y),U_0(\gz^{-1}(y))\right) \left[\frac1{|z-y|}-\frac{1}{|y|}\right]~ dy
\end{align}
for all $z\in \gz(\overline{B_1})$, then
\eqn
U(z) = U_0(0) + \frac{M}{D(\kappa, U_0\circ \gz^{-1})}  \int_{\real^3} w\left(\ka,r(y),U_0(\gz^{-1}(y))\right) \left[\frac1{|z-y|}-\frac{1}{|y|}\right]~ dy
\eeqn
solves \eqref{eq: Poisson for U} for all $z\in \real^3$. Here the function $w\left(\ka,r(y),U_0(\gz^{-1}(y))\right)$ is supported on $\gz(\overline{B_1})$.
\end{lemma}
\begin{proof}
We may think of $\gz$ as being extended to $\real^3$, so that $U_0\circ \gz^{-1}$ is a globally defined function. The definition of $U$ does not depend on how $\gz$ is extended outside $B_1$, because $w(\ka,r,U_0\circ \gz^{-1})$ is supported on $\gz(\overline{B_1})$. That in turn follows from Proposition \ref{Vlasov radial} and the property $w(\kappa, r, u)=0$ if $u\le 0$ mentioned above. By the regularity of $w$ established in Lemma \ref{lem: w reg} below, we have $U\in C^2(\real^3)$ with 
\eqn
-\Delta U = \frac{4\pi M}{D(\kappa,U_0\circ \gz^{-1})}w(\kappa, r, U_0\circ \gz^{-1}).
\eeqn
It remains to show that $w(\kappa, r, U_0\circ \gz^{-1})=w(\kappa, r, U)$. In fact, this is obviously true on $\gz(\overline{B_1})$, since $U=U_0\circ \gz^{-1}$ there by definition. To show that $w(\kappa, r, U_0\circ \gz^{-1})$ and $w(\kappa, r, U)$ also agree on $\real^3\setminus \gz(\overline{B_1})$, we just need $U<0$ on $\real^3\setminus \gz(\overline{B_1})$. That this is true follows immediately from the facts that $U$ is harmonic on $\real^3\setminus \gz(\overline{B_1})$, that $U=U_0\circ \gz^{-1}=0$ on $ \gz(\pa B_1)$, and that $U$ tends to a negative constant at infinity, by an application of the maximum principle.
\end{proof}

We have reduced the problem to solving \eqref{eq: perturb reduced}. Let us substitute $z=\gz(x)$.   
Define the operator 
\eqn \label{F-operator}
\F(\zeta,\kappa)(x)  = -U_0(x)+U_0(0)  + \frac M {D(\kappa,U_0\circ g_\zeta^{-1})} 
  \int_{B_2} w\left(\kappa,r(y), U_0(g_\zeta^{-1}(y))\right)  
  \left[ \frac1{|g_\zeta(x)-y|}  -  \frac1{|y|} \right]\ dy  
\eeqn
for $x$ in the unit ball $B_1$. 
Thus  our goal reduces to solving $\F(\zeta,\kappa)=0$ for the function $\zeta\in X$ as a function 
of the parameter $\kappa$ for small $\kappa$.   
We have chosen the constant terms conveniently so that $\F(0,0)=0$, which corresponds to the 
radial solution. 
\subsection{Radial solutions}

With $\kappa=0$ 
we have, from \eqref{Poisson} and \eqref{w-equation},  
\eqn \label{U_0 eqn}
-\Delta U_0  =  4\pi\rho_0  =  4\pi w(0,0,U_0) 
=  16\pi^2 \sqrt{2} \int_{-U_0(x)}^0   \phi(E,0)  \sqrt{U_0(x)+E}~dE. 
  \eeqn 
Defining 
\begin{equation}\label{def: G(s)}
 G(u)=w(0,0,u)=4\pi \sqrt{2} \int_{-u}^0  \phi(E,0) \ \sqrt{u+E}~dE,    \end{equation}
we can rewrite \eqref{U_0 eqn} as 
\eqn \label{U_0 eqn 1}
-\Delta U_0 = 4\pi G(U_0).    \eeqn
Due to \eqref{cond: phi 1}, $G(u)$ vanishes for $u<0$ and  is positive for $u>0$.   
					We can convert 
the conditions on $\phi$ given above to some conditions of $G$ as follows.
\begin{lemma}             \label{lem: G property}
If $\phi$ satisfies the conditions \eqref{cond: phi 1}-\eqref{cond: phi 4}, then 
$G\in C^{1,\nu}(\real)$, $G(u) = 0$ for $u\le 0$, $G(u)>0$ for $u>0$, and
\eqn \label{cond: G 1}\lim_{u\to 0^+} u^{-1}G(u)=0, ~\lim_{u\to \infty} u^{-1}G(u)=\infty,\eeqn 
\eqn  \label{cond: G 2}\lim_{u\to \infty}u^{-5}G(u)=0,\eeqn
\end{lemma}
				\begin{proof}
The regularity of $G$ follows from Lemma \ref{lem: w reg}, since $G(u)=w(0,0,u)$. 
We have 
\eqn 
G(u) = 4\pi\sqrt{2}\ u^{3/2} \int_0^1 \phi(-u\tau,0)\ \sqrt{1-\tau}\ d\tau. \eeqn  				
By \eqref{cond: phi 2}, $\forall\ep>0, \exists\delta>0$ such that for $u<\delta$ we have 
$$ 
 u^{-1}G(u) <  4\pi \sqrt2\ u^{1/2} \int_0^1 \ep (u\tau)^{-1/2} \ \sqrt{1-\tau}\ d\tau  
 = C\ep.  $$ 				
This proves the first limit in \eqref{cond: G 1}.  				
Similarly by \eqref{cond: phi 3}, $\forall\ep>0, \exists\delta>0$ such that for $u>1/\delta$ we have 
$$
 u^{-1}G(u) > C  u^{1/2} \int_{1/2}^1 \frac1\ep  (u\tau)^{-1/2} \ \sqrt{1-\tau}\ d\tau  = \frac C\ep,   $$
which proves the second limit in \eqref{cond: G 1}.  
Finally, by \eqref{cond: phi 2} and \eqref{cond: phi 4}, $\forall\ep>0, \exists\delta>0$ such that for $K>1/\delta$ we have  
\begin{align*}
u^{-5}\,G(u)  &\le    C u^{-7/2} \int_0^{K/u}\phi (-u\tau,0) \ \sqrt{1-\tau}\ d\tau +C u^{-7/2} \int_{K/u}^1\phi (-u\tau,0) \ \sqrt{1-\tau}\ d\tau  \\
 &\le Cu^{-7/2}\sup_{-K\le t \le 0}|t|^{1/2}\phi(t,0)\int_0^{K/u}(u\tau)^{-1/2}\sqrt{1-\tau}~d\tau+ Cu^{-7/2}\int_{K/u}^1\ep (u\tau)^{7/2}\sqrt{1-\tau}~d\tau.
\end{align*}
Now choose $u$ sufficiently large to make the above less than $C\ep$.
This proves \eqref{cond: G 2}.   
\end{proof}

As in Lemma \ref{lem: u_0}, conditions \eqref{cond: G 1} and \eqref{cond: G 2} 
are precisely what is needed for the existence of 
a positive radial solution $U_0$ to  \eqref{U_0 eqn 1}  with zero boundary condition defined on $B_1$  
and satisfying $U_0'(|x|)<0$ for $0<|x| \le1$.  
We extend $U_0(|x|)$ to be radial and harmonic in $\{|x|>1\}$ such that $U_0 \in C^1(\real^3)$.  
That is, $U_0(|x|) = (1-1/|x|) U_0'(1)$ for $|x|>1$.  
This completes the proof of Proposition \ref{Vlasov radial} since the regularity assertion follows easily.

\section{Vlasov model: analysis of the linearized operator}\label{sec: 7}
\subsection{Linearization}
\begin{theorem}  \label{deriv Vlasov}
The operator $\F: B_{\epsilon}(X)\times \real \to X$ with $\ep>0$ sufficiently small is continuously Fr\'echet differentiable with 
$\frac{\pa \F}{\pa \zeta}$ given by \eqref{eq: F' Rein} and \eqref{eq: M' Rein} below. 
\end{theorem} 
We postpone the proof of this theorem until Section \ref{sec: 8}. The operator $\F(\zeta,\kappa)$ is re-described in \eqref{eq: F Vlasov new}-\eqref{def: w}. We now compute $\L=\frac{\pa \F}{\pa \zeta}(0,0)$ using \eqref{eq: F' Rein} and \eqref{eq: M' Rein}. Noticing 
$$
\M(0,0)=D(0,U_0) = M, \quad w(0,0,u)=G(u),\quad \rho_0=G(U_0),
$$
which are merely definitions, we soon get
\begin{align} \label{derivF at bif} 
&[\L\xi](x) = 
\left[\frac{\pa\F}{\pa\zeta} (0,0)\xi \right](x)  \notag\\ 
  =  &~- \int_{B_1} \rho_0'(y)\frac{\xi(y)}{|y|} \left( \frac1{|x-y|}  -  \frac1{|y|} \right)~dy  
      - \left\{ \int_{B_1}\rho_0(|y|)\frac{x-y}{|x-y|^3}~dy \right\} \cdot \frac{\xi(x)x}{|x|^2}   \notag \\
      &~+\frac 1{M} \left\{  \int_{B_1} \rho_0'(y)\frac{\xi(y)}{|y|} ~dy\right\}\left\{\int_{B_1}  \rho_0(y) \left( \frac1{|x-y|}  -  \frac1{|y|} \right)~dy \right\}.
\end{align}
In the above, prime ($'$) means radial derivative $\partial_r$, where $r=|x|$. Now using again the calculation \eqref{eq: u_0'/x reduction} and the equivalent expression of $\F(0,0)=0$:
$$U_0(x)-U_0(0) = \int_{B_1}\rho_0(y)\left(\frac{1}{|x-y|}-\frac1{|y|}\right)~dy,$$
we get

\begin{align}   \label{kernel eqn0}
&[L\xi](x) = \left[\frac{\pa\F}{\pa\zeta} (0,0)\xi \right](x)  \notag\\ 
=  &~\frac{U_0'(x)}{|x|}\xi(x)- \int_{B_1} \rho_0'(y)\frac{\xi(y)}{|y|} \left( \frac1{|x-y|}  -  \frac1{|y|} \right)~dy  \notag\\
&~~+\frac 1{M} (U_0(x)-U_0(0))\int_{B_1} \rho_0'(y)\frac{\xi(y)}{|y|} ~dy.  
\end{align} 

Comparing with the linearized operator in the Euler model, we observe that the first two terms in \eqref{def: L} and \eqref{kernel eqn0} are the same. The difference in the last term results from the different ways the mass balancing factor appears in the Euler model and the Vlasov model. In the Euler model, the mass factor appears inside the enthalpy function $h$, whereas in the Vlasov model, it appears directly in front of the last term in the equation. The difference in the last term results in some crucial alteration in the analysis of the radial part of the kernel of $\L$. In particular, for the Vlasov model, the triviality of the kernel can be established in general and does not require a condition on the total mass as in the Euler model.
\subsection{Analysis of the Vlasov kernel} 
In this section we prove that  $\L$ is an isomorphism on $X$.  
\begin{theorem} \label{Vlasov injectivity}
The linearized operator $\L$  is injective. 
\end{theorem} 

By \eqref{kernel eqn0}, $\xi\in X$ belongs to the kernel of $\L$ if and only if 
the expression in \eqref{kernel eqn0} vanishes.   
Defining $\alpha(x) = \frac{U_0'(x)}{|x|}\xi(x)$ for convenience, 
we have   
\eqn\label{alpha eqn}
0 = \alpha(x) -  \int_{B_1}\frac{\rho_0'(y)}{U_0'(y)} \alpha (y) \left( \frac1{|x-y|}  -  \frac1{|y|} \right) ~dy
+\frac 1{M} (U_0(x)-U_0(0))\int_{B_1} \frac{\rho_0'(y)}{U_0'(y)}\xi(y) ~dy.  
\eeqn 
			 Our goal is  to prove that 
$\alpha\equiv0$.   
Note that $\alpha(0)=0$.    
Taking the Laplacian of both sides, we get 
\eqn \label{kernel eqn}
\Delta\alpha + 4\pi \frac{\rho_0'}{U_0'} \alpha  
-  \frac1M \left(\int_{B_1}\frac{\rho_0'}{U_0'} \alpha~dy\right) 4\pi\rho_0 =0. 
\eeqn
 
\begin{lemma} 
No radial function belongs to the nullspace of the linearized operator $\L$.
\end{lemma} 
\begin{proof}
This is the most delicate and novel part of the isomorphism proof. 
Let  $\alpha=\alpha(|x|)$ be a radial function satisfying \eqref{alpha eqn}.  
Note that $\alpha'(0)=0$, where the prime denotes the radial derivative.    
Integrating \eqref{kernel eqn} over $|x|<1$, 
the second and third terms exactly cancel each other and so we get 
$\alpha'(1)=0$.  
Summarizing, we have \eqref{kernel eqn}  together with 
\eqn\label{alpha vanishing}
\alpha(0)=\alpha'(0)=\alpha'(1)=0 . \eeqn
  We want to prove that $\alpha\equiv 0$.  
			
By \eqref{U_0 eqn} and \eqref{U_0 eqn 1}, we have
\eqn 
\frac{\rho_0'}{U_0'} =  \frac{(G(U_0))'}{U_0'}= G'(U_0)  \eeqn 
on $B_1$.
So we can rewrite \eqref{kernel eqn} as  
\eqn \label{alpha eqn3}
\Delta\alpha +  4\pi G'(U_0)\alpha - \frac{1}{M} \left(\int_{B_1}G'(U_0)\alpha~dy\right)4\pi G(U_0)=0.  \eeqn 
This is the basic equation for $\alpha$ from which, together with \eqref{alpha vanishing}, 
 we want to prove that it vanishes.  
 Note that, because of \eqref{alpha vanishing}, if $\int_{B_1}G'(U_0)\alpha ~dy=0$, then $\alpha=0$.   

On the contrary let us assume that $\int_{B_1}G'(U_0)\alpha ~dy\ne0$ and 
define $\beta = \frac{M}{\int_{B_1}G'(U_0)\alpha ~dy} \alpha$, so that 
\eqn \label{beta eqn}
\Delta\beta +  4\pi G'(U_0)\beta - 4\pi G(U_0) =0 \eeqn 
and 
\eqn\label{beta bdry value}
\beta(0)=\beta'(0)=\beta'(1)=0.  \eeqn 
To show that $\beta=0$, we consider the family of solutions to the following variation of \eqref{U_0 eqn 1}:
\eqn\label{radial v eqn}
v''+\frac2r v' + 4\pi SG(v)=0,\quad v(0)=a, v'(0)=0.
\eeqn
Here $S$ and $a$ are parameters, $r=|x|$. 
We denote the unique solution to \eqref{radial v eqn} by $v(r;a,S)$. 
By \eqref{U_0 eqn 1},  $v(r;U_0(0),1)=U_0(r)$. 
Differentiating \eqref{radial v eqn} with respect to $S$, and setting $S=1$, $a=U_0(0)$, we get
\eqn \label{vS eqn}
\Delta v_S + 4\pi G'(U_0)v_S + 4\pi G(U_0)=0,\quad v_S(0)=0, v_S'(0)=0,
\eeqn
where $v_S(r)=v_S(r;U_0(0),1)$, $v_a(r)=v_a(r;U_0(0),1)$.  
The subscripts in $v_S$ and $v_a$ denote partial derivatives.  
Since \eqref{beta eqn} and \eqref{vS eqn} are essentially linear ODEs with the same vanishing initial conditions and homogeneous terms and opposite nonhomogeneous terms, 
we must have 
\eqn  \label{constant multiple}
\beta = -v_S.
\eeqn
Note that a simple rescaling yields us $v(Rr;a,S) = v(r;a,R^2S)$. Differentiating with respect to $R$ and setting $R=S=1$, $a=u_0(0)$, we get
\eqn
rv' = 2v_S.
\eeqn
Taking the $r$ derivative and setting $r=1$, we get
\eqn\label{eq: last eqn}
2v_S'(1) = v'(1)+v''(1) = v'(1)-2v'(1)-G(v(1))=-v'(1).
\eeqn
Here we have used \eqref{radial v eqn}, $v(1)=U_0(1)=0$ and $G(0)=0$. The condition $\beta'(1)=0$ in \eqref{beta bdry value} and the relation \eqref{constant multiple} imply $v_S'(1)=0$, which implies $v'(1)=0$ by \eqref{eq: last eqn}. 
But the boundary condition is $v(1)=0$, which means that the initial data for $v(r)$ at $r=1$ vanish.  
Thus $v\equiv0$, which contradicts $v=U_0 >0$.    
\end{proof}

\begin{proof}[Proof of Theorem \ref{Vlasov injectivity}.]
It remains to consider the non-radial part of the kernel of $\L$. 
But the argument is identical to the Euler case, because the first two terms in \eqref{def: L} 
and \eqref{kernel eqn0} are the same. 
The only thing we used from the last term of \eqref{def: L} is that it is radial 
and hence is orthogonal to any non-radial spherical harmonic. 
This is still the case for the last term in \eqref{kernel eqn0}.
\end{proof}


 
 Now we prove the compactness.  
\begin{lemma}\label{lem: structure L 2}
$\L: X\to X$ has the form $\L=J+K$ where $J$ is an isomorphism, and $K$ is a compact operator.
\end{lemma}

\begin{proof} 
Recall that in the definition of  $\L$  given in \eqref{kernel eqn0}, only the last term differs from the Euler case. 
Since the last term is a rank one operator, it is compact.
\end{proof}

By the standard implicit function theorem, Theorem \ref{main Vlasov thm}   
follows by combining the preceding results, as in the Euler case. 

\section{Vlasov model: Fr\'{e}chet differentiability} \label{sec: 8}

In this section we prove the Fr\'{e}chet differentiability of the map
\begin{align}\label{eq: F Vlasov new}
&~\F(\zeta,\kappa)(x)\notag\\
=&~-U_0(x)+U_0(0)  
+\frac{M}{\M(\zeta,\kappa)}\int_{B_2}w\left(\kappa, r(y),U_0(g_{\zeta}^{-1}(y))\right)  
\left[\frac{1}{|g_{\zeta}(x)-y|}-\frac{1}{|y|}\right]~dy
\end{align}
where
\begin{equation}
\M(\zeta,\kappa)=\int_{B_2}w\left(\kappa, r(y),U_0(g_{\zeta}^{-1}(y))\right)~dy  
=  D(\kappa,U_0\circ g_\zeta^{-1})
\end{equation}
and
\eqn    \label{def: w}
w(\kappa,r,u) =     2\pi\int_{-u}^{0}    \int_{-\sqrt{2(E+u)}} ^{\sqrt{2(E+u)}} \phi(E,\kappa rs)~dsdE .   \eeqn

This is a different map from the one in the Euler model but the space $X$ will be  the same. 
We wish to prove Theorem \ref{deriv Vlasov} which states that the operator $\F$ is continuously Fr\'echet differentiable on $B_{\ep}(X)\times \real$.  
We first state a lemma that describes the regularity of $w(\kappa,r,u)$. 
\begin{lemma}\label{lem: w reg}
$w(\kappa,r,u)$ defined by \eqref{def: w} is in 
$C^1(\real\times [0,\infty)\times \real)$, $\pa_r w\in C^1(\real\times [0,\infty)\times \real)$, 
and for every bounded set $B\subset \real\times[0,\infty)\times \real$ there exists a  constant $C>0$ with
\eqn \label{w1}
|\pa_r w(\kappa,r,u)|\le Cr,
\eeqn
\eqn  \label{w2}
|w(\kappa_1,r,u_1)-w(\kappa_2,r,u_2)|\le C(|\kappa_1-\kappa_2|r+|u_1-u_2|),
\eeqn
\eqn  \label{w3}
|\pa_u w(\kappa_1,r,u_1)-\pa_u w(\kappa_2,r,u_2)|\le C(|\kappa_1-\kappa_2|+|u_1-u_2|^{\nu})
\eeqn
for all $(\kappa_1,r,u_1)$, $(\kappa_2,r,u_2)\in B$. Here $\nu = \frac12-\mu$, with $\mu$ given as in \eqref{cond: phi 2}.
\end{lemma}
					\begin{proof}
The proof generalizes that of Lemma 2.1 in \cite{rein2000stationary}.
By assumption \eqref{cond: phi 2} we have $0\le \phi(E,L) \le C|E|^{-1/2}$ for bounded $L$ and $E$.  Thus 
$$ 
|w(\ka,r,u)|  \le  C_1 \int_{-u}^0  \sqrt{\frac{E+u}{|E|}} ~dE =  C_1\int_0^u  \sqrt{\frac{u-s}{s}}~ds \le C_2.   $$
In a similar way, \eqref{w2} follows.  
Now 
$$
\pa_rw(\ka,r,u) = 2\pi\ka \int_{-u}^{0}    \int_{-\sqrt{2(E+u)}} ^{\sqrt{2(E+u)}}  \pa_L\phi(E,\kappa rs)s~dsdE .$$
Because $\pa_L\phi(E,0)=0$, we have 
$ |\partial_L\phi(E,\kappa rs)|\le |\pa_L^2\phi(E,\theta \kappa rs)|\kappa r s  $  
for some $\theta\in(0,1)$. Furthermore $|\pa_L^2\phi(E,L)| \le C|E|^{-1/2}$ for bounded $E$ and $L$ by \eqref{cond: phi 2}. Thus \eqref{w1} follows.  

Finally we have 
$$
\pa_uw(\ka,r,u)  =  \pi\sqrt2 \int_{-u}^0  
\left\{ \phi(E,\ka r\sqrt{2(E+u)}) + \phi(E,-\ka r\sqrt{2(E+u)}) \right\} \frac1{\sqrt{E+u}}~dE.     $$
					Without loss of generality we assume $u_1\ge u_2>0$.  
By \eqref{cond: phi 2} we have 
\begin{align} \label{eq: pa w three lines}
&~|\pa_u w(\kappa,r,u_1)-\pa_u w(\kappa,r,u_2)| \notag\\
 \le  & 
~C\left[\sup_{|E|,|L|\le B_0} |E|^{\mu}|\phi(E,L)|\right]  \int_{-u_1}^{-u_2}   \frac{|E|^{-\mu}}{\sqrt{E+u_1}}~dE  \\
& ~+  C   \left[\sup_{|E|,|L|\le B_0} |E|^{1/2}|\pa_L\phi(E,L)|\right] 
\int_{-u_2}^0 \frac1{\sqrt{|E|(E+u_1)}}  
\left[  \sqrt{E+u_1} - \sqrt{E+u_2}  \right] ~dE   \notag\\
& ~+ C  \left[\sup_{|E|,|L|\le B_0} |E|^{\mu}|\phi(E,L)|\right] \int_{-u_2}^0   |E|^{-\mu} 
\left [ \frac1{\sqrt{E+u_2}} - \frac1{\sqrt{E+u_1}} \right ]  ~dE  \notag\\
\le &~C|u_1-u_2|^{\frac12 - \mu} +C|u_1-u_2|^{\frac12 } + C|u_1-u_2|^{\frac12 - \mu}.\notag
\end{align}
					The last inequality is justified as follows.
In the first integral in \eqref{eq: pa w three lines} we change variables 
$F=-E$ and then $F=tu_2$ and $u_1=su_2$ with $1\le s\le2$ to obtain 
\begin{align*}\label{eq: ratio sub}
&  \int_{u_2}^{u_1}\frac{1}{F^{\mu}(u_1-F)^{\frac12}}~dF
=  u_2^{\frac12-\mu}\int_1^s\frac{1}{t^{\mu}(s-t)^{\frac12}}~dt  
\le   u_2^{\frac12-\mu}\int_1^s\frac{1}{(s-t)^{\frac12}}~dt    \notag\\
& =   2u_2^{\frac12-\mu}(s-1)^{\frac12}  
\le 2[u_2(s-1)]^{\frac12-\mu} 
= 2(u_1-u_2)^{\frac12-\mu}.
\end{align*}
On the other hand, if $s>2$, the same integral is bounded by
\eqn 
 u_2^{\frac12-\mu}  \int_{0}^s\frac{1}{t^{\mu}(s-t)^{\frac12}}~dt   
= Cu_2^{\frac12-\mu}s^{\frac12-\mu}\le C(u_2 2(s-1))^{\frac12-\mu} = C[2(u_1-u_2)]^{\frac12-\mu}.
\eeqn 
The other terms in \eqref{eq: pa w three lines} can be estimated in a similar fashion, 
which leads to \eqref{w3}. 
\end{proof}

To show that $\F(\zeta,\kappa)\in X$, we estimate its spatial derivatives. As for the Euler case, we can differentiate \eqref{eq: F Vlasov new} under the integral sign. Consequently, there exists $\epsilon>0$, such that for all $\zeta\in B_{\epsilon}(X)$,
\eqn \label{eq: partial_i F Rein}
\partial_i\F(\zeta,\kappa)(x) 
= -\partial_i U_0(x) 
+\frac{M}{\M(\zeta, \kappa)} 
\int_{B_2}w\left(\kappa, r(y),U_0(g_{\zeta}^{-1}(y))\right) 
\frac{-(g_{\zeta}(x)-y)}{|g_{\zeta}(x)-y|^3}~dy\cdot \partial_i \gz(x).
\eeqn

Next we estimate the second term in \eqref{eq: partial_i F Rein}.  
\begin{align}
~|w\left(\kappa, r(y),U_0(g_{\zeta}^{-1}(y))\right) - &w\left(\kappa, 0,U_0(g_{\zeta}^{-1}(0))\right)| 
\le ~C(r(y)+|U_0(g_{\zeta}^{-1}(y))-U_0(g_{\zeta}^{-1}(0))|) \notag\\
\le &~C(r(y)+\|\nabla U_0\|_{\infty}|\gz^{-1}(y)|)
\le ~C(1+\|\zeta\|_X)|y|.
\end{align}
Therefore 
if $|\kappa|\le \bar{\kappa}$, and $\zeta\in B_{\epsilon}(X)$ for some $\epsilon$ small enough, there is a constant $C>0$ such that
\begin{equation}
\|\F(\zeta,\kappa)\|_X\le C(1+\|\zeta\|_X).
\end{equation}

Next we compute the formal derivative. As in the Euler model, we define 
$F(x,s) = \F(\zeta+s\xi, \kappa)(x) $ 
and we define the formal derivative as
\begin{equation}
[\F'(\zeta,\kappa)\xi](x) = \partial_s F(x,0) = \partial_s\bigg|_{s=0}\F(\zeta+s\xi,\kappa)(x).
\end{equation}

						\begin{lemma}
The formal derivative $[\F'(\zeta,\kappa)\xi](x)$ exists and equals 
\begin{align}\label{eq: F' Rein}  
&~  -\frac{M}{\M(\zeta,\kappa)}\int_{B_2}\partial_u w\left(\kappa,r(y),U_0(\gz^{-1}(y))\right)\notag\\
&\qquad \qquad  \nabla U_0(\gz^{-1}(y))D(\gz^{-1})(y)\ \xi(\gz^{-1}(y))\frac{\gz^{-1}(y)}{|\gz^{-1}(y)|^2}\left[\frac{1}{|\gz(x)-y|}-\frac{1}{|y|}\right]~dy\notag\\
&~-\frac{M}{\M(\zeta,\kappa)}\int_{B_2}w\left(\kappa,r(y),U_0(\gz^{-1}(y))\right)\frac{\gz(x)-y}{|\gz(x)-y|^3}~dy\cdot \frac{x\xi(x)}{|x|^2}\notag\\
&~-\frac{M[\M'(\zeta,\kappa)\xi]}{\M^2(\kappa,\zeta)}\int_{B_2}w\left(\kappa, r(y),U_0(g_{\zeta}^{-1}(y))\right)\left[\frac{1}{|g_{\zeta}(x)-y|}-\frac{1}{|y|}\right]~dy,
\end{align}
						where
\begin{equation}\label{eq: M' Rein}
\M'(\zeta,\kappa)\xi = -\int_{B_2}\partial_u w\left(\kappa,r(y),U_0(\gz^{-1}(y))\right)\nabla U_0(\gz^{-1}(y))D\gz^{-1}(y)\xi(\gz^{-1}(y))\frac{\gz^{-1}(y)}{|\gz^{-1}(y)|^2}~dy
\end{equation}
\end{lemma}

\begin{proof}
We use the cutoff function method as before. As is the case for estimates \eqref{est: 1/r^2} and \eqref{est: 1/r}, the key estimates 
\begin{equation}
\left|w\left(\kappa,r(y),U_0(g_{\zeta+s\xi}^{-1}(y))\right)\frac{1}{|g_{\zeta+s\xi}(x)-y|}\right|\le \frac{C}{|g_{\zeta+s\xi}(x)-y|},
\end{equation}
\begin{equation}
\left|\partial_s \left[w\left(\kappa,r(y),U_0(g_{\zeta+s\xi}^{-1}(y))\right)\frac{1}{|g_{\zeta+s\xi}(x)-y|}\right]\right|\le \frac{C}{|g_{\zeta+s\xi}(x)-y|^2} 
\end{equation}
are easily proven. 
\end{proof}

\begin{lemma}\label{lem: partial_i F' Rein} 
The spatial derivatives of $\F'(\zeta,\kappa)\xi$ are 
\begin{align}\label{eq: partial_i F' Rein}
&~\partial_i[\F'(\zeta,\kappa)\xi](x)\notag\\
= &~  \frac{M}{\M(\zeta,\kappa)}\int_{B_2}\partial_u w\left(\kappa,r(y),U_0(\gz^{-1}(y))\right)\notag\\
&\qquad  \nabla U_0(\gz^{-1}(y))D\gz^{-1}(y)\xi(\gz^{-1}(y))\frac{\gz^{-1}(y)}{|\gz^{-1}(y)|^2}\frac{\gz(x)-y}{|\gz(x)-y|^3}~dy\cdot \partial_i \gz(x)\notag\\
&+\sum_j \frac{M}{\M(\zeta,\kappa)}\int_{B_2}\bigg[w_r\left(\kappa,r(y),U_0(\gz^{-1}(y))\right)\partial_{j}r(y) \notag\\
&+ w_u\left(\kappa,r(y),U_0(\gz^{-1}(y))\right)\cdot\nabla U_0(\gz^{-1}(y))\partial_j\gz^{-1}(y)\bigg]\frac{\gz(x)-y}{|\gz(x)-y|^3}~dy\cdot \partial_i \gz(x)\frac{x_j\xi(x)}{|x|^2}\notag\\
&-\frac{M}{\M(\zeta,\kappa)}\int_{B_2}w\left(\kappa,r(y),U_0(\gz^{-1}(y))\right)\frac{\gz(x)-y}{|\gz(x)-y|^3}~dy\cdot \partial_i\left(\frac{x\xi(x)}{|x|^2}\right)\notag\\
&~+\frac{M[\M'(\zeta,\kappa)\xi]}{\M^2(\zeta,\kappa)}\int_{B_2}w\left(\kappa, r(y),U_0(g_{\zeta}^{-1}(y))\right)\frac{g_{\zeta}(x)-y}{|g_{\zeta}(x)-y|^3}~dy\cdot \partial_i \gz(x).
\end{align}
\end{lemma}
\begin{proof}
We apply the cutoff function method to \eqref{eq: F' Rein}. 
The second term of \eqref{eq: F' Rein} is written as
\begin{align}
\sum_j \int_{B_2}w\left(\kappa,r(y),U_0(\gz^{-1}(y))\right)\partial_{y_j}\left(\frac{1}{|\gz(x)-y|}\right)~dy ~\frac{x_j\xi(x)}{|x|^2}
\end{align}
and then integrated by parts to get
\begin{equation}
-\sum_j \int_{B_2}\partial_{y_j}\left(w\left(\kappa,r(y),U_0(\gz^{-1}(y))\right)\right)\frac{1}{|\gz(x)-y|}~dy~\frac{x_j\xi(x)}{|x|^2}.
\end{equation}
Details are omitted. 
\end{proof}
\begin{lemma}
The mixed partials of $F$ are 
\begin{equation}\label{eq: mixed partial Rein}
\partial_s\partial_i F(x,s)=\partial_i\partial_sF(x,s) = \partial_i[\F'(\zeta+s\xi,\kappa)](x).
\end{equation}
\end{lemma}
\begin{proof}
The second equality in \eqref{eq: mixed partial Rein} is the content of \Cref{lem: partial_i F' Rein}. 
To get the first equality, we note that by \eqref{eq: partial_i F Rein}, we have 
\begin{align}\label{eq: partial_i F xs Rein}
\partial_i F(x,s) = &~   -\partial_i U_0(x) \notag\\
&+\frac{M}{\M(\zeta+s\xi,\kappa)}\int_{B_2}w\left(\kappa, r(y),U_0(g_{\zeta+s\xi}^{-1}(y))\right)\frac{-(g_{\zeta+s\xi}(x)-y)}{|g_{\zeta+s\xi}(x)-y|^3}~dy\cdot \partial_i g_{\zeta+s\xi}(x).
\end{align}
To calculate the $s$ derivative of the integral in \eqref{eq: partial_i F xs Rein}, we apply the cutoff function to get
\begin{align}
\int_{B_2}w\left(\kappa, r(y),U_0(g_{\zeta+s\xi}^{-1}(y))\right)\frac{-(g_{\zeta+s\xi}(x)-y)}{|g_{\zeta+s\xi}(x)-y|^3}\left[1-\chi\left(\frac{|g_{\zeta+s\xi}(x)-y|}{\epsilon}\right)\right]~dy
\end{align}
and compute the derivative as we did for \eqref{eq: def I(s)}.
\end{proof}
\begin{lemma}
If $|\kappa|\le \bar{\kappa}$, and $\zeta\in B_{\epsilon}(X)$ for some $\epsilon$ small enough, 
there is a constant $C>0$ such that
\begin{equation}
\|\F'(\zeta,\kappa)\xi\|_X\le C\|\xi\|_X.
\end{equation}
\end{lemma}
					\begin{proof}
We apply \Cref{lem: Lip estimate at zero} to \eqref{eq: partial_i F' Rein}. 
To work out the first term, we just need 
\begin{equation}\label{est: I1 Rein}
\partial_u w\left(\kappa,r(y),U_0(\gz^{-1}(y))\right)\nabla U_0(\gz^{-1}(y))D(\gz^{-1})(y)\xi(\gz^{-1}(y))  
\frac{\gz^{-1}(y)}{|\gz^{-1}(y)|^2}
\end{equation}
to be bounded by $C\|\xi\|_X|y|$, 
which is similar to our treatment of the first term of \eqref{eq: d_i of F_1'}. 
This bound is indeed achieved by the properties of $w$.  
Since the second term in  \eqref{eq: partial_i F' Rein}    has a factor $\frac{x_j\xi(x)}{|x|^2}$, 
which is already bounded by $C\|\xi\|_X|x|$, we just need 
\begin{equation}
w_r\left(\kappa,r(y),U_0(\gz^{-1}(y))\right)\partial_{j}r(y)+w_u\left(\kappa,r(y),U_0(\gz^{-1}(y))\right)\cdot\nabla U_0(\gz^{-1}(y))\partial_j\gz^{-1}(y)
\end{equation}
to be bounded, which is again true.   
For the third term in \eqref{eq: partial_i F' Rein}, we need 
\begin{equation}
\left|w\left(\kappa,r(y),U_0(\gz^{-1}(y))\right)-w\left(\kappa,0,U_0(0)\right)\right|\le C|y|,
\end{equation}
which is also true since $w$ is $C^1$.  The last term is similar.
\end{proof}

\begin{lemma}
If $|\kappa|\le \bar{\kappa}$, and $\zeta\in B_{\epsilon}(X)$ for some $\epsilon$ small enough, there is a constant $C>0$ such that
\begin{equation}
\|(\F'(\zeta_1,\kappa)-\F'(\zeta_2,\kappa))\xi\|_X\le C\|\zeta_1-\zeta_2\|_X^{\nu}\|\xi\|_X.
\end{equation}
\end{lemma}
\begin{proof}
We compute the spatial derivatives of $(\F'(\zeta_1,\kappa)-\F'(\zeta_2,\kappa))\xi$ using \eqref{eq: partial_i F' Rein} and we estimate as in \Cref{prop: continuity of formal deriv}. We first note that  
$|\M(\zeta_1,\kappa)-\M(\zeta_2,\kappa|\le C\|\zeta_1-\zeta_2\|_X $
because
\begin{equation}
\left|w\left(\kappa,r(y),U_0(\gzone^{-1}(y))\right)-w\left(\kappa,r(y),U_0(\gztwo^{-1}(y))\right)\right|\le C\|\zeta_1-\zeta_2\|_X.
\end{equation}
For the first term in \eqref{eq: partial_i F' Rein}, the key estimate is
\begin{equation}\label{est: key est 1}
\left|\partial_u w\left(\kappa,r(y),U_0(\gzone^{-1}(y))\right)-\partial_u w\left(\kappa,r(y),U_0(\gztwo^{-1}(y))\right)\right|\le C\|\zeta_1-\zeta_2\|_X^{\nu}. 
\end{equation}
 For the second term in \eqref{eq: partial_i F' Rein}, the key estimate is 
\begin{equation}
\left|\partial_r w\left(\kappa,r(y),U_0(\gzone^{-1}(y))\right)-\partial_r w\left(\kappa,r(y),U_0(\gztwo^{-1}(y))\right)\right|\le C\|\zeta_1-\zeta_2\|_X^{\nu}, 
\end{equation}
which  is a consequence of the fact that 
\begin{equation}\label{est: partial r w}
|\partial_r w(\kappa,r,u)-\partial_r w(\kappa,r,u')|\le Cr|u-u'|^{1/2}
\end{equation}
locally uniformly. 
For the third term in \eqref{eq: partial_i F' Rein}, the key estimate is
\begin{equation}
\left|w\left(\kappa,r(y),U_0(\gzone^{-1}(y))\right) - w\left(\kappa,r(y),U_0(\gztwo^{-1}(y))\right)\right|\le C\|\zeta_1-\zeta_2\|_X|y|.
\end{equation}
For the fourth term, we need to estimate $(\M'(\zeta_1,\kappa)-\M'(\zeta_2,\kappa))\xi$ using \eqref{eq: M' Rein}. The key estimate again is \eqref{est: key est 1}.
\end{proof}

\begin{lemma}
If $|\kappa|\le \bar{\kappa}$, and $\zeta\in B_{\epsilon}(X)$ for some $\epsilon$ small enough. Let $\xi$ be such that $\zeta+s\xi\in B_{\epsilon}(X)$ for all $s\in [-1,1]$. Then
\begin{equation}
\|\F(\zeta+\xi,\kappa)-\F(\zeta,\kappa)-\F'(\zeta,\kappa)\xi\|_X\le C\|\xi\|_X^{1+\nu}.
\end{equation}
\end{lemma}
\begin{proof}
The proof is basically identical to that of \Cref{prop: Frechet}.
\end{proof}
This completes the proof of Theorem \ref{deriv Vlasov}.

\bibliographystyle{acm}
\bibliography{rotstarbiblio} 
\end{document}